\documentclass[11pt]{amsart}

\usepackage{amsmath,amssymb,parskip}
\usepackage{hyperref}
\usepackage{todonotes}

\usepackage{bbm}
\usepackage{graphicx}
\usepackage[all]{xy}
\usepackage{enumerate}
\usepackage{float}
\usepackage{psfrag}
\usepackage{a4wide}
\usepackage{upgreek}

\newtheorem{theo}{Theorem}[section]

\newtheorem{defi}[theo]{Definition}
\newtheorem{prop}[theo]{Proposition}
\newtheorem{lemm}[theo]{Lemma}
\newtheorem{coro}[theo]{Corollary}
\newtheorem{nota}[theo]{Notation}

\newtheorem{exam}[theo]{Example}

\newtheorem{rem}[theo]{Remark}

\newenvironment{rema}{\begin{rem}\rm }{\hfill $\blacktriangleleft$\end{rem}}

\def\del{\partial} 
\def\bb#1{\mathbb{#1}} \def\bbcp{\mathbb{C}\mathbb{P}} \def\m#1{\mathcal{#1}} \def\mfr#1{\mathfrak{#1}}
\def\id{\mathbbm{1}}

 \def\co{\colon\thinspace}
\def\momeg{(M,\omega)}

\def\ham#1{\mathrm{Ham}#1}

\def\ee{\mathrm{e}}
\def\max{\mathrm{max}}
\def\akl{{kA_n+lA_1}}
\def\GW{\mathrm{GW}}
\def\PD{\mathrm{PD}}
\def\QH{{QH}}
\def\Q{\mathbb Q}
\def\N{\mathbb N}
\def\R{\mathbb R}
\def\C{\mathbb C}
\def\Z{\mathbb Z}
\def\F{\mathbb F}
\def\X{\mathbb X}
\def \bbcp{\mathbb C\mathbb P}
\newcommand{\pt}{[\mathrm{pt}]}

\newcommand{\cM}{\mathcal{M}}
\newcommand{\cO}{\mathcal{O}}
\newcommand{\Mbar}{\overline{\cM}}
\newcommand{\eva}{\mathrm{ev}}
\newcommand{\vir}{{\mathrm{vir}} }
\newcommand{\Ext}{\mathrm{Ext}}
\newcommand{\Lin}{\mathrm{Lin}}
\newcommand{\Vertices}{\mathrm{Vert}}
\newcommand{\Edge}{\mathrm{Edge}}
\newcommand{\val}{\mathrm{val}}
\newcommand{\vb}{\mathfrak{b}}
\newcommand{\PbP}{ \bbcp^2\#\,\overline{ \bbcp}\,\!^2}
\newcommand{\conezero}{\mathcal{\#}\{c_1=0\}}

\newcommand*{\quot}[2]%
{\ensuremath{%
   \raisebox{.35ex}{\ensuremath{#1}}\big/\raisebox{-.35ex}{\ensuremath{#2}}}}

\begin{document}

\title{{S}eidel's morphism of toric 4--manifolds}
\author{S\'ilvia Anjos and R\'emi Leclercq}
\date{\today}
% \date{June 30, 2014}
\thanks{$\,\!^\star${see \textit{Alice's adventures in Wonderland}, by Lewis Carroll (1865).}}
\address{S\'ilvia Anjos: Center for Mathematical Analysis, Geometry and Dynamical Systems, Mathematics Department,
Instituto Superior T\'ecnico, Av. Rovisco Pais, 1049-001 Lisboa, Portugal}
\email{sanjos@math.ist.utl.pt}
\address{R\'emi Leclercq: Universit\'e Paris-Sud, D\'epartement de Math\'ematiques, Bat. 425, 91400 Orsay, France}
\email{remi.leclercq@math.u-psud.fr}


\begin{abstract}
Following McDuff and Tolman's work on toric manifolds \cite{McDuffTolman06}, we focus on 4--dimensional NEF toric manifolds and we show that even though Seidel's elements consist of infinitely many contributions, they can be expressed by closed formulas. From these formulas, we then deduce the expression of the quantum homology ring of these manifolds as well as their Landau--Ginzburg superpotential. We also give explicit formulas for the Seidel elements in some non-NEF cases. 
These results are closely related to recent work by Fukaya, Oh, Ohta, and Ono \cite{FOOO11}, Gonz\'alez and Iritani \cite{GonzalesIritani11}, and Chan, Lau, Leung, and Tseng \cite{CLLT}. The main difference is that in the 4--dimensional case the methods we use are more elementary: they do not rely on open Gromov--Witten invariants nor mirror maps. We only use the definition of Seidel's elements and specific closed Gromov--Witten invariants which we compute via localization. So, unlike Alice$\,\!^\star\!$, the computations contained in this paper are not particularly pretty but they do stay on their side of the mirror. This makes the resulting formulas directly readable from the moment polytope. 
\end{abstract}

\maketitle

\section{Introduction}\label{sec:introduction}

Let $(M,\omega)$ be a closed connected symplectic manifold and let as usual $\ham{(M,\omega)}$ denote its Hamiltonian diffeomorphism group. Under a suitable condition of semipositivity, Seidel defined in \cite{Seidel97} a morphism, $\m S$,  from $\pi_1(\ham{(M,\omega)})$ to -- after a mild generalization due to Lalonde, McDuff, and Polterovich \cite{LMcDP99} -- $QH_*(M,\omega)^\times$, the group of invertible elements of the quantum homology of $(M,\omega)$. This morphism has been extensively used in order to get information on the topology of Hamiltonian diffeomorphism groups as well as the quantum homology of symplectic manifolds.
 It has also been extended in various directions, see the end of the introduction for some of these extensions related to the present work. 

A quantum class lying in the image of $\m S$ is called a \textit{Seidel element}. In \cite{McDuffTolman06}, McDuff and Tolman were able to specify the structure of the lower order terms of Seidel's elements associated to Hamiltonian circle actions whose maximal fixed point component, $F_\max$, is semifree. Recall that this condition means that the action is semifree on a neighborhood of $F_\max$ which means, in our case, that the stabilizer of each point is trivial or the whole circle. When the codimension of $F_\max$ is 2, their result immediately ensures that if there exists an almost complex structure $J$ on $M$ so that $(M,J)$ is Fano, i.e so that there are no $J$--pseudo-holomorphic spheres in $M$ with non-positive first Chern number, all the lower order terms vanish. In the presence of $J$--pseudo-holomorphic spheres with vanishing first Chern number, there is a priori no reason why arbitrarily large multiple coverings of such objects should not contribute to the Seidel elements. As a matter of fact, McDuff and Tolman exhibited an example of such a phenomenon when $(M,J)$ is a NEF pair, which by definition do not admit $J$--pseudo-holomorphic spheres with negative first Chern number.

In this paper, we show that even though there are indeed infinitely many contributions to the Seidel elements associated to the Hamiltonian circle actions of a NEF 4--dimensional toric manifold, these quantum classes can still be expressed by explicit closed formulas. Moreover, these formulas only depend on the relative position of representatives of elements of $\pi_2(M)$ with vanishing first Chern number as facets of the moment polytope. In particular, they are directly readable from the polytope.

More precisely, we consider (see Section \ref{sec:toric-manif-quant} for precise definitions):
\begin{itemize}
\item a 4--dimensional closed symplectic manifold $(M,\omega)$, endowed with a toric structure and admitting a NEF almost complex structure,
\item its corresponding Delzant polytope $P$, which is assumed to have $n \geq 4$ facets,
\item a Hamiltonian action generated by a circle subgroup $\Lambda$, with moment map $\Phi_\Lambda$.
\end{itemize}
We assume additionally, that the fixed point component of $\Lambda$ on which $\Phi_\Lambda$ is maximal is a 2--sphere, $F_\max \subset M$, whose momentum image is a facet of $P$, $D$. We denote by $A \in H_2(M;\bb Z)$ the homology class of $F_\max$ and by $\Phi_\max = \Phi_\Lambda(F_\max)$.

In this case, McDuff--Tolman's result ensures that the Seidel element associated to $\Lambda$ is 
\begin{align*}
  S(\Lambda) = A\otimes q t^{\Phi_\max} + \sum_{B \in H_2^S\!(M;\bb Z)^{>0}} a_B \otimes q^{1-c_1(B)} t^{\Phi_\max-\omega(B)} 
\end{align*}
where $H_2^S(M;\bb Z)^{>0}$ consists of the spherical classes of symplectic area $\omega(B)>0$ and $a_B\in H_*(M;\bb Z)$ denotes the contribution of $B$. As mentioned above, when there exists a Fano almost complex structure, all the lower order terms vanish and we end up with $ S(\Lambda)=A \otimes qt^{\Phi_\max}$. 

In the non-Fano case, one has to be careful about the number and relative position of facets, in the vicinity of $D$, corresponding to spheres in $M$ with vanishing first Chern number. We denote the number of such facets by $\conezero$. Theorem \ref{theo:contributions} lists all the contributions made to the Seidel element associated to $\Lambda$ in the 6 cases when $\conezero \leq 2$. We denote the facets and the corresponding homology classes in $M$ in a cyclic way, that is, $D$, which we denote by $D_n$ below, has neighbooring facets $D_{n-1}$ on one side and $D_{n+1}=D_1$ on the other, and they respectively induce classes $A_n$, $A_{n-1}$, and $A_{n+1}=A_1$ in $H_2(M;\bb Z)$.

 Figure \ref{fig:cases} shows the relevant parts of the different polytopes we need to consider. Dotted lines represent facets with non-zero first Chern number and we indicate near each facet with non-trivial contribution the homology class of the corresponding sphere in $M$. For example, in Case \textit{(3c)}, only three homology classes contribute: $A_{n-1}$, $A_n$, and $A_1$; $A_{n-1}$ and $A_1$ have vanishing first Chern number while $c_1(A_n) \neq 0$.
\begin{figure}[htbp]
  \centering
  \includegraphics{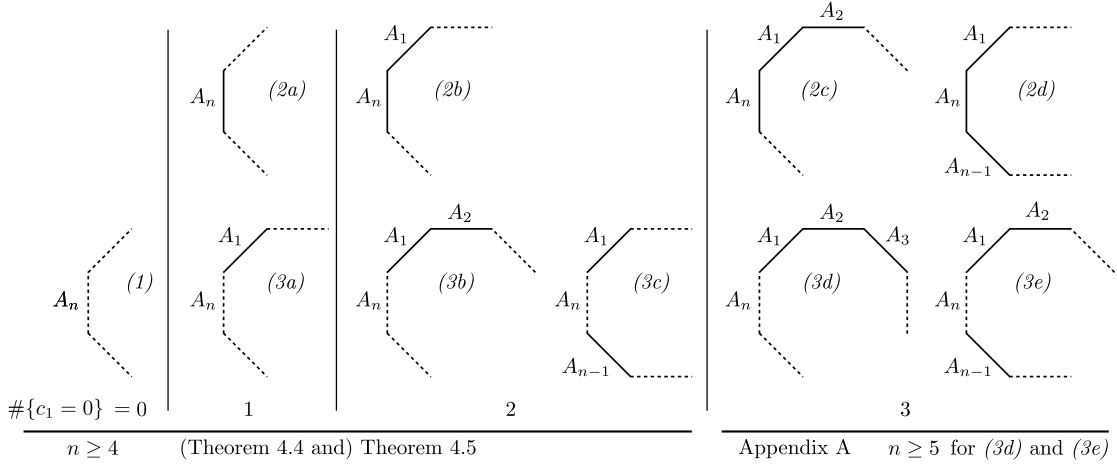}
  \caption{The cases appearing in Theorems \ref{theo:contributions} and \ref{theo:SeidelsMorphism}, and Appendix \ref{sec:BrutalComputationsSeidSMorph}}
  \label{fig:cases}
\end{figure}

% \begin{theo*}\label{theo:contributions}
\noindent \textbf{Theorem \ref{theo:contributions}.}$\,$ \textit{ 
With the notation and under the assumptions above, the following homology classes have non trivial contributions to $S(\Lambda)$:
  \begin{enumerate}
  \item $A_n$ contributes by $a_{A_n}=A_n$.
  \item When $c_1(A_n)=0$, 
    \begin{enumerate}[(2a)] 
    \item then $kA_n$ (with $k>0$) contributes by $a_{kA_n}=A_n$,
   \item and when $c_1(A_1)=0$, then $kA_n+lA_1$ (with $k\geq 0$ and $l>0$) contributes and its contribution is $a_{kA_n+lA_1}= \left\{ \begin{array}{l} A_n \mbox{ if } k\geq l \\ -A_1 \mbox{ otherwise.}  \end{array} \right.$
    \end{enumerate}
  \item When $c_1(A_n)\neq 0$,
    \begin{enumerate}[(3a)]
    \item when $c_1(A_1)=0$, then $kA_1$ (with $k>0$) contributes by $a_{kA_1}=-A_1$,
    \item when $c_1(A_1)=0$ and $c_1(A_2)=0$, then $kA_1+lA_2$ (with $k>0$ and $l>0$) also contributes, and its contribution is $a_{kA_1+lA_2}= \left\{ \begin{array}{l} -A_1 \mbox{ if } k\geq l \\ A_2 \mbox{ otherwise,}  \end{array} \right.$
    \item when $c_1(A_{n-1})=0$ and $c_1(A_1)=0$, then $kA_{n-1}$ and $lA_1$ (with $k>0$ and $l>0$) also contribute, and their contributions are $a_{kA_{n-1}}= -A_{n-1}$ and $a_{lA_1}= -A_1$.
    \end{enumerate}
  \end{enumerate}
Moreover, in each case, if the facets immediately next to the ones mentioned correspond to spheres with non-zero first Chern number, then these are the only non-trivial contributions.
}
% \end{theo*}

Now, under the same assumptions, Theorem \ref{theo:SeidelsMorphism} gives the explicit expression of the Seidel element associated to $\Lambda$ when $\conezero \leq 2$. Notice that we give (without proofs) the expression of the Seidel elements for $\conezero = 3$ in Appendix \ref{sec:BrutalComputationsSeidSMorph}.

% \begin{theo*}\label{theo:SeidelsMorphism}
\noindent \textbf{Theorem \ref{theo:SeidelsMorphism}.} 
\textit{
Under the assumptions above, and in the cases described by Figure \ref{fig:cases}, the Seidel element associated to $\Lambda$ is
\begin{align*}
& \hspace{-6.9cm} \mathit{(1)} \quad S(\Lambda)=A_n\otimes qt^{\Phi_\max} 
\end{align*}
\begin{align*}
&  \hspace{-1cm} \mathit{(2a)} \quad S(\Lambda)=A_n\otimes q\,\frac{t^{\Phi_\max}}{1-t^{-\omega(A_n)}} \\
&   \hspace{-1cm}\mathit{(2b)} \quad S(\Lambda)=\left[ A_n\otimes q\,\frac{t^{\Phi_\max}}{1-t^{-\omega(A_n)}}-A_1\otimes q\,\frac{t^{\Phi_\max-\omega(A_1)}}{1-t^{-\omega(A_1)}} \right]\cdot \frac{1}{1-t^{-\omega(A_n)-\omega(A_1)}} 
\end{align*}
\begin{align*}
& \mathit{(3a)} \quad S(\Lambda)=A_n\otimes qt^{\Phi_\max}-A_1\otimes q\,\frac{t^{\Phi_\max-\omega(A_1)}}{1-t^{-\omega(A_1)}} \\
& \mathit{(3b)} \quad S(\Lambda)  =  A_n\otimes qt^{\Phi_\max}-A_1\otimes q\,\frac{t^{\Phi_\max-\omega(A_1)}}{1-t^{-\omega(A_1)}}  \\
& \hspace{3cm} -\left(A_1\otimes q\,\frac{t^{\Phi_\max}}{1-t^{-\omega(A_1)}}-A_2\otimes q\,\frac{t^{\Phi_\max-\omega(A_2)}}{1-t^{-\omega(A_2)}}\right)\cdot \frac{t^{-\omega(A_1)-\omega(A_2)}}{1-t^{-\omega(A_1)-\omega(A_2)}} \\
& \mathit{(3c)} \quad S(\Lambda)=A_n\otimes qt^{\Phi_\max}-A_{n-1}\otimes q\,\frac{t^{\Phi_\max-\omega(A_{n-1})}}{1-t^{-\omega(A_{n-1})}}-A_1\otimes q\,\frac{t^{\Phi_\max-\omega(A_1)}}{1-t^{-\omega(A_1)}}.
\end{align*}
}
% \end{theo*}

\subsection*{Interest of our approach}
\label{sec:our-approach}

This work is closely related to recent work by Fukaya, Oh, Ohta, and Ono \cite{FOOO10}, Gonz\'alez and Iritani \cite{GonzalesIritani11}, and Chan, Lau, Leung, and Tseng \cite{CLLT}. Roughly speaking, for toric NEF symplectic manifolds, on one side Fukaya, Oh, Ohta, and Ono showed that quantum homology is isomorphic to the Jacobian of the open Gromov--Witten invariants generating function, $\mathrm{Jac}(W^{\mathrm{open}})$. On the other side, Gonz\'alez and Iritani expressed the Seidel elements in terms of Batyrev's elements via mirror maps. Finally, Chan, Lau, Leung, and Tseng proved that $W^{\mathrm{open}}$ coincides with the Hori--Vafa superpotential. Then by using this open mirror symmetry and the aforementioned results, they showed that the Seidel elements correspond to simple explicit monomials in $\mathrm{Jac}(W^{\mathrm{open}})$. In the 4--dimensional case, these results are clearly related to ours -- see for example the discussion on the Landau--Ginzburg superpotential in Example \ref{expl:45-pt-bu} below --, however our approach is somehow more elementary and stays on the symplectic side of the mirror. 

We now sketch our approach. The Seidel element of a symplectic manifold $(M,\omega)$ associated to a loop of Hamiltonian diffeomorphisms $\phi$ based at identity is defined by counting pseudo-holomorphic sections of $(M_\phi,\Omega)$ which is a symplectic fibration over $S^2$ with fibre $M$ and whose monodromy along the equator is given by $\phi$ (this construction is called the clutching construction, see Section \ref{sec:TotalSpaceToric} for more details). To compute Seidel's elements when $(M,\omega)$ is a toric 4--dimensional symplectic manifold and $\phi=\Lambda$ is one of the distinguished circle actions, we proceed as follows.
\begin{enumerate}
\item Following Gonz\'alez and Iritani \cite{GonzalesIritani11}, and Chan, Lau, Leung, and Tseng \cite{CLLT}, we notice that $(M_\Lambda,\Omega)$ is a toric 6--dimensional symplectic manifold, see Proposition \ref{prop:TotalSpace-Toric}\footnote{Actually, this first step does not require $M$ to be 4--dimensional.}. This allows us to reduce the computation of the Seidel elements to the computation of some 1--point Gromov--Witten invariants, see Section \ref{sec:proof-theor-refth}.
\item Then we compute the latter by induction using localization formulas from Spielberg's \cite{Spielberg99,Spielberg00} or Liu's \cite{Melissa13} for the base cases and the splitting axiom satisfied by Gromov--Witten invariants for the inductive steps, see Section \ref{sec:proof-theor-GWinv}.
\item Step (2) completely ends the computation up to some particular 0--point Gromov--Witten invariants which we preliminarily compute using a localization argument, see Section \ref{sec:proof-prop-nopointsinv}.
\end{enumerate}

\subsection*{Application in terms of Seidel's morphism and quantum homology}
\label{sec:appli-Seidel-morph}

As mentioned above, Seidel's morphism has been extensively studied for its applications. However not many things are known concerning $\m S$ itself, for example its injectivity. It is obvious that Seidel's morphism is trivial for symplectically aspherical manifolds since these particular manifolds do not admit non-constant pseudo-holomorphic spheres at all. In \cite{Seidel97}, Seidel showed that for all $m\geq 1$ Seidel's morphism detects an element of order $m+1$ in $\pi_1(\ham( \bbcp^{m},\omega_{\mathrm{st}}))$, with $\omega _{\mathrm{st}}$ the Fubini--Study symplectic form. In the case of $ \bbcp^2$ for example, this makes the Seidel morphism injective. Determining non-trivial elements of the kernel of $\m S$ in cases when $\m S$ is not ``obviously'' trivial would be interesting, for example to test the Seidel-type second order invariants introduced by Barraud and Cornea via their spectral sequence machinery \cite{BarraudCornea}. In order to find such classes, one should first compute all the Seidel elements in specific cases; here are families of examples for which the present work allows such computations.

\begin{exam}[Hirzebruch surfaces]
\label{expl:Hirz-Surf}
  It is well-known that Hirzebruch surfaces $\bb F_{2k}$ are symplectomorphic to $S^2 \times S^2$ endowed with the split symplectic form $\omega_\mu$ with area $\mu \geq 1$ for the first $S^2$--factor, and with area 1 for the second factor. 
  Recall that $\bb F_0$ is Fano, $\bb F_2$ is NEF, and that for all $k\geq 2$, $\bb F_{2k}$ admits spheres with negative first Chern number. As we shall see in Section \ref{sec:non-nef-examples}, the computations we present in this paper  allow us not only to compute directly the Seidel elements associated  to the circle actions of $\bb F_{2}$, but also  to compute the Seidel elements associated to the circle actions of $\bb F_{2k}$ for all $k \geq 2$, that is,  in the non-NEF cases. We present explicitly the case of $\bb F_4$.
  
Similar computations can be made for $\bb F_{2k+1}$ which can be identified with the 1--point blow-up of $ \bbcp^2$ endowed with its different symplectic forms.
\end{exam}

\begin{exam}[2-- and 3--point blow-ups of $ \bbcp^2$]
\label{expl:23-pt-bu}
In the same spirit, consider the symplectic manifold obtained from $ \bbcp^2$ by performing 2 or 3 blow-ups. It carries a family of symplectic forms $\omega_\nu$, where $\nu>0$ determines the cohomology class of $\omega_\nu$. It is well-known that it is symplectomorphic to $M_{\mu,c_1}$ or $M_{\mu,c_1,c_2}$, respectively the 1-- or 2--point blow-up of $S^2\times S^2$ endowed, as above, with the symplectic form $\omega_\mu$. Here, $c_1$ and $c_2$ are the capacities of the blow-ups. 

 In previous works, Pinsonnault \cite{Pinsonnault}, and Anjos and Pinsonnault \cite{AnjosPinsonnault13} computed the homotopy algebra of the Hamiltonian diffeomorphism groups of $M_{\mu,c_1}$ and $M_{\mu,c_1,c_2}$. In particular they showed that all the generators of its fundamental group  do not depend on the symplectic form nor the size of the blow-ups provided that $\mu>1$. In both cases, all the generators but one can be obtained as Hamiltonian circle actions associated to a Fano polytope while the last one is associated to a NEF polytope. When $\mu=1$, the fundamental group of the Hamiltonian diffeomorphism group is generated only by the former. So the computations we present here again allow us to compute all the Seidel elements of the 2-- and 3--point blow-ups of $ \bbcp^2$, regardless of the symplectic form and sizes of the blow-ups.
\end{exam}

Then we turn to quantum homology. Following \cite{McDuffTolman06}, we deduce from the expression of the Seidel elements described in Theorem \ref{theo:SeidelsMorphism} a presentation of the quantum homology of 4--dimensional NEF toric manifolds. Batyrev \cite{Batyrev93} and Givental \cite{Giv1,Giv2} showed that the quantum homology of Fano toric manifolds is isomorphic to a polynomial ring quotiented by relations given as the derivatives of the well-known Landau--Ginzburg superpotential. For NEF toric manifolds see also the works by  Chan and Lau  \cite{CL}, Fukaya, Oh, Ohta, and Ono \cite{FOOO10,FOOO10a}, Iritani \cite{Iritani}, Usher \cite{Usher10}, and references therein.   As an application of our computations we are able to give explicit  expressions for the potential in the NEF case which can be read directly from the moment polytope, and obviously can be related with Chan and Lau's results.

\begin{exam}[4-- and 5--point blow-ups of $ \bbcp^2$]
\label{expl:45-pt-bu}
To illustrate what is explained above, we explicitly compute the Seidel elements of the 4-- and 5--point blow-ups of $ \bbcp^2$. Note that these manifolds are NEF and do not admit any Fano almost complex structure. Then we deduce their quantum homology and we give the explicit expression of the related Landau--Ginzburg superpotential, see Section \ref{sec:nef-exples-45pt-bu}. Of course, this expression agrees with Chan and Lau's result \cite{CL} and in Remark \ref{rema:CLpotential} we indicate how.
\end{exam}

\subsection*{Extensions and applications}
\label{sec:motivations}

We now discuss some extensions of Seidel's morphism for which there is hope to get explicit information in the setting of and with similar techniques as the ones used in the present work.

\subsubsection*{Homotopy of $\ham$ in higher degrees}
\label{sec:quant-char-class}
As mentioned above, since \cite{Pinsonnault} and \cite{AnjosPinsonnault13} the homotopy algebra of the Hamiltonian diffeomorphism groups of the 2-- and 3--point blow-ups of $ \bbcp^2$ is completely understood. It would be interesting in this case to compute explicitly some invariants of the higher-degree homotopy groups generalizing Seidel's construction: the Floer-theoretic invariants for families defined by Hutchings in \cite{Hutchings} and the quantum characteristic classes introduced by Savelyev in \cite{Savelyev}.
Briefly recall that the former are morphisms $\pi_*(\ham(M,\omega)) \rightarrow \mathrm{End}_{*-1}(QH_*(M,\omega))$ obtained as higher continuation maps in Floer homology. The latter are defined via parametric Gromov--Witten invariants and lead to ring morphisms $H_*(\Omega\ham(M,\omega),\bb Q) \rightarrow QH_{2n+*}(M,\omega)$. Both constructions restrict to the Seidel representation, respectively in degree 1 and 0.

\subsubsection*{Bulk extension}
\label{sec:bulk-extension}
In this paper, what is called quantum homology should more precisely be refered to as the \textit{small} quantum homology ring. There is also a notion of \textit{big} quantum homology ring, obtained by considering not only the usual quantum product but also a family of deformations via even-degree cohomology classes of $M$, see e.g Usher \cite{Usher10} and Fukaya, Oh, Ohta, and Ono \cite{FOOO11} for a precise definition. For $b \in H^{\mathrm{ev}}(M)$, one ends up with $QH_*^{b}(M,\omega)$ isomorphic to $QH_*(M,\omega)$ as a vector space but with a twisted product. In \cite{FOOO11}, the authors extended Seidel's morphism to morphisms $\pi_1(\mathrm{Ham}(M,\omega)) \rightarrow QH_*^{b}(M,\omega)^\times$ and generalized in the toric case part of the results of McDuff and Tolman \cite{McDuffTolman06}. It would be interesting to see which information on the big quantum homology can be extracted from the present work.

\subsubsection*{Lagrangian setting}
\label{sec:lagrangian-setting}

The Seidel morphism has been extended to the Lagrangian setting in works by Hu and Lalonde \cite{HuLalonde}, and Hu, Lalonde, and Leclercq \cite{HuLalondeLeclercq}. Following McDuff and Tolman \cite{McDuffTolman06}, Hyvrier \cite{Hyvrier} computed the leading term of the Lagrangian Seidel elements associated to circle actions preserving some given monotone Lagrangian. He showed that when the latter is the real Lagrangian of a Fano toric manifold, all lower order terms vanish. It could be interesting to study the Lagrangian case in NEF toric manifolds, however the preliminary question of the structure of the lower order terms has to be tackled with different techniques than the ones used in \cite{Hyvrier} since they require the use of almost complex structures which generically lacks regularity. Let us also mention that Hyvrier's work as well as such a possible extension provide examples where one can apprehend the categorical refinement of the Lagrangian Seidel morphism due to Charette and Cornea \cite{CharetteCornea}.

\subsection*{Organization of the paper}

The paper is organized as follows. In Section \ref{sec:toric-manif-quant} we review the necessary background material, that is toric geometry, quantum homology, and Gromov--Witten invariants. Section \ref{sec:example} is devoted to the case of toric 4--dimensional NEF manifolds where we specify these notions. In Section \ref{sec:seidel-morphism-nef}, we precisely state the main theorems enumerating all the contributions to the Seidel morphism and the expression of the Seidel elements (Section \ref{sec:Seidelmorphism}) and we prove them (Section \ref{sec:proof-theor-refth} to Section \ref{sec:proof-theor-GWinv}). Finally, we describe explicit examples and applications mentioned in the introduction in Section \ref{sec:appl-expl-exampl}. In Appendix \ref{sec:BrutalComputationsSeidSMorph} we gather additional computations of Seidel's elements in more cases,  completing Theorem \ref{theo:SeidelsMorphism}.

\subsection*{Acknowledgments}
The authors would like to thank Eveline Legendre who first suggested to consider the total space of the fibration obtained via the clutching construction as a toric manifold, and Eduardo Gonz\'alez who suggested to investigate non-NEF Hirzebruch surfaces. They are also grateful to both of them as well as to Dusa McDuff and Siu-Cheong Lau for enlightening discussions at different stages of this work. They are particularly indebted to Chiu-Chu Melissa Liu for explaining to them the standard arguments appearing in the computations of the 0--point Gromov--Witten invariants of Proposition \ref{nopointsinvariants}. Finally, we would like to thank the referee for his/her hard and thorough work in reviewing the paper. We greatly appreciate his/her comments and we think that the modifications based on his/her suggestions have vastly improved the paper. Any remaining mistake is of course the responsibility of the authors. 

The present work is part of the authors's activities within CAST,  a Research Networking Program of the European Science Foundation.
The first author is partially funded by FCT/Portugal through project PEst-OE/EEI/LA0009/2013 and project PTDC/\-MAT/\-117762/2010. The second author is partially supported by ANR Grant ANR-13-JS01-0008-01.

% ***********************************************************************************************
%
%                                         PRELIMINARIES
%
% ***********************************************************************************************

\section{Toric manifolds and quantum homology}
\label{sec:toric-manif-quant}

\subsection{Toric geometry: the symplectic viewpoint}
\label{sec:toric-geom-sympl}

Recall that a closed symplectic $2m$--di\-men\-sional manifold $(M, \omega)$ is said to be toric if it is equipped with an effective Hamiltonian action of a $m$--torus $T$ and with a choice of a corresponding moment map $\Phi: M \to \mfr t^*$, where $\mfr t^*$ is the dual of the Lie algebra $\mfr t$  of the torus $T$. There is a natural integral lattice $\mfr t_\Z$ in $\mfr t$ whose elements $H$ exponentiate to circles $\Lambda_H$ in $T$, and hence also a dual lattice $\mfr t^*_\Z$ in $\mfr t^*$. The image $\Phi(M)$ is well-known to be a convex polytope $P$, called a {\it Delzant polytope}. It is {\it simple} ($m$ facets meet at each vertex), {\it rational} (the conormal vectors $\eta_i \in \mfr t$ to each facet may be chosen to be primitive and integral), and {\it smooth} (at each vertex $v \in P$ the conormals to the $m$ facets meeting at $v$ form a basis of the lattice $\mfr t_\Z$). We describe them as follows:
$$ P = P ({\kappa}) := \{ x \in {\mfr t^*} | \langle \eta_i, x \rangle \leq \kappa_i, \, i=1, \ldots, n\}, $$
where $P$ has $n$ facets $D_1, \hdots, D_n$ with {\it outward\footnote{It seemed more relevant to follow the same convention as in \cite{McDuffTolman06} even though the polytope is often defined by $P'=\{ x \in {\mfr t^*} | \langle \eta_i', x \rangle \geq - \kappa_i, \, i=1, \ldots, N\}$ for inward normals $\eta_i'$.} primitive integral conormals $\eta_i \in {\mfr t}_\Z$} and {\it support constants} $\kappa=(\kappa_1,\hdots, \kappa_n) \in \R^n$.

Delzant showed in \cite{Delzant88} that there is a one--to--one  correspondence between toric manifolds and Delzant polytopes given by the map that sends the  toric manifold $(M, \omega, T, \Phi)$ to the polytope $\Phi(M)$. (See \cite{KarshonKesslerPinsonnault} and the references therein for more details on this background material.)

\subsection{The clutching construction}\label{sec:TotalSpaceToric}
Let $\momeg$ be a closed symplectic manifold and $\Lambda=\{ \Lambda_\theta \}$ be a loop in $\ham{\momeg}$ based at identity. Denote by $M_\Lambda$ the total space of the fibration over $\bbcp^1$ with fiber $M$ which consists of two trivial fibrations over 2--discs, glued along their boundary via $\Lambda$. Namely, we consider $\bbcp^1$ as the union of the two 2--discs
\begin{align*}
  D_1 = \{ [1:z]\in \bbcp^1 \;|\, |z| \leq 1  \} \quad \mbox{and} \quad D_2 = \{ [z:1]\in \bbcp^1 \;|\, |z| \leq 1  \}
\end{align*}
glued along their boundary 
\begin{align*}
 \del D_1 =\{ [1:\ee^{-2i\pi\theta}] ,\, \theta\in[0,1[ \}=\{ [\ee^{2i\pi\theta}:1] ,\, \theta\in[0,1[ \}= \del D_2 \,.
\end{align*}
The total space is
\begin{align*}
  M_\Lambda = \quot{\big(M \times D_1 \bigsqcup  M \times D_2\big)}{\sim_\Lambda} \quad \mbox{with } (x,[1:\ee^{-2i\pi\theta}]) \sim (\Lambda_\theta(x),[\ee^{2i\pi\theta}:1]) .
\end{align*}
This construction only depends on the homotopy class of $\Lambda$. Moreover, $\Omega$, the family (parameterized by $S^2$) of symplectic forms of the fibers, can be ``horizontally completed'' to give a symplectic form on $M_\Lambda$, $\omega_{\Lambda,\kappa} = \Omega + \kappa\cdot \pi^*(\omega_0)$ where $\omega_0$ is the standard symplectic form on $S^2$ (with area 1), $\pi$ is the projection to the base of the fibration and $\kappa$ a big enough constant to make $\omega_{\Lambda,\kappa}$ non-degenerate. (Once chosen, $\kappa$ will be omitted from the notation.)

So we end up with the following Hamiltonian fibration: $$\xymatrix{\relax \momeg \ar@{^(->}[r] & (M_\Lambda,\omega_\Lambda) \ar[r]^{\pi} & (S^2, \omega_0).}$$
In \cite{McDuffTolman06}, McDuff and Tolman observed that, when $\Lambda$ is a circle action (with associated moment map $\Phi_\Lambda$), the clutching construction can be simplified since, then, $M_\Lambda$ can be seen as the quotient of $M \times S^3$ by the diagonal action of $S^1$, $e^{2\pi i \theta}\cdot(x,(z_1,z_2))=(\Lambda_\theta(x),(e^{2\pi i \theta} z_1, e^{2\pi i \theta} z_2))$. 
The symplectic form also has an alternative description in $M \times_{S^1} S^3$. Let $\alpha \in \Omega^1(S^3)$ be the standard contact form on $S^3$ such that $d\alpha = \chi^*(\omega_0)$ where $\chi: S^3 \to S^2$ is the Hopf map and $\omega_0$ is the standard area form on $S^2$ with total area 1. For all $c\in\bb R$, $\omega+cd\alpha -d(\Phi_\Lambda\alpha)$ is a closed 2--form on $M \times S^3$ which descends through the projection, $p \co M \times S^3 \rightarrow M \times_{S^1} S^3$, to a closed 2--form on $M_\Lambda$:
\begin{align}
  \label{eq:omega-c}
  \omega_c = p( \omega + c d\alpha -d(\Phi_{\Lambda} \alpha))
\end{align}
which extends $\Omega$. Now, if $c>\max \,\Phi_\Lambda$, $\omega_c$ is non-degenerate and coincides with $\omega_{\Lambda,\kappa}$ for some big enough $\kappa$.

In the case of a toric symplectic manifold fiber, the same arguments lead to the fact that  $(M_\Lambda,\omega_\Lambda)$ itself is toric. This fact has been already noticed  and used in other instances, e.g. by Gonz\'ales--Iritani \cite[Section 3.2]{GonzalesIritani11} and Chan--Lau--Leung--Tseng \cite[Section 4]{CLLT} in more general settings than what we will need in this paper, so that we only give here the specific statement which we will need, and refer the reader to the aforementioned works for details.

\begin{prop}\label{prop:TotalSpace-Toric}
Let $(M^{2m},\omega,T, \Phi)$ be a toric symplectic manifold with associated Delzant polytope $P$. Denote by $(M_\Lambda,\omega_\Lambda)$ the total space resulting from the clutching construction associated to $\Lambda$, Hamiltonian circle subgroup of $T$. $\Lambda$ admits a representative in $T$ given as the exponential of $\theta b $ where $\theta \in [0,1]$ and $b \in \mfr t_\Z$, the lattice of circle subgroups of $T$. 

Then there exist a $(m+1)$--dimensional torus $T_\Lambda \subset \ham(M_\Lambda,\omega_\Lambda)$, and a moment map $\Upphi_\Lambda\co M_\Lambda \rightarrow  \mfr t_\Lambda^* \simeq \mfr t^* \times\bb R$ such that $(M_\Lambda,\omega_\Lambda, T_\Lambda, \Upphi_\Lambda)$ is a toric symplectic manifold, whose associated Delzant polytope $P_\Lambda$ and integral lattice $\mfr t^{\Lambda}_\Z$  are  given by
\begin{align*}
    P_\Lambda = \{ (x, x_0)\in (\mfr t\times\bb R)^* \,|\; x\in P,\, c' + \langle x,b \rangle \leq x_0 \leq 0 \}  \; \mbox{ and } \; \mfr t^{\Lambda}_\Z =\mfr t_\Z \times \Z \subset \mfr t \times \R
\end{align*}
where $c'> \max \{ \langle x,b \rangle, x\in P\}$ coincides with the constant $c$ appearing in \eqref{eq:omega-c} above.

Moreover, the outward normals of $P_\Lambda$, $\eta_\Lambda$, are given in terms of the ones of $P,$ $\eta$, as follows:
\begin{align*}
    \eta_\Lambda = \{ (\eta_i,0), \eta_i\in \eta \} \cup \{(0,1),(b,-1)\} \,.
\end{align*}
\end{prop}
The polytopes $P$ and $P_\Lambda$ are  illustrated by Figure~\ref{fig:Pb}. The upper and lower facets of $P_\Lambda$ correspond to two copies of $P$, the former horizontal, the latter orthogonal to $(b,-1)\in \mfr t^*\times \bb R$.

\begin{figure}
  \centering
  \includegraphics{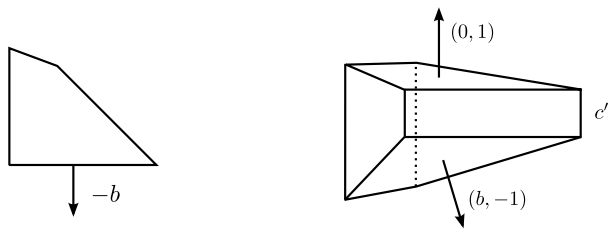}
  \caption{The polytopes associated to the fiber $M$ (left) and to the total space $M_\Lambda$ (right)}
  \label{fig:Pb}
\end{figure}

\subsection{Toric geometry: the algebraic viewpoint} \label{sec:toric-geom-algebr}

We now briefly review toric varieties since we will use this viewpoint extensively. Good basic references are Cox--Katz \cite{Cox-Katz99} and Batyrev \cite{Batyrev93}. There is also a good summary of the definition and some properties of smooth toric varieties in Spielberg \cite{Spielberg00}. In what follows we mainly use his notation.

Let $m > 0$ be a positive integer, $\mfr t_\Z = \Z^m$ be the $m$--dimensional integral lattice and $\mfr t_\Z^* = \mathrm{Hom}(\mfr t_\Z, \Z)$ be its dual space. Moreover, let $\mfr t = \mfr t_\Z \otimes_\Z \R$ and $\mfr t^* = \mfr t_\Z^* \otimes_\Z \R$ be the $\R$--scalar extensions of $\mfr t_\Z$ and $\mfr t_\Z^*,$ respectively.

A convex subset $\sigma \subset \mfr t$ is called a {\it regular $k$--dimensional cone} ($1\leq k\leq m$) if there exists  a $\Z$--basis of $\mfr t_\Z$, $\{\eta_1, \hdots, \eta_m\}$, such that the cone $\sigma$ is generated by $\eta_1, \hdots, \eta_k.$ The vectors $\eta_1, \hdots, \eta_k \in \mfr t_\Z$ are the {\it integral generators of $\sigma$}. If $\sigma'$ is a (proper) face of $\sigma$, we will write $\sigma' \prec  \sigma$. A finite system $\Sigma = \{ \sigma_1, \hdots \sigma_s\}$ of regular cones in $\mfr t$ is called a regular {\it $m$--dimensional fan} of cones if any face $\sigma'$ of a cone $\sigma \in \Sigma$ is in the fan and any intersection of two cones $\sigma_1, \sigma_2 \in \Sigma$ is again in the fan. A fan $\Sigma$ is called a {\it complete fan} if $\mfr t = \cup_i \sigma_i$. The $k$--skeleton $\Sigma^{(k)}$ of the fan $\Sigma$ is the set of all $k$--dimensional cones in $\Sigma$.  A subset of the 1--skeleton of a fan is called a {\it primitive collection} of $\Sigma$ if it is not the set of generators of a cone in $\Sigma$, while any of its proper subset is. We will denote the set of primitive collections of $\Sigma$ by $\mathcal P$. 

Suppose the 1--skeleton of $\Sigma$ is given by $\eta_1, \hdots, \eta_d$. Let $z_1, \hdots, z_d$ be a set of coordinates in $\C^d$ and let $\iota \colon \C^d \to \mfr t_\Z \otimes_\Z \C$ be a linear map such that $\iota (z_i)= \eta_i$. For each primitive collection $P=\{ \eta_{i_1}, \hdots, \eta_{i_p}\}$ we define a $(d-p)$--dimensional affine subspace in $\C^d$ by
$$ A(P):= \{ (z_1, \hdots, z_d) \in \C^d \mid \ z_{i_1}= \hdots=z_{i_p}=0 \}. $$  
Moreover, we define the set $U(\Sigma)$ to be the open algebraic subset of $\C^d$ given by 
$$ U(\Sigma)= \C^d \; \backslash \bigcup_{p \in \mathcal P} A(P).$$
The map $\iota \colon \C^d \to \mfr t_\C$ induces a map between tori $(\C^*)^d \to (\C^*)^m$ that we will also call $\iota$. Its kernel, $D(\Sigma) := \ker (\iota \colon (\C^*)^d \to (\C^*)^m)$, is a $(d-m)$--dimensional subtorus. Then the quotient $$X_\Sigma := U(\Sigma) /D(\Sigma)$$ is called the {\it toric manifold associated to $\Sigma$}. Note that there is a torus  of dimension $m$ acting on $X_\Sigma$. Moreover, Delzant \cite{Delzant88} showed that if  $X_\Sigma$ is a projective simplicial toric variety then it can be constructed as a symplectic quotient and therefore it is endowed with a symplectic form $\omega$ (it is also endowed with an action of a $m$--dimensional torus). From the moment polytope of this symplectic toric manifold it is possible to recover the fan $\Sigma$. However, as explained in \cite[Part B]{Cannas03}, changing the cohomology class of the symplectic form corresponds to changing the lengths of the edges of the polytope. The size of the faces of a polytope cannot be recovered from the fan which only encodes the combinatorics of the faces. Hence, the fan does not give the cohomology class of the symplectic form. 

Standard results about toric manifolds explain how to obtain the cohomology ring of the toric variety $X_\Sigma$. Assume the moment map $\Phi: X_\Sigma  \to \mfr t^*$ is chosen so that each of its components is mean-normalised. Let $P_\Sigma \subset \mfr t^*$ be the image of the moment map.  Let $D_1, \hdots, D_n$ be the facets of $P$ (the codimension--1 faces), and let $\eta_1, \hdots, \eta_n \in \mfr t_\Z$ denote the outward primitive integral normal vectors. Let $C$ be the set of subsets $I=\{ i_1, \hdots, i_k \} \subseteq \{1, \hdots,n \}$ such that $1\leq k \leq m$ and $D_{i_1} \cap \hdots \cap D_{i_k} \neq \emptyset.$ Consider the two following ideals in $\Q[Z_1, \hdots,Z_n]$:
$$ {\Lin}(\Sigma) = \left\langle \sum (x, \eta_i)Z_i \mid x \in \mfr t_\Z^* \right\rangle \quad \mbox{and}  \quad {\mathrm {SR}}(\Sigma)= \langle Z_{i_1}\hdots Z_{i_k} \mid \{ i_1, \hdots, i_k \} \notin C\rangle.$$
The ideal $ {\Lin}(\Sigma)$ is generated by linear relations and the ideal ${\mathrm {SR}}(\Sigma)$ is called the {\it Stanley--Reisner ideal.}
A subset $I \subseteq \{1, \hdots,n \}$ is called {\it primitive} if $I$ is not in $C$ but every proper subset is. Clearly, 
$$ {\mathrm {SR}}(\Sigma)= \langle Z_{i_1}\hdots Z_{i_k} \mid \{i_1, \hdots, i_k \} \subseteq \{1, \hdots,n \}  \mbox{ is primitive} \rangle.$$
The map which sends $Z_i$  to the Poincar\'e dual of $\Phi^{-1}(D_i)$ (which we shall also denote by $Z_i \in H^2(X_\Sigma;\Q)$) induces an isomorphism  
\begin{equation}\label{cohomologyring}
H^*(X_\Sigma; \Q) \cong \R[Z_1, \hdots, Z_n]/ \langle{\mathrm {Lin}}(\Sigma)+ {\mathrm {SR}}(\Sigma)\rangle.
\end{equation}
Moreover, there is a natural isomorphism between $H_2(X_\Sigma; \Z)$ and the set of tuples $(a_1, \hdots, a_n) \in \Z^n$ such that $\sum a_i\eta_i=0$, under which the pairing between such an element of $H_2(X_\Sigma; \Z)$ and $Z_i$ is $a_i$. The linear functional $\eta_i$ is constant on $D_i$ and let $\eta_i(D_i)$ denote its value. Under the isomorphism of \eqref{cohomologyring}  we have 
\begin{equation}\label{ChernClass}
[\omega]= \sum_i \eta_i(D_i)Z_i  \quad \mbox{and} \quad c_1(X_\Sigma)= \sum_iZ_i.
\end{equation}
Dually, let $R(\Sigma) \subset \Z^n$ be the subgroup of $\Z^n$ defined by 
\begin{equation}\label{homology}
R(\Sigma) :=  \{ (\gamma_1, \hdots, \gamma_{n}) \in \Z^{n} \, |\,  \gamma_1 \eta_1 + \hdots + \gamma_{n}\eta_{n}=0 \} \cong \Z^{n-m}.
\end{equation}
Then the group $R(\Sigma)$ is canonically isomorphic to $H_2(X_\Sigma; \Z)$.

\subsection{Small quantum homology and Gromov--Witten invariants}\label{sec:quantum-homology}
Except for our application in terms of the Landau--Ginzburg potential in Section \ref{sec:appl-expl-exampl}, we will work with the (small) quantum homology ring with coefficients in the ring $\Pi:= \Pi^{\mathrm{univ}}[q,q^{-1}]$. The variable $q$ is of degree 2 and $ \Pi^{\mathrm{univ}}$ is a generalised Laurent series ring in a variable of degree 0: 
\begin{align}
  \label{eq:pi-univ}
  \Pi^{\mathrm{univ}}:= \left\{ \sum_{\kappa \in \R} r_\kappa t^\kappa \,\big|\, r_\kappa \in \Q, \ \#\{ \kappa > c\mid r_\kappa \neq 0\} < \infty, \forall c \in \R \right \} \,.
\end{align}
% Correspondingly, quantum cohomology has coefficients in the dual ring $$\check{\Pi}:= \check{\Pi}^{\mathrm{univ}}[q,q^{-1}]$$ where $q$ is as before and 
% $$ \check{\Pi}^{\mathrm{univ}}:= \left\{ \sum_{\kappa \in \R} r_\kappa t^\kappa : r_\kappa \in \Q, \ \#\{ \kappa < c\mid r_\kappa \neq 0\} < \infty, \forall c \in \R \right \}.$$
The quantum homology $\QH_*(M; \Pi) =  H_*(M, \Q) \otimes_\Q \Pi$ %, \quad \mbox{and}  \quad \QH^*(M; \check{\Pi}) =  H^*(M, \Q) \otimes_\Q \check{\Pi} $$ 
is $\Z$--graded  so that $\deg (a \otimes q^d t^\kappa)= \deg (a) +2d$ with $a \in H_*(M)$. % or $a \in H^*(M)$.
% The quantum product in homology is a deformation of the intersection product that is defined as the Poincar\'e dual of the quantum cup product. Thus 
The {\it quantum intersection product} $a*b \in \QH_{i+j -\dim M}(M; \Pi)$, of classes $a \in H_i(M)$ and $b \in H_j(M)$ has the form 
$$a*b = \sum_{B \in H_2^S(M;\Z)} (a*b)_B \otimes q^{-c_1(B)}t^{-\omega(B)},$$
where $H_2^S(M;\Z)$ is the image of $\pi_2(M)$ under the Hurewicz map. The homology class $(a*b)_B \in H_{i+j-\dim M +2c_1(B)}(M)$ is defined by the requirement that 
$$(a*b)_B \cdot_M c = \GW^M_{B,3}(a,b,c) \quad \mbox{ for all } c \in H_*(M).$$   
In this formula $ \GW^M_{B,3}(a,b,c) \in \Q$ denotes the Gromov--Witten invariant that counts the number of spheres in $M$ in class $B$ that meet cycles representing the classes $a,b,c \in H_*(M)$. The product $*$ is extended to $\QH_*(M)$ by linearity over $\Pi$, and is associative. It also respects the $\Z$--grading and gives $\QH_*(M)$ the structure of a graded commutative ring, with unit $[M]$.

Gromov--Witten invariants can also be interpreted as homomorphisms 
\begin{align*}
 & \GW_{A,k}^M \co H^*(M;\Q)^{\otimes k} \otimes H_*(\overline{\mathcal M}_{0,k}; \Q) \longrightarrow \Q \\
 & \GW_{A,k}^M(a_1, \hdots, a_k; \beta)= \int_{\overline{\mathcal M}_{0,k}(A;J)} \mathrm {ev}_1^*a_1 \cup \hdots \mathrm {ev}_k^*a_k \cup \pi^*\PD (\beta),
\end{align*}
where $\overline{\mathcal M}_{0,k}(A;J)$ is the compactified moduli space of $J$--holomorphic spheres with $k$ marked points in $M$ representing the homology class $A$.  % This moduli space carries a natural evaluation map $\mathrm {ev}$ to $M^k$ and a natural projection, the forgetful map, to the moduli space $\overline{\mathcal M}_{0,k}$ of stable curves of genus 0 with $k$ marked points. Geometrically, one can think of this invariant as counting the number of $J$--holomorphic spheres with $k$ marked points that pass, at the $i$--th marked point, through a cycle $X_i \subset M$ Poincar\'e dual to $a_i$, and such that the image of the curve under the projection $\pi$ is restricted to a cycle $Y \subset \overline{\mathcal M}_{0,k}$ representing the class $\beta$. 
Let us recall that in general $\GW^M_{A,k}(a_1, \ldots, a_k)$ is the homomorphism
\begin{align*}
 \GW_{A,k}^M \co H^*(M;\Q)^{\otimes k} \rightarrow \Q \,, \quad  (a_1, \hdots, a_k) \mapsto \GW_{A,k}^M(a_1, \hdots, a_k; [\overline{\m M}_{0,k}])
\end{align*}
so that when $k=3$, $\GW_{A,3}^M(a_1, a_2, a_3)=\GW_{A,3}^M(a_1, a_2, a_3; \pt)$.

For easy reference, we gather here the properties of Gromov--Witten invariants which will be used explicitly at several places in the computations of Section \ref{sec:seidel-morphism-nef}: The first two are extracted from \cite[Proposition 7.5.6]{McDuffSalamon04} and the third is the particular case of \cite[Theorem 7.5.10]{McDuffSalamon04} for the invariants  $\GW^M_{A,4}(a_1, \ldots, a_4; \pt) = \GW^{M,\{ 1,2,3,4 \}}_{A,4}(a_1, \ldots, a_4)$ (see  \cite[Remark 7.5.1.(vi)]{McDuffSalamon04}) when $k=4$.
\begin{prop}\label{prop:axiom-for-GW-inv}
  Let $(M,\omega)$ be a semipositive compact symplectic manifold, $A \in H_2(M;\bb Z)$, $k \geq 1$, and $a_1, \ldots, a_k \in H^*(M;\Q)$. Then the following properties hold.
  \begin{itemize}
  \item[] \textbf{(Divisor)} If $(A,k) \neq (0,3)$ and $\mathrm{deg}(a_k)=2$ then
    \begin{align*}
      \GW^M_{A,k}(a_1, \ldots, a_k) = \GW^M_{A,k-1}(a_1, \ldots, a_{k-1}) \cdot \int_A a_k \,.
    \end{align*}
  \item[] \textbf{(Zero)} If $k \neq 3$ then $\GW^M_{0,k}=0$. If $k=3$ then
    \begin{align*}
      \GW^M_{0,3}(a_1,a_2,a_3)=\int_M a_1 \cup a_2 \cup a_3 \,.
    \end{align*}
  \item[] \textbf{(Splitting)} If $k=4$ then $\GW^M_{A,4}(a_1, \ldots, a_4; \pt)$ is equal to
    \begin{align*}
        \sum_{A=A_0+A_1} \sum_{\nu,\mu} \GW^M_{A_0,3}(a_1,a_2,e_\nu) \, g^{\nu\mu} \, \GW^M_{A_1,3}(e_\mu,a_3,a_4) 
    \end{align*}
where $(e_{\nu})_{\nu}$ is a basis of $H^*(M;\Q)$, $g_{\nu\mu}$ are the coefficients of the cup-product matrix: $g_{\nu\mu}=\int_{M} e_\nu \cup e_\mu$, and $g^{\nu\mu}$ the coefficients of its inverse.
  \end{itemize}
\end{prop}

\subsection{Gromov--Witten invariants of toric manifolds}
\label{SpielbergFormula}
In this section we present Spielberg's formula from \cite[Corollary 8.4]{Spielberg00}  for the computation of Gromov--Witten invariants of toric manifolds, which we will use in Section \ref{mainexample}. Note that Liu proved a more general result in \cite{Melissa13}, however since we only need to compute genus--0 Gromov--Witten invariants we will use  Spielberg's formulation and notation. 
\begin{defi} \cite[Definition 6.4]{Spielberg00}
Let $\Sigma$ be a complete regular fan in $\Z^m$ and let $P_\Sigma$ be its dual polytope. A graph  $\Gamma$ is a finite 1--dimensional CW--complex with the following decorations: 
\begin{enumerate}[1]
\item[1.] A map $\sigma: \Vertices(\Gamma) \to \Sigma^{(m)}$ mapping each vertex $\vb$ of the graph to a vertex $\sigma(\vb)$ of $P_\Sigma$; 
\item[2.] A map $d: \Edge(\Gamma) \to \Z_{>0}$ representing multiplicities of maps;
\item[3.] A map $S: \Vertices(\Gamma) \to  \mathfrak{B}(\{1,\dots, p\})$ associating to each vertex a set of marked points. 
\end{enumerate}
These decorations are subject to the following compatibility conditions: 
\begin{enumerate} [(a)]
\item If an edge  $e \in \Edge(\Gamma)$ connects two vertices $ \vb_1,\vb_2 \in \Vertices (\Gamma)$  labeled $\sigma(\vb_1)$ and $\sigma(\vb_2)$, then the two cones must be different and have a common  $(m-1)$--dimensional face: $\sigma(\vb_1) \cap \sigma(\vb_2) \in \Sigma^{(m-1)}$;
\item The graph represents a stable map with homology class $A$;
\item The CW--complex $\Gamma$ contains no loops; 
\item For any two vertices $\vb_1,\vb_2 \in \Vertices (\Gamma)$, the sets of associated marked points are disjoint: $S(\vb_1) \cap S(\vb_2) = \emptyset $;
\item Every marked point is associated with some vertex. 
\end{enumerate}
\end{defi}

The following notation will be useful to understand the statement of the theorem. 
We define the following subset of $\Vertices$: 
$$ \Vertices_{t,s} := \left\{ \vb \in \Vertices  \,\big|\, \val(\vb)=t, \, \deg(\vb) = t+s \right\},$$
where $\val: \Vertices (\Gamma) \to \Z_{>0}$ is the function assigning to each vertex the number of outgoing edges and $\deg: \Vertices (\Gamma) \to \Z_{>0}$ assigns to each vertex the number of its special points:
$$\deg (\vb) = \# S(\vb) + \# \{ e \in \Edge (\Gamma) \, | \, \vb \in \partial e \}$$
where $\partial e = \{ \vb_1(e), \vb_2(e)\} $ associates to an edge $e$ the two vertices $\vb_1(e), \vb_2(e)$ it connects. 

We also need the following result: 
\begin{lemm}\cite[Lemma 6.10]{Spielberg00} \label{weights}
Let $\sigma_1, \sigma_2 \in \Sigma^{(m)}$ be two $m$--cones in $\Sigma$ that have a common $(m-1)$--face $  \tau \in \Sigma^{(m-1)}$.  Let $\eta_{i_1}, \hdots, \eta_{i_{m-1}}$ be the generators of the common face  $\tau$, such that 
$$ \sigma_1 = \langle  \eta_{i_1}, \hdots, \eta_{i_{m-1}},\eta_{\tau(1)} \rangle \quad \mbox{and} \quad \sigma_2 = \langle  \eta_{i_1}, \hdots, \eta_{i_{m-1}},\eta_{\tau(2)}\rangle.$$ 
Let $\omega_1, \hdots, \omega_{n}$ be the weights of a diagonal action of $(\C^*)^{n}$ on $\C^{n}$ with respect to the standard basis.  The induced $\C^*$--action on the invariant 2--sphere $V_{\tau}$ has weight $\omega^{\sigma_1}_{\sigma_2}$ at the point $V_{\sigma_1}$ given by 
$$ \omega^{\sigma_1}_{\sigma_2} := \sum_{\ell=1}^{n} \langle \eta_\ell, u_m \rangle \omega_\ell, $$
where $\left \{u_1, \hdots, u_m \right\}$ is a basis of  $\mfr t_\Z^*$ dual to  $ \left \{\eta_{i_1}, \hdots, \eta_{i_{m-1}},\eta_{\tau(1)}\right \}$.
\end{lemm}
\begin{coro}\cite[Corollary 6.11]{Spielberg00}
Let $e \in \Edge(\Gamma)$ and $\vb_1, \vb_2 \in \partial e$ be the vertices at its two ends. Let $\sigma_i = \sigma (\vb_i)$ be the $m$--cones of the vertices $\vb_i$ and $\tau= \sigma_1 \cap \sigma_2$ its common $(m-1)$--face, that are generated as in the Lemma above. For a stable 
map $(C; x_1, \hdots, x_p; f)$ fixed by the torus action, let $C_e$ be the irreducible component of $C$ corresponding to the edge $e$. Let 
$F:= (\vb_1,e) \in \Vertices(\Gamma) \times \Edge(\Gamma)$ be such that $\vb_1 \in \partial e$. At the point $p_F:= f^{-1}(V_{\sigma (\vb_1)})\cap C_e$, the pull back to $C_e$ of the torus action on $V_\tau$ has the weight $\omega_F$ at $p_F$: 
 $$ \omega_F := \frac{1}{d_e}\sum_{\ell=1}^{n} \langle \eta_\ell, u_m \rangle \omega_\ell,$$
 where $d_e$ is the multiplicity of the component $C_e$ and the vectors $u_i$ are as in the lemma above. 
\end{coro}

We will introduce some further notation, grouping together certain weights on a graph $\Gamma$.  We will write  $\sigma_1 \diamond \sigma_2$ for the property of $\sigma_1$ and $\sigma_2$  having a common $(m-1)$--dimensional proper face: 
$ \sigma_1 \diamond \sigma_2 \Longleftrightarrow \sigma_1, \sigma_2 \in \Sigma^{(m)} \ \mbox{and} \  \sigma_1 \cap \sigma_2 \in \Sigma^{(m-1)}.$
The \emph{total weight of a $m$--dimensional cone $\sigma$} is defined to be 
$$ \omega^{\sigma}_{\rm{total}}:= \prod_{\alpha \,:\, \alpha \diamond \sigma} \omega^\sigma_\alpha \,.$$
Finally, let $\alpha \in \Sigma^{(m)}$ be a $m$--cone in the fan $\Sigma$ that has a common $(m-1)$--face $\tau$ with $\sigma_1: \alpha \diamond \sigma_1$. Then $\alpha$ and $\sigma_1$ have  $(m-1)$ generators in common; let $\eta_{i_\alpha} \in \Sigma^{(1)}$ be the generator of $\sigma_1$ that is not a generator of $\alpha$. We then set
$\lambda_e^\alpha:= \gamma_{i_\alpha},$ where $(\gamma_1, \hdots, \gamma_n)$ represents the homology class of $\tau$ (see \eqref{homology}).  

 Since we are interested only in 1--point Gromov--Witten invariants we will give a simplified version of Spielberg's formula. 
 \begin{theo}\cite[Corollary 8.4]{Spielberg00} \label{mainSpielberg}
 The 1--point genus--0 Gromov--Witten invariants for a toric variety $X_\Sigma$ are given by
 $$GW_{A,1}^{X_\Sigma} (Z_\ell) = \sum_\Gamma \frac{1}{|A_\Gamma|} T_\Gamma \cdot S_\Gamma $$
 where $A_\Gamma$ is the automorphism group of the graph $\Gamma$, 
 \begin{align*}
  T_\Gamma  = & \prod_{t=1}^{\infty}\prod_{\vb \in \Vertices_{t,*}(\Gamma)}(\omega^{\sigma(\vb)}_{\rm{total}})^{t-1} \cdot \left(\prod_{i=1}^t \frac{1}{\omega_{F_i(\vb)}} \right) \cdot \left( \frac{1}{\omega_{F_1(\vb)}}+ \hdots +\frac{1}{\omega_{F_t(\vb)}} \right)^{t-3} \\
  & \cdot \prod_{\stackrel{e \in \Edge}{\partial e=\{ \vb_1,\vb_2\} }} \left( \frac{(-1)^m m^{2m}}{(m!)^2 (\omega^{\sigma_1}_{\sigma_2})^{2m}} \prod_{\substack{\alpha \,:\, \alpha \neq \sigma_2 \\ \text{and } \alpha  \diamond \sigma_1}}  \frac{\displaystyle \prod_{i=\lambda_{e}^{\alpha}+1}^{-1} \left( \omega^{\sigma_1}_{\alpha} -\frac{i}{m} \cdot \omega^{\sigma_1}_{\sigma_2} \right)}{\displaystyle \prod_{i=0}^{\lambda_{e}^{\alpha}} \left( \omega^{\sigma_1}_{\alpha} -\frac{i}{m}
  \cdot \omega^{\sigma_1}_{\sigma_2} \right)}\right)
\end{align*}
\begin{align*}
  S_\Gamma = & \left[ \prod_{t,s} \prod_{\vb \in \Vertices_{t,s}(\Gamma)} \left( \frac{1}{\omega_{F_1(\vb)}}+ \hdots +\frac{1}{\omega_{F_t(\vb)}} \right)^{s} \right] \cdot \prod_{k=1}^n (\omega_k^{\sigma(1)})^{l_k}
 \end{align*}
 and where
 \begin{itemize}
 \item[-] we use the convention $0^0=1$;
 \item[-]  $Z^l = Z_1^{l_1} \hdots Z_n^{l_n}$;
 \item[-]  $\sigma(1)$  is the fixed point the marked point is mapped to;
 \item[-]  we define $ \omega_k^{\sigma(1)}:= \left\{ \begin{array}{cl}
 0  & \mbox{if}  \quad \eta_k \notin \Sigma_{\sigma(1)}^{(1)}, \\
 \omega_\alpha^{\sigma(1)} & \mbox{if} \quad \alpha \diamond \sigma(1)  \ \mbox{and} \ \eta_k \in \Sigma_{\sigma(1)}^{(1)} \backslash \Sigma_{\alpha}^{(1)}.
 \end{array} \right. $ 
 \end{itemize}  
 \end{theo}

% ***********************************************************************************************
%
%                                         Toric 4--dimensional NEF manifolds
%
% ***********************************************************************************************

\section{Toric 4--dimensional NEF manifolds} \label{sec:example}
Now we restrict ourselves to the case of toric 4--dimensional NEF manifolds. We explain the construction of $M_\Lambda$ and its properties including its cohomology ring. This will play a very important role in the next section. 

\subsection{Toric and homological data}
\label{sec:toric-topol-data}

We consider a 4--dimensional toric manifold $(M, \omega, T, \Phi)$ and its moment 2--dimensional Delzant polytope $P$. Assume it has $n$ facets that we denote by $D_i$, $i=1, \hdots,n$. Let $v_1, \hdots, v_n$ denote the outward primitive integral normal vectors and let $\Lambda_i$ denote the circle action corresponding to $v_i$, that is, $\Lambda_i$ is the circle action whose moment map is given by $\Phi_{\Lambda_i}:= \langle v_i, \Phi (\cdot) \rangle$. 

We pick a $\omega$--tame almost complex structure $J$ and denote by $c_1(M)$ the first Chern class of $(TM,J)$. We assume that $(M,J)$ is NEF, that is $\langle c_1(M) , B \rangle \geq 0$ for every class $B \in H_2(M, \Z)$ with a $J$--pseudo-holomorphic representative.

Moreover, we consider the particular case 
when there are at most 2 (consecutive) facets corresponding to spheres with vanishing first Chern number and assume their normal vectors  are $v_{n}$ and $v_{1}$ (recall that we denote $v_{n+1}$ by $v_1$ as for the $D_i's$). Since the polytope $P$ is Delzant we can assume that the facets $D_{n-1}$ and $D_n$ are perpendicular. Moreover, 
as explained in \cite[Section 2.5]{Fulton}, the vectors $v_i$ satisfy the relations 
\begin{equation}\label{relationvectorsI}
 v_{i-1}+v_{i+1} = d_i v_i,
 \end{equation} 
 where $-d_i =D_i \cdot D_i$ denotes the self-intersection of the facet $D_i$. Since the first Chern number vanishes on the facets $D_n$ and $D_1$ it follows that $D_n \cdot D_n = D_1 \cdot D_1 =-2$. Therefore we can assume that the vectors $v_i$ satisfy the following relations:
\begin{equation}\label{vi}
v_{n-1}=-e_2, \quad v_{n}=-e_1, \quad v_1=e_2-2e_1 \quad \mbox{and} \quad v_2=2e_2-3e_1,
\end{equation}
where the vectors $e_1,e_2$ form the canonical basis of $\Z^2$.

Next, using the clutching construction described in Section \ref{sec:TotalSpaceToric}, we construct the manifold $M_{\Lambda_{n}}$ associated to the loop $\Lambda_{n}$ which we will denote simply by $M_{\Lambda}$ in order to simplify the notation. As we noticed in Proposition \ref{prop:TotalSpace-Toric}, $M_\Lambda$ is a toric manifold with moment map $\Phi_\Lambda$. The moment image is a 3--dimensional polytope $P_\Lambda$ with $n+2$ facets which we denote by $D_1^\Lambda, \hdots, D_n^\Lambda, D_b^\Lambda, D_t^\Lambda$ with corresponding outward primitive integral normal vectors $\eta_1, \hdots, \eta_n, \eta_b, \eta_t$. $D^\Lambda_1, \hdots, D_n^\Lambda$ are the vertical facets of $P_\Lambda$ ``coming from'' the facets of $P$, while $D_b^\Lambda$ and $D_t^\Lambda$ are respectively the bottom and top facets. Note that the vectors $\eta_1, \hdots, \eta_n$ are induced by the normal vectors $v_1, \hdots, v_n$.  More precisely, $\eta_i=(v_i,0)$ with $i=1, \hdots,n.$ It follows from \eqref{vi} together with the clutching construction that the vectors $\eta_i$ satisfy the following relations 
\begin{align*}
  \begin{array}{lllll}
      \eta_{n-1}  = -e_2  & \quad &  \eta_{3}  = \alpha_3 e_1 +\beta_3 e_2  & \quad &  \eta_b  = - e_1 - e_3   \\
      \eta_{n}  = -e_1    &  &  \hdots   &  &  \eta_t  = e_3   \\
      \eta_{1}  = e_2- 2 e_1    & &  \eta_{j}  =  \alpha_j e_1 +\beta_j e_2  &  & \\
      \eta_{2}  = 2 e_2- 3 e_1    & &   \hdots  & &
  \end{array}
\end{align*}
where now the vectors $e_1,e_2, e_3$ form the canonical basis of $\Z^3$. Clearly, it follows from the definition of  $\eta_i$, with $i=1, \hdots,n$, together with \eqref{relationvectorsI} that 
\begin{equation}\label{relationvectorsII}
 \eta_{i-1}+\eta_{i+1} = d_i \eta_i.
 \end{equation} 

\begin{exam}
Consider the second Hirzebruch surface, with a polytope with normal (outward) vectors $(0,-1), (-1,0), (-2,1), (1,0)$  where the facet normal to $(-1,0)$
 corresponds  to a curve of zero Chern number (in this example we have only one facet where the first Chern number vanishes). In this case the vectors  $\eta_i$ are the following:
 $$ \eta_1=(-2,1,0), \ \eta_2=(1,0,0), \ \eta_3 =(0,-1,0), \ \eta_4=(-1,0,0), \ \eta_b=(-1,0,-1),  \ \eta_t=(0,0,1).$$
 \end{exam}
 
The vertical facets of $P_\Lambda$ and the corresponding outward normals are represented in Figure \ref{verticalfacets}. Note that the polytope is closed, but in Figure \ref{verticalfacets}  we only draw the facets in which we are interested.
\begin{figure}[htbp]
          \psfrag{A}{$D_{n-1}^{\Lambda}$}
          \psfrag{B}{$\eta_{n-1}$}
          \psfrag{C}{\hspace{-0.2cm}$D_{n}^\Lambda$}
          \psfrag{D}{$\eta_{n}$}
          \psfrag{F}{$D_{1}^\Lambda$}
          \psfrag{G}{$\eta_{1}$}
          \psfrag{H}{\hspace{-0.1cm}$D_{2}^\Lambda$}
          \psfrag{J}{$\eta_{2}$}
           \resizebox{!}{7cm}{\includegraphics{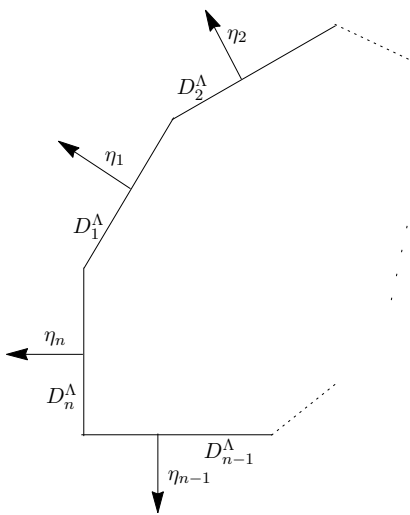}}
           \caption{Some vertical facets of the polytope $P_\Lambda$ and their outward normals.}
           \label{verticalfacets}
\end{figure}

The manifold $M_\Lambda$ is  6--dimensional, hence its fan $\Sigma$ lives in the lattice $\Z^3$. Then the 1--dimensional cones of the fan $\Sigma$  are generated by the vectors $\eta_i$ defined above. The set of primitive collections of the fan $\Sigma$ is given by the following set: 
\begin{equation*}
 \mathcal{P}= \{ \{\eta_1,\eta_3\}, \hdots, \{\eta_1,\eta_{n-1}\},  \{\eta_2,\eta_4\}, \hdots, \{\eta_2, \eta_n\}, 
  \{\eta_3,\eta_5\}, \hdots, \{\eta_3,\eta_n\},\{\eta_{n-2}, \eta_n\},\{ \eta_b,\eta_t \} \}.
\end{equation*}

From \eqref{cohomologyring} it follows that the cohomology ring of $M_\Lambda$ is given by the following isomorphism: 
$$H^*(M_\Lambda; \Q) \cong \Q[Z_1, \hdots, Z_n, Z_b, Z_t]/ \langle{\mathrm {Lin}}(\Sigma)+ {\mathrm {SR}}(\Sigma)\rangle$$
where $ {\mathrm {SR}}(\Sigma)$ is the Stanley--Riesner ideal of $\Sigma$ and ${\mathrm {Lin}}(\Sigma)$ is the ideal generated by the linear relations. The former is generated by the set of primitive collections: 
\begin{multline}\label{primitivecollections}
 Z_1Z_3, \hdots, Z_1Z_{n-1},  Z_2Z_4, \hdots, Z_2Z_n, Z_3Z_5, \hdots, Z_3Z_n,\hdots, \\
 Z_{n-3}Z_{n-1},Z_{n-3}Z_n,Z_{n-2}Z_n \ \ \mbox{and} \ \ Z_bZ_t,
\end{multline}
while the ideal ${\mathrm {Lin}}(\Sigma)$ is generated by the following three elements: 
\begin{align}
& Z_n  + 2Z_1+ 3Z_2- \alpha_3Z_3 - \hdots -\alpha_{n-2}Z_{n-2} + Z_b , \label{additive1}\\
& Z_{n-1} -Z_1-2Z_2 -\beta_3Z_3 - \hdots - \beta_{n-2}Z_{n-2} \mbox{, and}  \label{additive2}\\
& Z_t-Z_b.\label{additive3}
\end{align}
In view of the relations \eqref{additive1}--\eqref{additive3}, $Z_{n-1}$, $Z_n$ and $Z_t$ are linear combinations of the others, so that the set $\{Z_1,\ldots, Z_{n-2},Z_b\}$ is a basis of the degree 2 part of the cohomology ring. The degree 2 homology $H_2(M_\Lambda; \Z)$ can be identified with the group 
$R(\Sigma)\subset \Z^{n+2} $ given by
$$R(\Sigma) :=  \{(\gamma_1, \hdots, \gamma_n,\gamma_b, \gamma_t) \in \Z^{n+2} \, |\,  \gamma_1 \eta_1 + \hdots + \gamma_n \eta_n+\gamma_b\eta_b +\gamma_{t}\eta_{t}=0 \},$$
where we identify $\eta_b, \eta_t$ with $\eta_{n+1}, \eta_{n+2}$ respectively. If follows from the definition of the vectors $\eta_i$  that a basis for the degree 2 homology, $H_2(M_\Lambda; \Z),$ can be given by  the set $\{\lambda_1,\hdots, \lambda_{n-2}, \lambda_b\}$ which is dual to the basis of the degree 2 cohomology, that is, $Z_i (\lambda_j)= 1$ if $i=j$ and 0 otherwise.  More precisely, the generators are given by  
\begin{align*}
\lambda_1 &= \left(1, 0,\dots, 0,1,-2,0,0\right), \quad  \quad \lambda_2 = \left(0,1,0, \dots, 0,2,-3,0,0\right), \\
\lambda_j &= \left(0, \dots 0,1,0,\dots, 0,\beta_j, \alpha_j,0,0\right),  \quad j=3, \dots, n-2, \  \mbox{ and } \  \lambda_b  = \left(0,\dots, 0,-1,1,1\right), 
\end{align*}
where the entry 1 in $\lambda_j$ is located at the $j$--th entry.

From the description of the set of primitive collections, it is easy to get the set of maximal cones in $\Sigma$. Next we list some 3--dimensional cones (the ones that are going to be relevant for our computations):
\begin{align*}
  \begin{array}{lllllll}
\sigma_1 =\langle  \eta_{n-2}, \eta_{n-1}, \eta_b \rangle &  &  \sigma_4 =\langle  \eta_{1}, \eta_{2}, \eta_b \rangle  & & \sigma_7 =\langle  \eta_{n-1}, \eta_{n}, \eta_t \rangle & &  \sigma_{10}  =\langle  \eta_{2}, \eta_{3}, \eta_t \rangle  \\
\sigma_2 =\langle  \eta_{n-1}, \eta_{n}, \eta_b \rangle &  &  \sigma_5 =\langle  \eta_{2}, \eta_{3}, \eta_b \rangle  & & \sigma_8  =\langle  \eta_{1}, \eta_{n}, \eta_t \rangle   &  & \\
 \sigma_3 =\langle  \eta_{1}, \eta_{n}, \eta_b \rangle  &  & \sigma_6  =\langle  \eta_{n-2}, \eta_{n-1}, \eta_t \rangle  & &   \sigma_9 =\langle  \eta_{1}, \eta_{2}, \eta_t \rangle &  &
 \end{array}
\end{align*}

Consider now, for example, the invariant 2--sphere $V_{\sigma_2 \cap \sigma_3}$, connecting the fixed points corresponding to $\sigma_2$ and $\sigma_3$. Since 
$\sigma_2  =\langle  \eta_{n-1}, \eta_{n}, \eta_b \rangle$ and $\sigma_3  =\langle  \eta_{1}, \eta_{n}, \eta_b \rangle$, the homology class of $V_{\sigma_2 \cap \sigma_3}$ is Poincar\'e dual to $Z_nZ_b$. Hence the primitive relations yield
\begin{align*}
\langle & Z_1, V_{\sigma_2 \cap \sigma_3} \rangle   = Z_1Z_nZ_b =Z_1Z_2Z_b, \quad \langle   Z_b, V_{\sigma_2 \cap \sigma_3} \rangle  = 0, \\
\langle  & Z_2, V_{\sigma_2 \cap \sigma_3} \rangle   = 0, \quad \hdots   \quad  \langle   Z_{n-2}, V_{\sigma_2 \cap \sigma_3} \rangle  = 0. 
\end{align*}
Since $\{Z_1, \ldots, Z_{n-2},Z_b\}$ is dual to $\{\lambda_1,\hdots, \lambda_{n-2}, \lambda_b\}$, this implies that $V_{\sigma_2 \cap \sigma_3}= \lambda_1$. For another example, consider the homology class of $V_{\sigma_4\cap \sigma_5}$ which is Poincar\'e dual to $Z_2Z_b$.  Since $\eta_1 +\eta_3 = d_2 \eta_2$ (see \eqref{relationvectorsII}) it follows  that $2\alpha_3+ 3\beta_3=1$ and $\alpha_3 +2\beta_3=d_2$. 
Using \eqref{additive1} and \eqref{additive2} one obtains
\begin{align*}
\langle & Z_1, V_{\sigma_4 \cap \sigma_5} \rangle   = Z_1Z_2Z_b, \quad  \langle  Z_2, V_{\sigma_4 \cap \sigma_5} \rangle   = Z_2^2 Z_b= -d_2Z_1Z_2Z_b,   \quad \langle   Z_b, V_{\sigma_4 \cap \sigma_5} \rangle   = 0, \\
  \langle  & Z_3, V_{\sigma_4 \cap \sigma_5} \rangle   = Z_2Z_3 Z_b= Z_1Z_2Z_b,  \quad \langle   Z_4, V_{\sigma_2 \cap \sigma_3} \rangle   = 0, \ \hdots  \ \langle  Z_{n-2}, V_{\sigma_4 \cap \sigma_5} \rangle  = 0. 
\end{align*}
 
Therefore $V_{\sigma_4 \cap \sigma_5}= \lambda_1 -d_2\lambda_2 +\lambda_3$. Calculations of the homology classes of the other invariant spheres are similar.  Moreover, it is not hard to check that  the ones not identified in the diagram of Figure \ref{fig:Diagram}, all  include contributions of generators $\lambda_i$ distinct from $\lambda_1$, $\lambda_2$, and $\lambda_b$.

\begin{figure}[htbp]
          \psfrag{A}{$\sigma_6$}
          \psfrag{B}{$\sigma_1$}
          \psfrag{C}{$\sigma_7$}
          \psfrag{D}{$\sigma_2$}
          \psfrag{F}{$\sigma_8$}
          \psfrag{G}{$\sigma_3$}
          \psfrag{H}{$\sigma_9$}
          \psfrag{J}{$\sigma_4$}
           \psfrag{K}{$\sigma_{10}$}
          \psfrag{L}{$\sigma_5$}
          \psfrag{M}{$\lambda_b$}
          \psfrag{N}{$\lambda_b +\lambda_2 -2\lambda_1$}
          \psfrag{O}{$\lambda_1$}
          \psfrag{P}{$\lambda_2-2\lambda_1$}
          \psfrag{Q}{\hspace{-0.15cm}$\lambda_1-d_2\lambda_2+\lambda_3$}
           \resizebox{!}{8cm}{\includegraphics{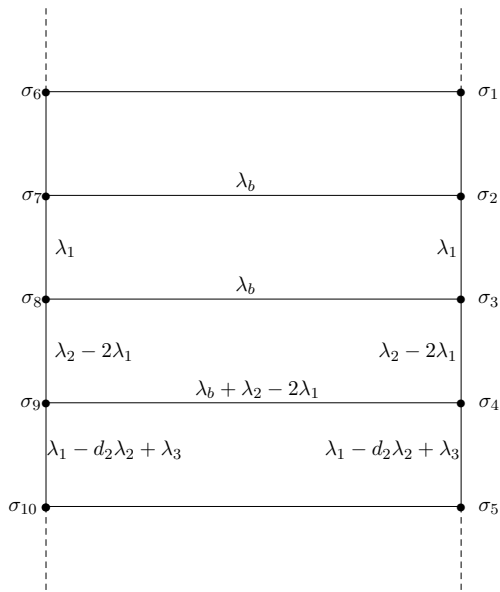}}
          \caption{Diagram representing some invariant 2--spheres of the toric manifold $M_\Lambda$ and their homology classes.}
          \label{fig:Diagram}
     \end{figure}

  Let $A_i \in H_*(M; \Z)$ with  $i=1, \hdots, n$ denote the homology class of the pre--image under the moment map $\Phi$ of the facet $D_i$. Since $M_\Lambda$ is the total space of a fibration with fiber $M$, these homology  classes can be identified with some invariant 2--spheres in $M_\Lambda$, $V_{\sigma_j \cap \sigma_k}$. More precisely, we have $A_n=\lambda_1$, $A_1=\lambda_2-2\lambda_1$. Let\footnote{The notation $A_\max$ is due to the fact that this is the homology class of a section of $M_\Lambda$ through points on the maximal fixed point component of the action (prior to the clutching construction).} $A_\max = \lambda_b=V_{\sigma_3 \cap \sigma_8}=V_{\sigma_2 \cap \sigma_7}$. Since 
$$c_1 (M_\Lambda)= Z_1+ \hdots + Z_n + Z_b+Z_t,$$
where $c_1(M_\Lambda)$ is the first Chern class of the tangent bundle of $M_\Lambda$, it follows easily that 
$ \langle c_1(M_\Lambda), \lambda_1\rangle$ $ = \langle c_1(M_\Lambda), \lambda_2\rangle=0$ and $\langle c_1(M_\Lambda), \lambda_b\rangle =1$. Therefore we have $ \langle c_1(M_\Lambda),A_n\rangle = \langle c_1(M_\Lambda), A_1\rangle=0$ and $\langle c_1(M_\Lambda), A_\max \rangle =1$.

   As we shall see in Section \ref{sec:Seidelmorphism}, in order to compute certain Gromov--Witten invariants we will need to know some more information about the ring structure of the cohomology of $M_\Lambda$, namely certain relations satisfied by the coefficients of the cup-product matrix $G=(g_{\nu\mu})_{\nu\mu}$, with $g_{\nu\mu}=\int_{M_\Lambda} e_\nu \cup e_\mu$ (for some basis $(e_{\nu})_{\nu}$ of the cohomology ring), and its inverse, $G^{-1}=(g^{\nu\mu})_{\nu\mu}$. 

By noticing that the cohomology of $M_\Lambda$ is non-zero only in even degrees, that the degree $0$ and degree $6$ groups are 1--dimensional (respectively generated by $\id$ and the fundamental class of $M_\Lambda$, $[M_\Lambda]$), and that $g_{\nu\mu}\neq 0$ only if the degrees of $e_\nu$ and $e_\mu$ sum up to 6, it is easy to see that, \emph{as soon as $(e_{\nu})_{\nu}$ is ordered so that the degree increases}, $G$ decomposes as
$$
\left( \vcenter{
\xymatrix@=1pt {    
0 & 0 \ar@{..}[rrrr] &&&& 0 & 1\\
0 \ar@{..}[dddd] & \ar@{}[rd]|*{\mathbf{0}} && \ar@{}[rrd]|*{\mathbf{B}} && & 0 \ar@{..}[dddd] \\
& &&&&& \\
 & \ar@{}[rdd]|*{\mathbf{B}^T} &&  \ar@{}[rrdd]|*{\mathbf{0}} &&& \\
&  &&&&& \\
0 &  &&&&& 0\\
1 & 0 \ar@{..}[rrrr] &&&& 0 & 0
}  
} \right)
$$
with $\mathbf{B}$ the matrix composed of the $(g_{\nu\mu})_{|e_\nu|=2,|e_\mu|=4}$.

Now, let us specify the basis.  Recall that the set $\{Z_1,\ldots, Z_{n-2}, Z_b\}$ is a basis of the degree 2 part of the cohomology. Notice that by  \eqref{primitivecollections} and \eqref{additive3} we have $Z_b^2=0$. Then the degree 4 part of the cohomology consists of all products $Z_iZ_j$ and $Z_iZ_b$ with $1\leq i\leq j\leq n-2$. In view of the relations coming from $SR(\Sigma)$, $Z_1Z_j=0$ for $3\leq j\leq Z_{n-2}$. Then, multiplying  \eqref{additive2} by $Z_1$ immediately leads to the relations $Z_1^2+2Z_1Z_2=0$. Hence, for $i=1$, only $Z_1Z_2$ and $Z_1Z_b$ need to be considered. Recall that we have  $2\alpha_3+ 3\beta_3=1$ and $\alpha_3 +2\beta_3=d_2$ as seen above. Then multiplying \eqref{additive1} and \eqref{additive2} by $Z_2$ gives $Z_2Z_3=Z_1Z_2+2Z_2Z_b$ and $Z_2^2=-d_2Z_1Z_2+(1-2d_2)Z_2Z_b$. Thus for $i=2$ we only have to consider  $Z_2Z_b$. Hence, we can explicitly write some part of $\mathbf{B}$:
\begin{align} \label{eq:table-part-of-B}  
\begin{array}{c|ccccc}
& Z_1Z_2 & Z_1Z_b &  Z_2Z_b & \dots \dots \\ \hline
Z_1 & -2 & -2 &  1 & 0 \mbox{ --- } \,0\\ 
Z_2 & 1 & 1 &  -d_2 & \\ 
Z_3 & 0 &  0&  1 & \\ 
Z_4 & 0 &  0&  0 & \\ 
| & 0 & 0 &  0 & \\ 
Z_{n-2}& 0 & 0 &  0 & \\ 
Z_b  & 1 & 0 & 0  &    
\end{array}
\end{align}
Indeed, the vanishing terms come from the relations given by the ideal ${\mathrm {SR}}(\Sigma)$, while the non-zero terms can be computed using the definition. For example, since $Z_1Z_b$ is Poincar\'e dual to $V_{\sigma_3 \cap \sigma_4}=\lambda_2 -2\lambda_1$ (see Figure \ref{fig:Diagram}), it follows that $Z_1Z_2Z_b$ is given by 
$$ \int_{M_\Lambda}Z_1Z_2Z_b = Z_2(\lambda_2 -2\lambda_1)=1.$$
Using this computation together with  the relations  given by the ideals ${\mathrm {SR}}(\Sigma)$ and ${\mathrm {Lin}}(\Sigma)$ we can  obtain the other non-vanishing terms.

In order to simplify the notation, we will denote $g_{\nu\mu}$ and $g^{\nu\mu}$ by using the indices of the corresponding elements $e_\nu$ and $e_\mu$. For example, for $e_\nu=Z_1$ and $e_\mu=Z_2Z_b$,  $g_{\nu\mu}$ will be denoted $g_{1,2b}$ and $g^{\nu\mu}$ will be denoted $g^{1,2b}$. Of course $G$ and $G^{-1}$ are symmetric so that $g_{\nu\mu}=g_{\mu\nu}$ and $g^{\nu\mu}=g^{\mu\nu}$. Moreover, note that by commutativity of the cup-product, permuting the indices does not change the value $g_{1,2b}=g_{b,12}=g_{2,1b}$. However, this fails for the coefficients of $G^{-1}$.

Since $G^{-1}G=\id$, we get relations between the coefficients of $G$ and $G^{-1}$ by multiplying particular lines of $G^{-1}$ with columns of $G$. For example,
\begin{align*}
\sum_\nu g^{1b,\nu}g_{\nu,1b}=1 &\Longleftrightarrow   -2g^{1b,1}+ g^{1b,2}=1 \\
\sum_\nu g^{1b,\nu}g_{\nu,12}=0 &\Longleftrightarrow   -2g^{1b,1}+g^{1b,2}+g^{1b,b}=0 \\
\sum_\nu g^{1b,\nu}g_{\nu,2b}=0 &\Longleftrightarrow   g^{1b,1}-d_2g^{1b,2}+g^{1b,3}=0 
\end{align*}
which lead to the fact that $g^{1b,b}=-1$. By using the lines of $G^{-1}$ corresponding to $Z_1Z_2$, $Z_2Z_b$, and again the columns of $G$ corresponding to $Z_1Z_2$, $Z_1Z_b$, and $Z_2Z_b$, we get some more relations between the coefficients of the matrix $G^{-1}$.  We gather in the next lemma the result of these computations. 
\begin{lemm}[Some coefficients of $G^{-1}$]
\label{lem:relationscoeffmatrix}
\begin{align*} 
& \left\{\! \begin{array}{l}  g^{1b,b}=-1 \\ g^{1b,2}-2g^{1b,1}=1 \\   g^{1b,1}-d_2g^{1b,2}+g^{1b,3}=0 \end{array}  \!\right., 
\left\{\! \begin{array}{l}  g^{12,b}=1 \\ g^{12,2}=2g^{12,1} \\   g^{12,1}-d_2g^{12,2}+g^{12,3}=0 \end{array}  \!\right.,  \; 
& \left\{\! \begin{array}{l}  g^{2b,b}=0 \\ g^{2b,2}=2g^{2b,1} \\   g^{2b,1}-d_2g^{2b,2}+g^{2b,3}=1  \end{array} \! \right..
\end{align*}
\end{lemm}

\subsection{Gromov--Witten invariants}
\label{mainexample}

We now compute some Gromov--Witten invariants of $M_\Lambda$ using Spielberg's machinery from \cite{Spielberg00}. In particular we will use a simplified version of its main theorem which we give in Section \ref{SpielbergFormula}. 

We  need to know the weights of the torus action at the different charts. By general theory each 3--dimensional cone  gives a chart of the toric manifold near a fixed point. 
For  our calculations  it will be convenient to know the following weights, which we compute using Lemma \ref{weights}. 
\begin{center}
\begin{tabular}{|c|c|}\hline
  \quad        \quad $\sigma_2  =\langle  \eta_{n-1}, \eta_{n}, \eta_b \rangle$  \quad  \quad & \quad  \quad$\sigma_7 =\langle  \eta_{n-1}, \eta_{n}, \eta_t \rangle$ \quad  \quad
 \\\hline
 $\omega_{\sigma_3}^{\sigma_2}=   a_1$ & $\omega_{\sigma_8}^{\sigma_7}=   a_1$ \\
   $\omega_{\sigma_1}^{\sigma_2}=   a_2 + \omega_t$ & $\omega_{\sigma_6}^{\sigma_7}=   a_2+ \omega_b$ \\
  $\omega_{\sigma_7}^{\sigma_2}=   \omega_b- \omega_t$ & $\omega_{\sigma_2}^{\sigma_7}=   \omega_t- \omega_b$       
  \\ \hline \hline
   \quad        \quad $\sigma_3  =\langle  \eta_{1}, \eta_{n}, \eta_b \rangle$  \quad  \quad & \quad  \quad$\sigma_8  =\langle  \eta_{1}, \eta_{n}, \eta_t \rangle$ \quad  \quad
            
\\\hline

$\omega_{\sigma_2}^{\sigma_3}=   -a_1$ & $\omega_{\sigma_7}^{\sigma_8}=   -a_1$ \\
   $\omega_{\sigma_4}^{\sigma_3}=   2a_1+a_2 + \omega_t$ & $\omega_{\sigma_9}^{\sigma_8}=   2a_1+a_2+ \omega_b$ \\
  $\omega_{\sigma_8}^{\sigma_3}=   \omega_b- \omega_t$ & $\omega_{\sigma_3}^{\sigma_8}=   \omega_t- \omega_b$       
  \\ \hline \hline
 \quad        \quad $\sigma_4  =\langle  \eta_{1}, \eta_{2}, \eta_b \rangle$  \quad  \quad & \quad  \quad$\sigma_9 =\langle  \eta_{1}, \eta_{2}, \eta_t \rangle$ \quad  \quad
\\\hline

$\omega_{\sigma_3}^{\sigma_4}=   -2a_1-a_2 - \omega_t$ & $\omega_{\sigma_8}^{\sigma_9}=   -2a_1-a_2 -\omega_b$ \\
   $\omega_{\sigma_5}^{\sigma_4}=   3a_1+2a_2 + 2\omega_t$ & $\omega_{\sigma_{10}}^{\sigma_9}=   3a_1+2a_2+ 2\omega_b$ \\
  $\omega_{\sigma_9}^{\sigma_4}=   \omega_b- \omega_t$ & $\omega_{\sigma_4}^{\sigma_9}=   \omega_t- \omega_b$
 \\\hline
 \end{tabular}
\end{center}
where the $a_1,a_2 \in \Z$ are linear functions on the weights $\omega_1, \hdots, \omega_n$. 
Now we are ready to begin calculating Gromov--Witten invariants of this manifold. In the next lemma we will compute some invariants which will be needed later in the proof of Theorem \ref{GWinvariants}. 
\begin{lemm}[Gromov--Witten invariants]\label{lemm:GW_Z1Zb-ZnZb}
\begin{gather*}
\GW_{A_\max+A_n,1}^{M_\Lambda}(Z_iZ_j)=\GW_{A_\max+A_n+A_1,1}^{M_\Lambda}(Z_iZ_j)= \left\{ \begin{array}{cl}
1 & \mbox{if} \quad  i=1, j=b \\
0 & \mbox{if} \quad  i=1, j=2 \\
0 & \mbox{if} \quad i=2, j=b 
 \end{array} \right.
\ \ \mbox{and } \\
\GW_{A_\max+A_1,1}^{M_\Lambda}(Z_iZ_j)= \left\{ \begin{array}{cl}
2 & \mbox{if}  \quad i=1, j=b \\
2 & \mbox{if} \quad i=1, j=2 \\
-1 & \mbox{if} \quad i=2, j=b 
 \end{array} \right.
\end{gather*}
\end{lemm}
\begin{proof}
We first compute  the invariant 
$\GW_{A_\max+A_n+A_1,1}^{M_\Lambda}(Z_1Z_b)$.
We use the formula from Section \ref{SpielbergFormula}. Since the marked point has to lie in the cone $\sigma_3$ or $\sigma_4$, we need to consider the graphs which contain one of these cones and which represent the class $A_\max+A_n+A_1$.  It follows that we should consider the following graphs: 
\begin{figure}[htbp]
          \psfrag{A}{\tiny$\sigma_4$}
          \psfrag{B}{\tiny$\sigma_3$}
          \psfrag{C}{\tiny$\sigma_2$}
          \psfrag{D}{\tiny$\sigma_7$}
          \psfrag{G}{\tiny$\sigma_8$}
          \psfrag{F}{\tiny$\sigma_9$}
          \psfrag{M}{\tiny$(2)$}
          \psfrag{N}{\tiny$(1)$}
          \psfrag{O}{\tiny$(3)$}
          \psfrag{P}{\tiny$(4)$}
           \psfrag{Q}{\tiny$(5)$}
           \resizebox{!}{3.5cm}{\includegraphics{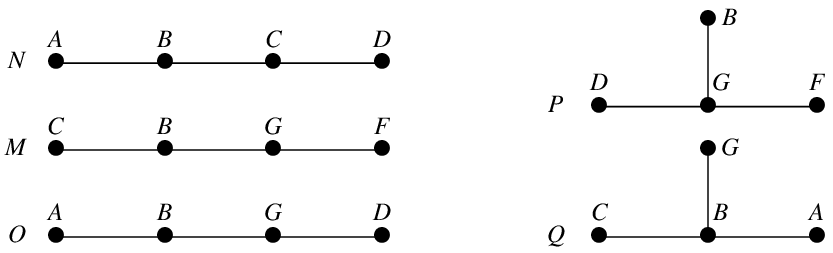}}
          \label{Graphs}
     \end{figure}

 Therefore Theorem \ref{mainSpielberg} gives the following computation
 \begin{align*} 
  &\GW_{ A_\max+A_n+A_1,1}^{M_\Lambda} (Z_1Z_b)  = (1)+\hdots+(5)  = - \frac{(a_1+a_2+\omega_t)(a_2+\omega_t)}{a_1(a_1+\omega_b-\omega_t)} \\
  &  \qquad\qquad\;\;\;\; -  \frac{(a_1+a_2+\omega_t)(a_1+a_2+\omega_b)}{(\omega_b-\omega_t)^2} - \frac{(a_1+2(\omega_b-\omega_t))(a_1+a_2+\omega_t)(a_1+a_2+\omega_b)}{(\omega_b-\omega_t)^2(a_1+\omega_b-\omega_t)} \\  
& \qquad\qquad\;\;\;\; +  \frac{(a_1+a_2+\omega_b)^2}{(\omega_b-\omega_t)^2}  +  \frac{(a_1+a_2+\omega_t)^2(a_1+\omega_b-\omega_t)}{(\omega_b-\omega_t)^2}= 1 \,.
\end{align*}
We can compute the invariant $$ \GW_{ A_\max+A_n+A_1,1}^{M_\Lambda} (Z_1Z_2)$$ in a similar way. In this case the marked point lies in the cone $\sigma_4$ or $\sigma_9$ so we need to consider the same graphs as in the computation above plus the following graph: 
\begin{figure}[htbp]
          \psfrag{A}{\small$\sigma_2$}
          \psfrag{B}{\small$\sigma_7$}
          \psfrag{C}{\small$\sigma_8$}
          \psfrag{D}{\small$\sigma_9$}
          \psfrag{N}{$(6)$}
           \resizebox{!}{.7cm}{\includegraphics{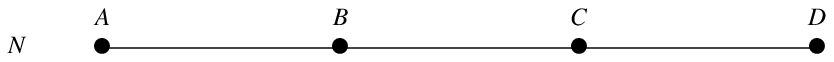}}
          \label{Graph}
     \end{figure}
     
The formula now gives for $ \GW_{ A_\max+A_n+A_1,1}^{M_\Lambda} (Z_1Z_2)   = (1)+ \hdots +(6)$:
\begin{align*} 
  &\GW_{ A_\max+A_n+A_1,1}^{M_\Lambda} (Z_nZ_b)  = \frac{(3a_1+2a_2+2\omega_t)(a_2+\omega_t)(a_1 +a_2+\omega_t)}{a_1(\omega_b-\omega_t)(a_1+\omega_b-\omega_t)} \\
 & \quad + \frac{(3a_1+2a_2+2\omega_b)(a_1+a_2+\omega_t)(a_1+a_2+\omega_b)}{(\omega_b-\omega_t)^2(a_1+\omega_t-\omega_b)} 
 \end{align*}
  \begin{align*} 
& \quad + \frac{(3a_1+2a_2+2\omega_t)(a_1+a_2+\omega_t)(a_1+a_2+\omega_b)}{(\omega_b-\omega_t)^2(a_1+\omega_b-\omega_t)}  -\frac{(3a_1+2a_2+2\omega_b)(a_1+a_2+\omega_b)^2}{a_1(\omega_b-\omega_t)^2} \\ 
& \quad - \frac{(3a_1+2a_2+2\omega_t)(a_1+a_2+\omega_t)^2}{a_1(\omega_b-\omega_t)^2} +\frac{(3a_1+2a_2+2\omega_b)(a_1+a_2+\omega_b)(a_2+\omega_b)}{a_1(\omega_t-\omega_b)(a_1+\omega_t-\omega_b)} =0.
\end{align*}
The remaining invariants can be computed  using the same formula, therefore we leave their computation for the interested reader. 
\end{proof}

% ***********************************************************************************************
%
%                                         Seidel's morphism of blow-ups
%
% ***********************************************************************************************

\section{Seidel morphism in the NEF case}
\label{sec:seidel-morphism-nef}

In this section we explain how to compute the Seidel element associated to a Hamiltonian circle action fixing a facet of a toric 4--dimensional NEF symplectic manifold.

\subsection{The Seidel morphism}\label{sec:Seidelmorphism}

% Recall from Section \ref{sec:TotalSpaceToric} that, starting from any closed symplectic manifold $(M,\omega)$ and a loop of Hamiltonian diffeomorphisms $\Lambda \subset \ham(M,\omega)$, one can construct a Hamiltonian fibration over $S^2$ with fiber $M$:
% $$ (M,\omega) \hookrightarrow (M_\Lambda,\omega_\Lambda) \rightarrow (S^2,\omega_0) $$
Recall from Section \ref{sec:TotalSpaceToric} that, starting from any closed symplectic manifold $(M,\omega)$ and a loop of Hamiltonian diffeomorphisms $\Lambda \subset \ham(M,\omega)$, one can construct a Hamiltonian fibration $\pi \co (M_\Lambda,\omega_\Lambda) \rightarrow (S^2,\omega_0)$ with fiber $(M,\omega)$, where $\omega_\Lambda=\Omega+\kappa \cdot \pi^*(\omega_0)$ for some big enough $\kappa$. Then, following \cite{Seidel97}, one can define Seidel's morphism, under some appropriate semi-positivity assumption on $(M,\omega)$, by counting pseudo-holomorphic section classes in $H^S_2(M_\Lambda;\bb Z)$, with respect to some arbitrary choice of such a section. This choice was made canonical in \cite{LMcDP99}.

In view of our goal, we now focus on the following specific case:
\begin{enumerate}[(i)]
\item The manifold $M$ admits an almost complex structure $J$ so that $(M,J)$ is NEF (that is, there are no $J$--pseudo-holomorphic spheres with $\langle c_1(M),B \rangle < 0$).
\item The symplectic manifold $(M,\omega)$ is a toric 4--dimensional manifold, whose associated Delzant polytope has $n \geq 4$ facets.
\item $\Lambda$ is a circle action, with moment map $\Phi_\Lambda$, whose maximal fixed point component corresponds to a divisor, denoted by $F_\max$.
\end{enumerate}
\begin{nota}
  Since the first Chern class of $M$ (and of $M$ only) is extensively used in what follows, we will denote $c_1(M)$ by $c_1$ and $\langle c_1(M),B \rangle$ by $c_1(B)$.
\end{nota}

We now extract from \cite{McDuffTolman06} the results which will be used in this section. Notice that in our specific setting, $F_\max$ is semifree and has dimension 2. We denote by $\Phi_\max=\Phi_\Lambda(F_\max)$ the maximal value of the moment map. Concerning the choice of the section mentioned above, recall that in the toric case it is convenient to choose $\sigma_\max = \{x\} \times D_1 \cup_\Lambda \{x\} \times D_2$ (see the description of the clutching construction, Section \ref{sec:TotalSpaceToric}) for any fixed point of the $S^1$--action $x$ lying in $F_\max$. If we let  $A_\max= [\sigma_\max] \in H_2^S(M;\bb Z)$ then all the contributions to the Seidel morphism come from the section classes $A_\max + B$ with $B\in H_2^S(M;\bb Z)$ and are determined by counting Gromov--Witten invariants in the classes $A_\max + B$, see e.g \cite[Definition 2.4]{McDuffTolman06}. Lastly, by  \cite[Lemma 2.2]{McDuffTolman06} the sum of the weights which appear in the formula giving the Seidel morphism, as part of the exponent of the $q$ variable, is $m_\max=-1$.

\begin{theo}[Theorem 1.10 and Lemma 3.10 of \cite{McDuffTolman06}]\label{theo:MainMcDT}
Under the assumptions \emph{(i)--(iii)} above, the Seidel element associated to the circle action $\Lambda$ is 
\begin{align*}
  S(\Lambda) = [F_\max] \otimes q t^{\Phi_\max} + \sum_{B \in H_2^S\!(M;\bb Z)^{>0}} a_B \otimes q^{1-c_1(B)} t^{\Phi_\max-\omega(B)} 
\end{align*}
where $H_2^S(M;\bb Z)^{>0}$ consists of the spherical classes of symplectic area $\omega(B)>0$ and $a_B\in H_*(M;\bb Z)$ is the contribution of the section class $A_\max+B$ defined by requiring that $ a_B \cdot_M c = \GW^{M_\Lambda}_{A_\max + B , 1}(c)$ for all homology classes $c\in H_*(M ;\bb Z)$.
Moreover,
\begin{enumerate}[(i)]
\item If $a_B\neq 0$ either $c_1(B)=0$ and $a_B \in H_2(M;\bb Z)$ or $c_1(B)=1$ and $a_B\in H_4(M;\bb Z)$. 
\item If $a_B\neq 0$ then $B$ intersects $F_\max$.
\item If $c_1(B')\geq 1$ for all $J$--holomorphic spheres $B'$ which intersect $F_\max$, then all the lower order terms vanish.
\item If $c_1(B')\geq 1$ for all $J$--holomorphic spheres $B'$ which intersect $F_\max$ but are not included in $F_\max$, then $a_B\neq 0\Rightarrow c_1(B)=0$.
\end{enumerate}
\end{theo}

\begin{rema}\label{rem:On-Main-McDT-Theo}
Item \textit{(i)} above reads: If $a_B\neq 0$ then $c_1(B)=0$ and $|a_B|=2$. Indeed, when $M$ is 4--dimensional, $|a_B|=4$ means that $a_B$ has to be a multiple of the fundamental class $[M]$, however this case can easily be ruled out. (See for example the end of the proof of \cite[Theorem 1.10]{McDuffTolman06}.)

Item \textit{(ii)} is  \cite[Lemma 3.10]{McDuffTolman06} and shows that, even though the formula above might contain infinitely many terms, computing the Seidel morphism is somehow ``local'' (that is, one does not need to know the whole polytope).
\end{rema}

Recall the notation we introduced in Section \ref{sec:example}: We consider the case when the polytope $P$, associated to $M$, admits $n \geq 4$ facets, $D_1, \ldots D_n$. These facets correspond to divisors whose homology classes we respectively denote by $A_1, \ldots A_n$. We put $A_n = [F_\max]$ and we see the indices mod $n$. For any $n$--tuple $\bar{a}=(a_1,\dots,a_n) \in \Z^n$, we denote by $A_{\bar a}=\sum_{i} a_i A_i$ the homology class of the union of (possibly multiply covered) spheres in $M$ whose projection to $P$ is given by $D_{\bar a}=\cup_{i} D_i$. 

Thus Theorem \ref{theo:MainMcDT}, combined with Remark \ref{rem:On-Main-McDT-Theo}, implies that the Seidel element is given by
\begin{align*} 
  S(\Lambda) = A_n\otimes q t^{\Phi_\max} + \sum_{{\bar a}} a_{A_{\bar a}} \otimes q t^{\Phi_\max-\omega({A_{\bar a}})} 
\end{align*}
where $a_{A_{\bar a}}\neq 0$ if and only if 
\begin{enumerate}
\item $D_{\bar a}$ is connected and intersects $D_{n}$,
\item $c_1(A_{\bar a})=0$ (i.e, by NEF condition, for all $i$ so that $a_i\neq 0$, $c_1(A_i)=0$).
\end{enumerate}

In Theorem \ref{theo:contributions} below, we compute each contribution $a_{A_{\bar a}}$ in the case of polytopes where any $D_{\bar a}$ satisfying (1) and (2) contains at most two facets corresponding to spheres with vanishing first Chern number. Notice that in case the facets corresponding to divisors with vanishing first Chern number are not $D_n$ and/or $D_1$ (that is, Cases \textit{(3b)} and \textit{(3c)}), the content of Section \ref{sec:toric-topol-data} has to be slightly adapted.

\begin{theo}\label{theo:contributions}
Let $(M, \omega)$ be a closed NEF toric 4--dimensional symplectic manifold. Assume that its associated Delzant polytope has $n \geq 4$ facets. Let $\Lambda$ be a circle action, whose maximal fixed point component is a divisor $F_\max$ and denote $A_n=[F_\max]$ its homology class. The following homology classes have non trivial contributions to $S(\Lambda)$, the Seidel element associated to $\Lambda$:
  \begin{enumerate}
  \item $A_n$ contributes by $a_{A_n}=A_n$.
  \item If $c_1(A_n)=0$, 
    \begin{enumerate}[(2a)] 
    \item then $kA_n$ (with $k>0$) contributes by $a_{kA_n}=A_n$,
   \item and if $c_1(A_1)=0$, then $kA_n+lA_1$ (with $k\geq 0$ and $l>0$) contributes and its contribution is $a_{kA_n+lA_1}= \left\{ \begin{array}{cc} A_n &\mbox{if } k\geq l, \\ -A_1 &\mbox{otherwise.}  \end{array} \right.$
    \end{enumerate}
  \item If $c_1(A_n)\neq 0$,
    \begin{enumerate}[(3a)]
    \item if $c_1(A_1)=0$, then $kA_1$ (with $k>0$) contributes by $a_{kA_1}=-A_1$,
    \item if $c_1(A_1)=0$ and $c_1(A_2)=0$, then $kA_1+lA_2$ (with $k>0$ and $l>0$) also contributes, and its contribution is $$a_{kA_1+lA_2}= \left\{ \begin{array}{cc} -A_1 &\mbox{if } k\geq l, \\ A_2 &\mbox{otherwise.}  \end{array} \right.$$
    \item if $c_1(A_{n-1})=0$ and $c_1(A_1)=0$, then $kA_{n-1}$ and $lA_1$ (with $k>0$ and $l>0$) also contribute, with respective contributions $a_{kA_{n-1}}= -A_{n-1}$ and $a_{lA_1}= -A_1$.
    \end{enumerate}
  \end{enumerate}
Moreover, in each case, if the facets immediately next to the ones mentioned above correspond to spheres with non-zero first Chern number, then these are the only non-trivial contributions.
\end{theo}

As a corollary, we compute the Seidel element associated to $\Lambda$ in these different cases. (See also Figure \ref{fig:cases} in the introduction.) Recall that we also compute in Appendix \ref{sec:BrutalComputationsSeidSMorph} the Seidel element associated to $\Lambda$ when there exist three divisors in the vicinity of $A_n$ with vanishing first Chern number.

\begin{theo}\label{theo:SeidelsMorphism}
Under the assumptions and with the notation of Theorem \ref{theo:contributions} above, the Seidel element associated to $\Lambda$ is as follows.
  \begin{enumerate}
  \item\label{list:no-0} If $c_1(A_n)$, $c_1(A_{n-1})$ and $c_1(A_1)$ are all non-zero, then $S(\Lambda)=A_n\otimes qt^{\Phi_\max}$.
  \item If $c_1(A_n)=0$, 
   \begin{enumerate}[(2a)] 
   \item\label{list:1-0-A} but  $c_1(A_{n-1})$ and $c_1(A_1)$ are non-zero, then $$S(\Lambda)=A_n\otimes q\,\frac{t^{\Phi_\max}}{1-t^{-\omega(A_n)}} \,,$$
   \item \label{list:2-0-A-Cons} and  $c_1(A_1)=0$ but $c_1(A_{n-1})$ and $c_1(A_2)$ non-zero, then 
$$ S(\Lambda)=\left[ A_n\otimes q\,\frac{t^{\Phi_\max}}{1-t^{-\omega(A_n)}}-A_1\otimes q\,\frac{t^{\Phi_\max-\omega(A_1)}}{1-t^{-\omega(A_1)}} \right]\cdot \frac{1}{1-t^{-\omega(A_n)-\omega(A_1)}} \,.
$$
    \end{enumerate}
  \item If $c_1(A_n)\neq 0$, 
   \begin{enumerate}[(3a)]
    \item\label{list:1-0-notA} if $c_1(A_1)=0$ and $c_1(A_{n-1}), c_1(A_2)$ non-zero, then  $$S(\Lambda)=A_n\otimes qt^{\Phi_\max}-A_1\otimes q\,\frac{t^{\Phi_\max-\omega(A_1)}}{1-t^{-\omega(A_1)}} \,,$$
    \item\label{list:2-0-notA-Cons} if $c_1(A_1)=c_1(A_2)=0$ but $c_1(A_{n-1})$ and $c_1(A_{3})$ non-zero, then 
      \begin{multline*}
  S(\Lambda)  =  A_n\otimes qt^{\Phi_\max}-A_1\otimes q\,\frac{t^{\Phi_\max-\omega(A_1)}}{1-t^{-\omega(A_1)}}  \\
\qquad -\left(A_1\otimes q\,\frac{t^{\Phi_\max}}{1-t^{-\omega(A_1)}}-A_2\otimes q\,\frac{t^{\Phi_\max-\omega(A_2)}}{1-t^{-\omega(A_2)}}\right)\cdot 
\frac{t^{-\omega(A_1)-\omega(A_2)}}{1-t^{-\omega(A_1)-\omega(A_2)}} \,,
      \end{multline*}
     \item\label{list:2-0-notA-notCons} if $c_1(A_{n-1})=c_1(A_1)=0$, $c_1(A_{n-2})$ and $c_1(A_2)$ non-zero, then 
$$S(\Lambda)=A_n\otimes qt^{\Phi_\max}-A_{n-1}\otimes q\,\frac{t^{\Phi_\max-\omega(A_{n-1})}}{1-t^{-\omega(A_{n-1})}}-A_1\otimes q\,\frac{t^{\Phi_\max-\omega(A_1)}}{1-t^{-\omega(A_1)}} \,.$$
    \end{enumerate}
  \end{enumerate}
\end{theo}

We start by deducing Theorem \ref{theo:SeidelsMorphism} from Theorem \ref{theo:contributions}. The proof of the latter is postponed to the next subsection since it is much more involving.

\begin{proof}[Proof of Theorem \ref{theo:SeidelsMorphism}]
It is a staigthforward consequence of Theorems \ref{theo:MainMcDT} and \ref{theo:contributions}.  
\begin{itemize}
\item[(\ref{list:no-0}):]  By Theorem \ref{theo:contributions}, only $A_n$ contributes and its contribution is of the form $S(\Lambda)=A_n\otimes qt^{\Phi_\max}$.
\item[(\ref{list:1-0-A}):] Here $A_n$ and its iterations induce the only non-trivial contributions. The contribution of $kA_n$ being $A_n\otimes q^{-c_1(kA_n)} t^{\Phi_\max-\omega(kA_n)}$, we get the result by summing over $k$ (starting at $k=0$):
  \begin{align*}
    S(\Lambda) = A_n\otimes q \, t^{\Phi_\max} \left( \sum_{k=0}^\infty (t^{-\omega(A_n)})^k \right) = A_n\otimes q\,\frac{t^{\Phi_\max}}{1-t^{-\omega(A_n)}}
  \end{align*}
\item[(\ref{list:1-0-notA}):]  This case is similar to (\ref{list:1-0-A}) except that we sum the contributions of all the $kA_1$'s starting at $k=1$ (thus, the new $-\omega(A_1)$ as power of $t$).
\item[(\ref{list:2-0-notA-notCons}):]  This case is similar to (\ref{list:1-0-notA}) (but for both $A_{n-1}$ and $A_1$).
\end{itemize}

Now we turn to (\ref{list:2-0-notA-Cons}). The first two terms coincide with the sum of the contributions induced by $A_n$ and $kA_1$. However, we also have to count the contributions of $kA_1+lA_2$. As before, we can see that
\begin{multline*}
   -A_1\otimes q\,\frac{t^{\Phi_\max-\omega(A_1)-\omega(A_2)}}{(1-t^{-\omega(A_1)})(1-t^{-\omega(A_1)-\omega(A_2)})} \\ 
= \sum_{k=1, l=0}^{\infty} a_{k(A_1+A_2)+lA_1} \otimes q t^{\Phi_\max-(k+l)\omega(A_1)-k\omega(A_2)}
\end{multline*}
which sums the contributions of $k(A_1+A_2)+lA_1$ (with $k\geq 1$ and $l\geq 0$), that is, the contributions of all terms of the form $kA_1+lA_2$ with $k\geq l \geq 1$. In the same way,
\begin{multline*}
  A_2\otimes q\,\frac{t^{\Phi_\max-2\omega(A_2)-\omega(A_1)}}{(1-t^{-\omega(A_2)})(1-t^{-\omega(A_1)-\omega(A_2)})} % \\
= \sum_{k,l=1}^{\infty} a_{k(A_1+A_2)+lA_2} \otimes q t^{\Phi_\max-k\omega(A_1)-(k+l)\omega(A_2)}
\end{multline*}
which sums the contributions of all terms of the form $kA_1+lA_2$ with $k< l$. Thus the formula given for the case (\ref{list:2-0-notA-Cons}) is indeed the sum of all non-trivial contributions.

Finally, let us look at (\ref{list:2-0-A-Cons}). First decompose 
\begin{align*}
  \frac{1}{1-t^{-\omega(A_n)-\omega(A_1)}} = 1+ \frac{t^{-\omega(A_n)-\omega(A_1)}}{1-t^{-\omega(A_n)-\omega(A_1)}}
\end{align*}
and by replacing, we check that
\begin{align*}
  S(\Lambda) &= \left[ A_n\otimes q\,\frac{t^{\Phi_\max}}{1-t^{-\omega(A_n)}}-A_1\otimes q\,\frac{t^{\Phi_\max-\omega(A_1)}}{1-t^{-\omega(A_1)}} \right]\cdot \frac{1}{1-t^{-\omega(A_n)-\omega(A_1)}}\\
&= A_n\otimes q\,\frac{t^{\Phi_\max}}{1-t^{-\omega(A_n)}}-A_1\otimes q\,\frac{t^{\Phi_\max-\omega(A_1)}}{1-t^{-\omega(A_1)}} \\ & \quad+ A_n\otimes q\,\frac{t^{\Phi_\max-\omega(A_n)-\omega(A_1)}}{(1-t^{-\omega(A_n)})(1-t^{-\omega(A_n)-\omega(A_1)})} \\ 
& \quad -A_1\otimes q\,\frac{t^{\Phi_\max-2\omega(A_1)-\omega(A_n)}}{(1-t^{-\omega(A_1)})(1-t^{-\omega(A_n)-\omega(A_1)})}
\end{align*}
Now the first term counts the contributions of all terms of the form $kA_n$ (as in (\ref{list:1-0-A}) above), the second term counts the contributions of $kA_1$ (or $A_n+kA_1$, see above) and then the last two count (as for (\ref{list:2-0-notA-Cons}) but with $A_n$ playing the role of $A_1$ and $A_1$ playing the role of $A_2$) all the contributions of the terms of the form $kA_n+lA_1$ (with $k$ and $l$ both non-zero).
\end{proof}

\subsection{Proof of Theorem \ref{theo:contributions}}
\label{sec:proof-theor-refth}

The proof is more or less a case-by-case proof and we focus on Case \textit{(\ref{list:2-0-A-Cons})}, since all the difficulties which one might encounter are already present and since the methods used to compute the Gromov--Witten invariants are the same.
Notice that Case \textit{(2b)} is one of the spectific cases described in Section \ref{sec:example}.

 We need to determine the class $a_B$ of Theorem \ref{theo:MainMcDT} where $B=\akl \in H_2(M;\bb Z)$. Recall that this class is determined by the requirement that 
$$ a_B \cdot c = \GW_{A_\max+B,1}^{M_\Lambda} (c),   \quad \mbox{for all } c \in H_*(M;\Z).$$
  In the notation for the Gromov--Witten invariant we can either use the homology class  $c$ or its Poincar\'e dual.  Let us define $B_{k,l}:= A_\max+ \akl$.
Now we claim that in order to prove  the theorem in Case \textit{(\ref{list:2-0-A-Cons})} it is sufficient to compute the following Gromov--Witten invariants.
\begin{theo}\label{GWinvariants}
  For any $k,l \in \N$ we have 
\begin{gather*}
\GW_{B_{k,l},1}^{M_\Lambda}(Z_1Z_2)= \left\{ \begin{array}{cl}
0 &\mbox{if } k \geq l \\
2 &\mbox{if } k < l \end{array} \right., \quad 
\GW_{B_{k,l},1}^{M_\Lambda}(Z_1Z_b)= \left\{ \begin{array}{cl}
1 &\mbox{if } k \geq l \\
2 &\mbox{if } k < l \end{array} \right. 
\ \ \mbox{and } \\
\GW_{B_{k,l},1}^{M_\Lambda}(Z_2Z_b)= \left\{ \begin{array}{cl}
0 &\mbox{if } k \geq l \\
-1 &\mbox{if } k < l \end{array} \right.
\end{gather*}
where $Z_1,Z_2,Z_b \in H^2(M_\Lambda; \Q)$ are defined in Section \ref{sec:toric-topol-data}.
\end{theo}
Since the proof of this theorem is quite long and technical, we postpone it to Sections \ref{sec:proof-prop-nopointsinv} and \ref{sec:proof-theor-GWinv}, and we first finish the proof of Theorem \ref{theo:contributions} by proving the claim.

The class $a_B$ is a linear combination of the homology classes of the pre-images, under the moment map $\Phi$ of the facets of the polyope $P=\Phi(M)$, that is, 
\begin{equation}\label{a_B}
a_B= \sum_{i=1}^{n}a_i A_i, 
\end{equation}
where $a_i \in \Z$.  
Since the dimension of the $\Z$--module $H^S_2(M;\bb Z)$ is $n-2$, we  can assume that two of the coefficients $a_i$ vanish. The following lemma shows that we can choose the coefficients $a_2=a_3=0$.
\begin{lemm}
  All the classes $A_{i}$ are linear combinations of the basis elements $\{\lambda_1,  \hdots, $ $ \lambda_{n-2}\}$, defined in Section \ref{sec:example}.
\end{lemm}
\begin{proof}
 It is known from the diagram of Figure \ref{fig:Diagram}  that 
$A_n=\lambda_1$  and $A_1=\lambda_2-2 \lambda_1$ which gives $\lambda_1 =A_n$ and $\lambda_2=2A_n+A_1$. Recall that $\eta_i= \alpha_i e_1+ \beta_i e_2$ where $i=1, \dots, n$. Let $\gamma_{i,j}:= \alpha_j\beta_i- \alpha_i\beta_j$. It is not hard to check that Relation \eqref{relationvectorsII} implies that $\gamma_{i,i+1}=1$. Moreover $\gamma_{i,j} \neq 0$ if $j \neq i+1$ because the polytope is convex.  We can write   all the $A_i's$  as linear combinations of the basis elements $\lambda_i$, using the same argument as we use in  Section \ref{sec:example} for $A_n$ and $A_1$, which yields: 
\begin{align*}
& A_{n-1}=\lambda_{n-2}, && A_4=\lambda_3 + \gamma_{5,3}\,\lambda_4 +\lambda_5,\\
& A_{n-2}=\lambda_{n-3}+\gamma_{n-1,n-3}\,\lambda_{n-2}, & \cdots\qquad & A_3= \lambda_2 +\gamma_{4,2}\,\lambda_3 +\lambda_4,   \\
& A_{n-3}=\lambda_{n-4} + \gamma_{n-2,n-4}\,\lambda_{n-3} + \lambda_{n-2}, && A_2= \lambda_1 -d_2 \, \lambda_2 + \lambda_3. 
\end{align*} 
Since $\lambda_{n-2}=A_{n-1}$ it follows from the second equation that $\lambda_{n-3}= A_{n-2} -\gamma_{n-1,n-3}A_{n-1}.$ Substituting this in the third  equation we can find an expression of  
$\lambda_{n-4}$ as a linear combination of $A_{n-2}$ and $A_{n-1}$. Going around the polytope we easily see that we can, recursively, determine an expression of each $\lambda_i$  as a linear combination of the $A_i's$  with $i\neq 2,3$. 
In particular, we obtain expressions for $\lambda_3$ and $\lambda_4$ which implies, by the last two equations, that $A_2$ and $A_3$ are linear combinations of the remaining $A_i's$.
\end{proof}

 Therefore, from now on, we assume $a_2=a_3=0$ in the linear combination \eqref{a_B}.  Recall that 
$$ a_B \cdot c = \GW_{A_\max+B,1}^{M_\Lambda} (\PD ( c))$$ for $c \in H_2(M;\bb Z)$. If $c$ does not contain $A_{n-1}, A_n,A_1, A_2$ then clearly the Gromov--Witten invariant $\GW_{A_\max+B,1}^{M_\Lambda} (\PD (c) )$ vanishes when $B=\akl$.  Therefore 
$$ 0= \GW_{A_\max+B,1}^{M_\Lambda} (\PD (A_{3}) ) = a_B \cdot A_{3} = a_{4}$$
because $a_2=a_{3}=0$. Then, using that $a_4=0$, we get
$$ 0= \GW_{A_\max+B,1}^{M_\Lambda} (\PD (A_{4}) ) = a_B \cdot A_{4} = a_{5}$$ 
and by repeating the process around the polytope we get for all $k$, $3 \leq k \leq n-2$,
$$ 0= \GW_{A_\max+B,1}^{M_\Lambda} (\PD (A_{k}) ) = a_B \cdot A_{k} = a_{k+1}$$ 
so  that all the coefficients vanish except $a_n, a_1$. That is, we obtain $a_B=a_n  A_n +a_1A_1$ for some $a_n, a_1 \in \Z$  when  $B=\akl$. Since $\PD (A_2)= Z_2Z_b$ and $\PD (A_1)= Z_1Z_b$  it follows from Theorem \ref{GWinvariants} that if $k \geq l$ then 
\begin{align*}
 &0  = \GW_{\sigma_{B_{k,l},1}}^{M_\Lambda} (Z_2Z_b) = a_B \cdot A_2  = (a_n  A_n +a_1A_1) \cdot A_2= a_1,  \\
& 1   = \GW_{\sigma_{B_{k,l},1}}^{M_\Lambda} (Z_1Z_b) = a_B \cdot A_1  = (a_n  A_n +a_1A_1) \cdot A_1= a_n - 2a_1.  
\end{align*}
We conclude that $a_n =1$,  $a_1 =0$ and $a_B= A_n$ in this case.  If $k < l$ then we obtain
\begin{equation*}
-1  = \GW_{\sigma_{B_{k,l},1}}^{M_\Lambda} (Z_2Z_b) = a_1  \ \mbox{and} \
2   = \GW_{\sigma_{B_{k,l},1}}^{M_\Lambda} (Z_1Z_b) = a_n - 2a_1. 
\end{equation*}
Therefore, in this case, $a_n =0$,  $a_1 =-1$ and $a_B= -A_1$. This concludes the proof of Theorem \ref{theo:contributions}, Case \textit{(\ref{list:2-0-A-Cons})}.

\subsection{An intermediate result}
  \label{sec:proof-prop-nopointsinv}

Before giving the proof of Theorem \ref{GWinvariants}, we first need an intermediate result about some particular 0--point Gromov--Witten invariants. Recall that, by the divisor axiom, the 0--point invariant  $\GW^{M_\Lambda}_{0}(A)$, for $A \neq 0 \in H_2(M_\Lambda;\bb Z)$, is given by
$$\GW^{M_\Lambda}_0(A)= \frac{1}{h(A)^3} \GW_{A,3}^{M_\Lambda}(h,h,h)$$
where $h \in H^2(M_\Lambda;\bb Q)$ is such that $h(A)=\int_A h \neq 0$. From now on we will suppress the indication of the number of marked points  when that number is clear from the context and the expression for the Gromov--Witten invariant.

\begin{prop}\label{nopointsinvariants}
Let $k$ and $l$ be non-negative integers. Then 
$$\GW^{M_\Lambda}(kA_n+lA_1) = \left\{ \begin{array}{cl}
\displaystyle-\frac{1}{k^3} & \mbox{if} \quad l=0, \vspace{.1cm}\\
\displaystyle-\frac{1}{l^3} & \mbox{if} \quad k=0, \vspace{.1cm}\\
\displaystyle -\frac{1}{k^3} & \mbox{if} \quad k=l, \\
0 & \mbox{otherwise}. \end{array}\right.$$
\end{prop}

\begin{proof} In Steps 1 and 2 below, we prove the result in the first two cases. Then, in Step 3., we prove the result in the remaining cases by adapting Steps 1 and 2.  A good reference for what follows is  \cite{Melissa13}.\\

\textbf{Step 1.} Let $k > 0$. We begin with some preliminaries about moduli spaces of stable curves. Let $\Mbar_{0,n}( \bbcp^1,k)$ denote the moduli space
of genus $0$, $n$--pointed, degree $k$ stable maps to $ \bbcp^1$.
Let $p:\Mbar_{0,1}( \bbcp^1,k)\to \Mbar_{0,0}( \bbcp^1,k)$
be the universal curve, and let $\eva:\Mbar_{0,1}( \bbcp^1,k) \to  \bbcp^1$
be the evaluation map at the marked point. 
$\Mbar_{0,0}( \bbcp^1,k)$ is a smooth Deligne--Mumford stack of dimension $2k-2$ and the map $p$ is the forgetting morphism, which forgets the marked point. The following short exact sequence over $ \bbcp^1$: 
$$
0\to \cO_{ \bbcp^1} \to \cO_{ \bbcp^1}(1)\oplus \cO_{ \bbcp^1}(1) \to \cO_{ \bbcp^1}(2)=T_{ \bbcp^1}\to 0,
$$
induces the short exact sequence
$$
0\to T _{ \bbcp^1}^*=\cO_{ \bbcp^1}(-2) \to \cO_{ \bbcp^1}(-1)\oplus \cO_{ \bbcp^1}(-1)\to \cO_{ \bbcp^1}\to 0.
$$
Given a genus $0$, $0$--pointed, degree $k$ stable map $u:C\to  \bbcp^1$, 
we have a short exact sequence of vector bundles over the domain $C$:
\begin{equation}\label{eqn:short}
0\to u^*\cO_{ \bbcp^1}(-2) \to u^*\cO_{ \bbcp^1}(-1)\oplus 
u^*\cO_{ \bbcp^1}(-1)\to \cO_C \to 0 \,.
\end{equation}
Since
$H^0(C,u^*\cO_{ \bbcp^1}(-2))=\{0\}$ and $H^0(C,u^*\cO_{ \bbcp^1}(-1)\oplus u^*\cO_{ \bbcp^1}(-1))=\{0\}$,
the long exact sequence in cohomology associated to \eqref{eqn:short}
becomes
$$
0\to H^0(C,\cO_C)\to H^1(C, u^*\cO_{ \bbcp^1}(-2))
\to H^1(C,u^*\cO_{ \bbcp^1}(-1)\oplus u^*\cO_{ \bbcp^1}(-1))\to 0,
$$
where the complex dimension of $H^1(C,u^*\cO_{ \bbcp^1}(-2))$ and $H^1(C,u^*\cO_{ \bbcp^1}(-1)\oplus u^*\cO_{ \bbcp^1}(-1))$ are respectively $2k-1$ and $2k-2$. 

Next we define two bundles over $\Mbar_{0,0}( \bbcp^1,k)$:
$$
E_k:= p_*\eva^*\cO_{ \bbcp^1}(-2)\quad \mbox{and} \quad
V_k:= p_*\eva^*(\cO_{ \bbcp^1}(-1)\oplus \cO_{ \bbcp^1}(-1)) \,.
$$
The bundle $E_k$ has rank $2k-1$ and its fiber over $[u:C\to  \bbcp^1]$ is 
$H^1(C, u^*\cO_{ \bbcp^1}(-2))$, while $V_k$ has rank $2k-2$ and fiber $H^1(C, u^*\cO_{ \bbcp^1}(-1)\oplus u^*\cO_{ \bbcp^1}(-1))$.
They belong to the following short exact sequence
$$
0\to \cO_{\cM} \to E_k \to V_k \to 0,
$$
where $\cO_{\cM}$ is the trivial line bundle over $\Mbar_{0,0}( \bbcp^1,k)$. 
Therefore, the Euler and Chern classes of these bundles satisfy
\begin{equation}\label{chern}
e(E_k)=c_{2k-1}(E_k)=0,\quad
e(V_k) = c_{2k-2}(V_k)= c_{2k-2}(E_k), 
\end{equation}
Finally, recall that $\int_{[\Mbar_{0,0}( \bbcp^1,k)]}e(V_k)=\frac{1}{k^3}$ (see Manin \cite{M}).\\

\textbf{Step 2.} We now consider the case of a toric fibration $\pi: M_\Lambda \to  \bbcp^1$ where the total space is a toric manifold of (complex) dimension 3 and each fiber is diffeomorphic to the toric surface $M$. Using the previous notation, we want to show that 
$$
 \GW(kA_n)= \int_{[\Mbar_{0,0}(M_\Lambda,kA_n)]^\vir} 1 = - \frac{1}{k^3}.
$$

We first introduce some notation. We have
$$
H_{\C^*}^*(\mathrm{point};\Z)= H^*(B\C^*;\Z)=H^*( \bbcp^\infty;\Z)=\Z[u],
$$
where $u =c_1(\cO_{ \bbcp^\infty}(-1))$ is the
first Chern class of the tautological line
bundle over $B\C^*= \bbcp^\infty$. Let
$L_{mu}$ denote the $\C^*$--equivariant line
bundle over a point given by the 1--dimensional
$\C^*$--representation $t\mapsto t^m$. Then
$$
(c_1)_{\C^*}(L_{mu}) = m u \,\in\, H^2_{\C^*}(\mathrm{point};\Z)
=\Z [u].
$$

The action of $\C^*$ on $ \bbcp^1$ by $t\cdot[x,y]= [tx,y]$ has two fixed points: 
$0=[0:1]$ and $\infty=[1:0]$ and at these points
$$
(c_1)_{\C^*}(T_0 \bbcp^1)=u,
\quad (c_1)_{\C^*}(T_\infty \bbcp^1)=-u.
$$

There is a unique lift of this action to $M_\Lambda$ which acts trivially on $\pi^{-1}(0)$. 
This lift induces a $\C^*$--action on
$\Mbar_{0,0}(M_\Lambda,kA_n)$ and we have
$$
\Mbar_{0,0}(M_\Lambda,kA_n)^{\C^*} = F_0 \cup F_\infty
$$
where $F_0$ and $F_\infty$ can be identified with $\Mbar_{0,0}( \bbcp^1, k)$ as moduli spaces of maps 
to $\pi^{-1}(0)$ and $\pi^{-1}(\infty)$, respectively.

By virtual localization \cite{GP}, 
$$
\int_{[\Mbar_{0,0}(M_\Lambda,kA_n)]^\vir}1
=\int_{[F_0]^\vir}\frac{1}{e_{\C^*}(N^\vir_{F_0})}
+\int_{[F_\infty]^\vir }\frac{1}{e_{\C^*}(N^\vir_{F_\infty})} 
$$
where $N^\vir_{F_0}$ and $N^\vir_{F_\infty}$ are the virtual normal bundles to $F_0$ and $F_\infty$, respectively.

Let $\xi = [u:C\to  \bbcp^1]  \in F_0$. 
As explained in \cite{Melissa13}, the tangent space $T^1_\xi$ and the obstruction space $T^2_\xi$ 
at the moduli point $\xi  \in \Mbar_{0,0}(M_\Lambda,kA_n)$  fit in the {\it tangent-obstruction exact sequence}:
\begin{align}
  \begin{split}
\label{eq:tangentobstruction-es}
 0\to \Ext^0(\Omega_C, \cO_C) \to & H^0(C, u^*TM_\Lambda) \to T^1_\xi \\
& \to \Ext^1(\Omega_C, \cO_C) \to H^1(C, u^*TM_\Lambda) \to T^2_\xi \to 0
  \end{split}
\end{align}
where 
\begin{itemize}
\item $ \Ext^0(\Omega_C, \cO_C)$, respectively $\Ext^1(\Omega_C, \cO_C)$, is the space of infinitesimal automorphisms, respectively deformations, of the domain $C$,
\item $H^0(C, u^*TM_\Lambda)$, respectively $H^1(C, u^*TM_\Lambda)$, is the space of infinitesimal deformations of, respectively obstructions to deforming, the map $u$.
\end{itemize}
Equivalently, 
\begin{align*}
 0\to \Ext^0(\Omega_C, \cO_C) \to & H^0(C, u^*T \bbcp^1)\oplus L_u \to T^1_\xi \\
& \to \Ext^1(\Omega_C, \cO_C) \to H^1(C, u^*\cO(-2))\to T^2_\xi\to 0 \;.
\end{align*}
Together with the fact that $e(E_k)=0$, this leads to
$$
\int_{[F_0]^\vir} \frac{1}{e_{\C^*}(N^\vir_{F_0})} = \int_{\Mbar_{0,0}( \bbcp^1,k)}
\frac{e(E_k)}{e(L_u)} =0 \,.
$$

Suppose now that $\xi \in F_\infty$. 
In this case \eqref{eq:tangentobstruction-es} is equivalent to
\begin{align*}
0\to \Ext^0(\Omega_C, \cO_C) \to & H^0(C, u^*T \bbcp^1)\oplus L_{-u} \to T^1_\xi \\
& \to \Ext^1(\Omega_C, \cO_C) \to H^1(C, u^*\cO(-2))\otimes L_{-u}\to T^2_\xi\to 0 
\end{align*}
so that 
$$
\int_{[F_\infty]^\vir} \frac{1}{e_{\C^*}(N^\vir_{F_\infty})}=\int_{\Mbar_{0,0}( \bbcp^1,k)}
\frac{e(E_k\otimes L_{-u})}{u}
$$
where
$$
e(E_k\otimes L_{-u}) =\sum_{i=0}^{2k-1} (-u)^i c_{2k-1-i}(E_k) =
-u e(V_k) + \sum_{i=2}^{2k-1} c_{2k-1-i}(E_k)(-u)^i
$$
by \eqref{chern}. Together with the aforementioned result due to Manin, this now yields
$$
\int_{[F_\infty]^\vir}\frac{1}{e_{\C^*}(N^\vir_{F_\infty})} 
= - \int_{\Mbar_{0,0}( \bbcp^1,k)}e(V_k) =-\frac{1}{k^3} \,.
$$

This proves that $\GW(kA_n) = - \tfrac{1}{k^3}$, which finishes the proof of  the first case of the proposition. The second case follows by symmetry.\\

\textbf{Step 3.}
For the third and fourth cases we adapt Steps 1 and 2 above to the case of genus 0, 1--pointed,  stable maps $u: C \to  \bbcp^1 \times  \bbcp^1$ of degree $k$ to the first sphere and of degree $l$ to the second sphere. We denote the moduli space of such maps by $\Mbar_{0,1}( \bbcp^1\times  \bbcp^1,(k, l))$, it is a Deligne--Mumford stack of dimension $2k +2l$.

As above, we define the evaluation map $\eva:\Mbar_{0,2}( \bbcp^1\times  \bbcp^1,(k, l)) \to  \bbcp^1 \times  \bbcp^1$ and the forgetful map $p:\Mbar_{0,2}( \bbcp^1\times  \bbcp^1,(k, l))\to \Mbar_{0,1}( \bbcp^1\times  \bbcp^1,(k, l))$ which forgets the second marked point, and we consider the following short exact sequence over $ \bbcp^1\times  \bbcp^1$:
\begin{align*}
0\to & \cO_{ \bbcp^1}(-2) \times \cO_{ \bbcp^1}(-2) \to \\
& \to (\cO_{ \bbcp^1}(-1)\oplus \cO_{ \bbcp^1}(-1)) \times (\cO_{ \bbcp^1}(-1)\oplus \cO_{ \bbcp^1}(-1))\to \cO_{ \bbcp^1} \times \cO_{ \bbcp^1}\to 0 \,.
\end{align*}
Given $[u: C \to  \bbcp^1\times  \bbcp^1] \in \Mbar_{0,1}( \bbcp^1\times  \bbcp^1,(k, l))$, this exact sequence pulls-back to
\begin{multline*}
0\to u^*(\cO_{ \bbcp^1}(-2) \times \cO_{ \bbcp^1}(-2) )\to \\
\to u^*((\cO_{ \bbcp^1}(-1)\oplus \cO_{ \bbcp^1}(-1)) \times (\cO_{ \bbcp^1}(-1)\oplus \cO_{ \bbcp^1}(-1)))
\to u^*( \cO_{ \bbcp^1} \times \cO_{ \bbcp^1})\to 0 \,.
\end{multline*}
%where $u^*(\cO_{ \bbcp^1}(-2) \times \cO_{ \bbcp^1}(-2))$ and $u^*((\cO_{ \bbcp^1}(-1)\oplus \cO_{ \bbcp^1}(-1)) \times (\cO_{ \bbcp^1}(-1)\oplus \cO_{ \bbcp^1}(-1)))$ are degree $-2k -2l$ bundles over C.
In a similar way to the previous case we define bundles 
\begin{align*}
E_{k,l}&:= p_*\eva^*(\cO_{ \bbcp^1}(-2) \times \cO_{ \bbcp^1}(-2)) \quad \mbox{and} \\
V_{k,l}&:= p_*\eva^*((\cO_{ \bbcp^1}(-1)\oplus \cO_{ \bbcp^1}(-1)) \times (\cO_{ \bbcp^1}(-1)\oplus \cO_{ \bbcp^1}(-1)))
\end{align*}
over $\Mbar_{0,1}( \bbcp^1\times  \bbcp^1,(k, l))$. Now $E_{k,l}$ and $V_{k,l}$ have rank $2k+2l -2$ and $2k+2l -4$, respectively. 
In this case we have the following short exact sequence of bundles
$$
0\to \cO_{\cM} \to E_{k,l} \to V_{k, l} \to 0 \,,
$$
where, again, $\cO_{\cM}$ is the trivial bundle. So relations \eqref{chern} become in this case
\begin{equation*}
e(E_{k,l})=c_{2k+2l-2}(E_{k, l})=0,\quad
e(V_{k, l}) = c_{2k+2l-4}(V_{k, l})= c_{2k+2l-4}(E_{k,l}) \,.
\end{equation*}

We consider the same $\C^*$--action as above, with fixed points $0=[0:1]$ and $\infty=[1:0]$, and its lift to $M_\Lambda$ acting trivially on $\pi_{-1}(0)$.
It induces a $C^*$--action on $\Mbar_{0,0}(M_\Lambda,kA_n +A_1)$. Analogously  to the first case we have
$$
\Mbar_{0,0}(M_\Lambda,kA_n+lA_1)^{\C^*} = F_0 \cup F_\infty
$$
where 
$F_0$ and $F_\infty$ can now be identified with $\Mbar_{0,1}( \bbcp^1\times  \bbcp^1, (k,l))$.

Again, by virtual localization \cite{GP}, 
$$
\int_{[\Mbar_{0,0}(M_\Lambda,kA_n+lA_1)]^\vir}1
=\int_{[F_0]^\vir}\frac{1}{e_{\C^*}(N^\vir_{F_0})}
+\int_{[F_\infty]^\vir }\frac{1}{e_{\C^*}(N^\vir_{F_\infty})} \;.
$$

However, in this case, since  $\dim \Mbar_{0,1}( \bbcp^1\times  \bbcp^1, (k,l))= 2k+2l$ and both Euler classes $e (E_{k,l})$ and $e(E_{k,l}\otimes L_{-u})$ have smaller degree than this dimension we conclude that both integrals 
$$\int_{[F_0]^\vir}\frac{1}{e_{\C^*}(N^\vir_{F_0})}=  \int_{\Mbar_{0,1}( \bbcp^1\times  \bbcp^1 ,(k,l))} \frac{e(E_k)}{e(L_u)} \quad \quad 
\mbox{and}$$
$$\int_{[F_\infty]^\vir }\frac{1}{e_{\C^*}(N^\vir_{F_\infty})}= \int_{\Mbar_{0,1}( \bbcp^1\times  \bbcp^1,(k,l))}
\frac{e(E_k\otimes L_{-u})}{u}$$ vanish, 
unless $k=l$ when we can reduce the calculation of the Gromov--Witten invariant to the  first case by considering curves in class $k(A_n+ A_1)$.
\end{proof}

\subsection{Proof of Theorem \ref{GWinvariants}}
\label{sec:proof-theor-GWinv}

We are now ready to prove Theorem \ref{GWinvariants} which will conclude the proof of Theorem \ref{theo:contributions}.

We use an induction argument. First notice that using the results from Spielberg recalled in Section \ref{SpielbergFormula}, 
% Spielberg's results in ~\cite{Spielberg99} or ~\cite{Spielberg00} 
we can easily compute the value of the three Gromov--Witten invariants of Theorem \ref{GWinvariants} for the base cases  $k=0,1$ and $l=0,1$ (see Lemma \ref{lemm:GW_Z1Zb-ZnZb} for the computation of some of these invariants). Now we assume they hold for  all values $i,j$ such that $i \leq k-1$ and  $j \leq l-1$  and we will prove they also hold for $i=k$ and $j=l$.
Because $[M] \cdot [\sigma]=1$ for any section class $\sigma$, the divisor axiom for Gromov--Witten invariants (see Proposition \ref{prop:axiom-for-GW-inv}) implies that the 1--point invariant $\GW_{A_\max+B,1}^{M_\Lambda} (c)$ equals the 3--point invariant 
$\GW_{A_\max+B,3}^{M_\Lambda} ([M],[M],c)$. It follows easily, from the fan description of the manifold $M_\Lambda$ in Section \ref{sec:example}, that  $\PD([M])=Z_b$. Therefore we need  to compute the Gromov--Witten  invariants 
$$\GW_{A_\max+B,3}^{M_\Lambda} (Z_b,Z_b,Z)$$ 
with $Z \in H^4(M;\bb Z)$ since the degrees satisfy the equation $2\deg Z_b+ \deg Z = 2N + 2c_1(B_{k,l})+ 2m-6$ where $\dim M_\Lambda=2N=6$, $\deg Z_b=2$, $c_1(B_{k,l})=1$ and $m=3$ is the number of marked points. 

The main idea of the proof is to compute well-chosen Gromov--Witten invariants via the splitting axiom along two different partitions and then deduce relations from the two resulting expressions. Namely, we start with
$\GW_{B_{k,l},4}^{M_\Lambda} (Z_1,Z_b,Z_b,Z_1; \pt)$, from which we will deduce:
\begin{lemm}\label{lemma:GWcomp1}
  $\GW_{B_{k,l}}^{M_\Lambda} (Z_1Z_b)$ and $\GW_{B_{k,l}}^{M_\Lambda} (Z_1Z_2)$ satisfy the following equations:
  \begin{align}
    (k-2l)\,\GW_{B_{k,l}}^{M_\Lambda} (Z_1Z_b)+\,\GW_{B_{k,l}}^{M_\Lambda} (Z_1Z_2)&= k-2l \,, &\quad \mbox{if $k \geq l$,}   \label{final1casea} \\
    (k-2l)\,\GW_{B_{k,l}}^{M_\Lambda} (Z_1Z_b)+\,\GW_{B_{k,l}}^{M_\Lambda} (Z_1Z_2)&= 2k-4l+2 \,, &\quad \mbox{if $k < l$.}  \label{final1caseb}
  \end{align}
\end{lemm}

\begin{proof}
  \textbf{Step 1. We use the partition $S_0=\{1,2\}, S_1=\{3,4\}$ of the index set $\{1,2,3,4\}$} and apply the splitting axiom so that we get:
\begin{align*}
  \GW&_{B_{k,l},4}^{M_\Lambda} (Z_1,Z_b,Z_b,Z_1; \pt)  
  =  \;\sum_{C_0+C_1=B_{k,l} }\GW_{C_0,3}^{M_\Lambda} (Z_1,Z_b, e_\nu)\, g^{\nu\mu}\, \GW_{C_1,3}^{M_\Lambda} (e_\mu, Z_b,Z_1) 
\end{align*}
where the sum runs over all $C_0$, $C_1$ such that
$$
\left\{ \begin{array}{l}
C_0=A_\max+k_0A_n+l_0A_1 \\
C_1=k_1A_n+l_1A_1 \end{array} \right. \quad  \mbox{or} \quad \left\{ \begin{array}{l}
C_0=k_0A_n+l_0A_1 \\
C_1=A_\max+k_1A_n+l_1A_1 \end{array} \right. $$
with $k_0+k_1=k$ and $l_0+l_1=l$. In order to ease the reading, we used in the equality above \textit{as well as in the rest of this proof}, the Einstein summation convention with respect to the basis of the cohomology (and thus forgot $\sum_{\nu,\mu}$ from the notation).

This leads us to 
\begin{align*}
  \GW_{B_{k,l},4}^{M_\Lambda} (Z_1,Z_b,&Z_b,Z_1; \pt)  
 % =  \;\sum_{A_0+A_1=B_{k,l} }\GW_{A_0,3}^{M_\Lambda} (Z_1,Z_b, e_\nu)\, g^{\nu\mu}\, \GW_{A_1,3}^{M_\Lambda} (e_\mu, Z_b,Z_1)  \\
  = \;\GW_{A_\max,3}^{M_\Lambda} (Z_1,Z_b, e_\nu)\, g^{\nu\mu}\, \GW_{\akl,3}^{M_\Lambda} (e_\mu, Z_b,Z_1)  \\  
 &\quad+ \GW_{B_{k,l},3}^{M_\Lambda} (Z_1,Z_b, e_\nu)\, g^{\nu\mu}\, \GW_{0,3}^{M_\Lambda} (e_\mu, Z_b,Z_1)  \\
 &\quad+\sum_{1 \leq k_0+l_0 \leq k+l -1}\GW_{B_{k_0,l_0},3}^{M_\Lambda} (Z_1,Z_b, e_\nu)\, g^{\nu\mu}\, \GW_{k_1A_n+l_1A_1,3}^{M_\Lambda} (e_\mu, Z_b,Z_1)  \\
 &\quad+  \GW_{0,3}^{M_\Lambda} (Z_1,Z_b, e_\nu)\, g^{\nu\mu}\, \GW_{B_{k,l},3}^{M_\Lambda} (e_\mu, Z_b,Z_1)  \\  
 &\quad+ \GW_{\akl,3}^{M_\Lambda} (Z_1,Z_b, e_\nu)\, g^{\nu\mu}\, \GW_{A_\max,3}^{M_\Lambda} (e_\mu, Z_b,Z_1) \\
 &\quad +\sum_{1 \leq k_0+l_0 \leq k+l -1}\GW_{k_0A_n+l_0A_1,3}^{M_\Lambda} (Z_1,Z_b, e_\nu)\, g^{\nu\mu}\, \GW_{B_{k_1,l_1},3}^{M_\Lambda} (e_\mu, Z_b,Z_1) 
\end{align*}
Now, by using the divisor axiom (see Proposition \ref{prop:axiom-for-GW-inv}) together with the fact that $\int_{A_n}Z_b = \int_{A_1}Z_b=0$, we end up with:
\begin{align} \label{eq1}
   \GW_{B_{k,l},4}^{M_\Lambda} (Z_1,Z_b,&Z_b,Z_1; \pt) = 2\, \GW_{B_{k,l},3}^{M_\Lambda} (Z_1,Z_b, e_\nu)\, g^{\nu\mu}\, \GW_{0,3}^{M_\Lambda} (e_\mu, Z_b,Z_1) \,.
\end{align}
Moreover, $\int_{B_{k,l}}Z_1= Z_1(A_\max+kA_n+lA_1) =k -2l$,  $\int_{B_{k,l}}Z_b= 1$ and by the zero axiom (see Proposition \ref{prop:axiom-for-GW-inv}):
$$\GW_{0,3}^{M_\Lambda} (e_\mu, Z_b,Z_1)= \int_{M_\Lambda} e_\mu \cup Z_b \cup Z_1 = \left \{ \begin{array}{cl}
-2 & \mbox {if } e_\mu=Z_1, \\
1 & \mbox {if } e_\mu=Z_2, \\
0 & \mbox{otherwise}. \end{array} \right.$$ 
So one gets that \eqref{eq1} leads to
\begin{equation}\label{eq1final}
\GW_{B_{k,l},4}^{M_\Lambda} (Z_1,Z_b,Z_b,Z_1; \pt) = 2(k-2l) \sum_ {{\nu}\,:\,{|e_\nu|=4}}\GW_{B_{k,l},1}^{M_\Lambda} (e_\nu)\, (g^{\nu n}-2g^{\nu1}) \,.
\end{equation}
\begin{rema}\label{rema:compute-sums-GW-inv}
From the diagram of Figure \ref{fig:Diagram}, $\GW_{B_{k,l},1}^{M_\Lambda}(e_\nu) \neq 0$ only if  the class $e_\nu$ is Poincar\'e dual to one of the following homology classes: $A_{n-1}$, $A_n$, $A_1$, $A_2$, $A_\max$, or $ A_\max + A_1$, since the marked point should lie in one of the following cones: $\sigma_2$, $\sigma_3$, $\sigma_4$, $\sigma_7$, $\sigma_8$, or $\sigma_9$. Their Poincar\'e duals are the classes $Z_{n-1}Z_b$, $Z_nZ_b$, $Z_1Z_b$, $ Z_2Z_b$, $Z_1Z_n$, and $ Z_1Z_2$, respectively. Note that the only ones that belong to the basis of the cohomology are $Z_1Z_b$, $Z_2Z_b$, and $Z_1Z_2$. Therefore, at most three terms appear in the summation in Equation \eqref{eq1final} above and the coefficients can be computed thanks to Lemma \ref{lem:relationscoeffmatrix}.
\end{rema}
In the case of Equation \eqref{eq1final}, we end up with
\begin{equation}\label{eq1final2}
\GW_{B_{k,l},4}^{M_\Lambda} (Z_1,Z_b,Z_b,Z_1; \pt) = 2(k-2l) \GW_{B_{k,l},1}^{M_\Lambda} (Z_1Z_b) \,.
\end{equation}

  \textbf{Step 2. We use the partition $S_0=\{1,4\}, S_1=\{2,3\}$}.\\ The same Gromov--Witten invariant is given by the following expression
\begin{multline}\label{eq2}
    \GW _{B_{k,l},4}^{M_\Lambda} (Z_1,Z_b,Z_b,Z_1; \pt)  =  \GW_{0,3}^{M_\Lambda} (Z_1,Z_1, e_\nu)\, g^{\nu\mu}\, \GW_{B_{k,l},3}^{M_\Lambda} (e_\mu, Z_b,Z_b) \\
 +\sum_{1 \leq k_0+l_0 \leq k+l}\GW_{k_0A_n+l_0A_1,3}^{M_\Lambda} (Z_1,Z_1, e_\nu)\, g^{\nu\mu}\, \GW_{B_{k_1,l_1},3}^{M_\Lambda} (e_\mu, Z_b,Z_b).
\end{multline}
Since (by the zero axiom):
\begin{align*}
&\GW_{0,3}^{M_\Lambda} (Z_1,Z_1, e_\nu) = \int_{M_\Lambda} Z_1^2 \cup e_\nu = \left \{ \begin{array}{cl}
4 & \mbox{if } e_\nu=Z_1 \\
-2 & \mbox{if } e_\nu= Z_2 \mbox{ or } e_\nu= Z_b \\
0 & \mbox{otherwise} \end{array} \right. \,, \\
 &\mbox{and }\quad \int_{k_0A_n+l_0A_1} e_\nu = \left \{ \begin{array}{cl}
k_0-2l_0 & \mbox {if } e_\nu=Z_1, \\
l_0 & \mbox {if } e_\nu=Z_2, \\
0 & \mbox{otherwise} \end{array} \right.
\end{align*}
 it follows from the divisor axiom that \eqref{eq2} is equal to  
 \begin{multline}\label{eq2simplified}
     =  \sum_{\mu \,:\, |e_\mu|=4}(4g^{1\mu}-2g^{2\mu}-2g^{b\mu})\GW_{B_{k,l},1}^{M_\Lambda} (e_\mu)  \\
 +\sum_{\mu} \sum_{1 \leq k_0+l_0 \leq k+l}(k_0-2l_0)^2\GW(k_0A_n+l_0A_1) \, ((k_0-2l_0)g^{1\mu} 
+l_0\,g^{2\mu})\, \GW_{B_{k_1,l_1},1}^{M_\Lambda} (e_\mu) 
\end{multline}
where $\GW(k_0A_n+l_0A_1)$ denotes the 0--point invariant in class $k_0A_n+l_0A_1$. These were computed in Proposition \ref{nopointsinvariants}. In order to simplify the expression, we will denote them by $\GW_0$. We will also omit the index 1 indicating the number of marked points for the various 1--point Gromov--Witten invariants appearing in what remains of the proof.

In view of Remark \ref{rema:compute-sums-GW-inv} above and Lemma \ref{lem:relationscoeffmatrix}, equation \eqref{eq2simplified} actually reads
\begin{equation*}
\begin{split}
= & -2\GW_{B_{k,l}}^{M_\Lambda} (Z_1Z_2) +   \sum_{1 \leq k_0+l_0 \leq k+l}(k_0-2l_0)^2\GW_0 \left[(k_0\,g^{1,1b}+l_0)\,\GW_{B_{k_1,l_1}}^{M_\Lambda} (Z_1Z_b)  \right. \\
& \left. + k_0 \,g^{1,2b} \GW_{B_{k_1,l_1}}^{M_\Lambda} (Z_2Z_b) + k_0 \,g^{1,12} \GW_{B_{k_1,l_1}}^{M_\Lambda} (Z_1Z_2))  \right].
\end{split} 
\end{equation*}
Then, using Proposition \ref{nopointsinvariants},  we separate the summation in three summations: $k_0=0$, $l_0=0$, and $k_0=l_0$: 
\begin{equation}\label{eq2final}
\begin{split}
& =  -2\GW_{B_{k,l}}^{M_\Lambda} (Z_1Z_2)  + \sum_{l_0=1}^{l}4l_0^3\left(-\frac{1}{l_0^3} \right) \,\GW_{B_{k,l_1}}^{M_\Lambda} (Z_1Z_b) \\
 &  + \sum_{k_0=1}^{k}k_0^3\left(-\frac{1}{k_0^3} \right)\,  \left[g^{1,1b} \,\GW_{B_{k_1,l}}^{M_\Lambda} (Z_1Z_b)  + g^{1,2b} \GW_{B_{k_1,l}}^{M_\Lambda} (Z_2Z_b) + g^{1,12} \GW_{B_{k_1,l}}^{M_\Lambda} (Z_1Z_2))  \right] \\
 &  - \sum_{k_0=1}^{\rm{min}(k,l)}k_0^3\left(\frac{1}{k_0^3} \right)\,\left[(g^{1,1b}+1) \,\GW_{B_{k_1,l_1}}^{M_\Lambda} (Z_1Z_b)  + g^{1,2b} \GW_{B_{k_1,l_1}}^{M_\Lambda} (Z_2Z_b) + g^{1,12} \GW_{B_{k_1,l_1}}^{M_\Lambda} (Z_1Z_2))  \right]
\end{split}
\end{equation}

Applying the induction hypotheses and Lemma \ref{lem:relationscoeffmatrix} we can simplify even further this expression. However we need to consider  two different cases: 
\begin{enumerate}[(a)]
\item If $k \geq l $ then \eqref{eq2final} is equal to 
\begin{equation}\label{eq2finalcasea}
\begin{split}
=  & -2\GW_{B_{k,l}}^{M_\Lambda} (Z_1Z_2)  - 4l + \sum_{k_1=0}^{l-1}(-2g^{1,1b}  + g^{1,2b}  -2 g^{1,12}) -  \sum_{k_1=l}^{k-1} g^{1,1b} -\sum_{k_0=1}^{l} (g^{1,1b}+1)\\
= & -2\GW_{B_{k,l}}^{M_\Lambda} (Z_1Z_2) +l \,(g^{1,2b}  -2 g^{1,12}) -(k+2l) \, g^{1,1b} -5l
\end{split}
\end{equation}
\item If $k <l$ then \eqref{eq2final} is equal to 
\begin{equation}\label{eq2finalcaseb}
\begin{split}
& =-2\GW_{B_{k,l}}^{M_\Lambda} (Z_1Z_2)  -4\sum_{l_1=0}^{k} 1  -4\sum_{l_1=k+1}^{l-1} 2 - \sum_{k_0=1}^k (2g^{1,1b}  - g^{1,2b}  +2 g^{1,12})  \\
&  \quad \quad -  \sum_{k_0=1}^{k} (g^{1,1b}  - g^{1,2b}  +2 g^{1,12}) \\
&= -2\GW_{B_{k,l}}^{M_\Lambda} (Z_1Z_2) -4(k+1) -8(l-1-k) +6k -2k(2 g^{1,12} -g^{1,2b}) % \\
%  & = -  2\GW_{B_{k,l}}^{M_\Lambda} (Z_1Z_2) +4k -8l +4. 
\end{split}
\end{equation}
%  where the last step follows from \eqref{coeffmatrixII}.  
\end{enumerate}

\textbf{Step 3. We use the fact that the results of Steps 1 and 2 coincide}, i.e. when $k \geq l$, \eqref{eq1final2}=\eqref{eq2finalcasea} while when $k < l$,  \eqref{eq1final2}=\eqref{eq2finalcaseb}.

First we consider the case $k \geq l$, \eqref{eq1final2}=\eqref{eq2finalcasea} leads to
\begin{equation*}
2(k-2l)\,\GW_{B_{k,l}}^{M_\Lambda} (Z_1Z_b)+2\,\GW_{B_{k,l}}^{M_\Lambda} (Z_1Z_2)= l \,(g^{1,2b}  -2 g^{1,12}) -(k+2l) \, g^{1,1b} -5l.
\end{equation*}
In particular, when $k=1, l=0$ and $k=1, l=1$, it follows from the base cases (Lemma \ref{lemm:GW_Z1Zb-ZnZb}) that the matrix elements satisfy:
\begin{align} \label{coeffmatrixI}
 g^{1,1b}=-2  \qquad \mbox{and} \qquad 2 g^{1,12} -g^{1,2b} =3,
\end{align}
respectively. Getting back to the general case, we finally deduce:

\begin{enumerate}[(a)] 
\item For $k \geq l $, \eqref{eq1final2}=\eqref{eq2finalcasea} together with \eqref{coeffmatrixI} give
\begin{equation*}% \label{final1casea}
(k-2l)\,\GW_{B_{k,l}}^{M_\Lambda} (Z_1Z_b)+\,\GW_{B_{k,l}}^{M_\Lambda} (Z_1Z_2)= k-2l \,,
\end{equation*}
\item and for $k < l$, \eqref{eq1final2}=\eqref{eq2finalcaseb} together with \eqref{coeffmatrixI} give
\begin{equation*}% \label{final1caseb}
(k-2l)\,\GW_{B_{k,l}}^{M_\Lambda} (Z_1Z_b)+\,\GW_{B_{k,l}}^{M_\Lambda} (Z_1Z_2)= 2k-4l+2 \,.
\end{equation*}
\end{enumerate}
This ends the proof of the lemma.
\end{proof}

%= -  2\GW_{B_{k,l}}^{M_\Lambda} (Z_1Z_2) +4k -8l +4. 

We now proceed along the same lines but for two other Gromov--Witten invariants, namely,  $$\GW_{B_{k,l},4}^{M_\Lambda} (Z_1,Z_b,Z_b,Z_2; \pt) \quad \mbox { and } \quad \GW_{B_{k,l},4}^{M_\Lambda} (Z_2,Z_b,Z_b,Z_2; \pt) \,.$$
Since the method is exactly the same, we leave the computation to the interested reader and we simply give the four resulting equations. 

\begin{lemm}\label{lemma:GWcomp2}
  From $\GW_{B_{k,l},4}^{M_\Lambda} (Z_1,Z_b,Z_b,Z_2; \pt)$, we deduce
  \begin{align}
 l\,\GW_{B_{k,l}}^{M_\Lambda} (Z_1Z_b)+(k-2l)\GW_{B_{k,l}}^{M_\Lambda} (Z_2Z_b)-\GW_{B_{k,l}}^{M_\Lambda} (Z_1Z_2)&=l \,, &\quad \mbox{if $k \geq l$,}   \label{final2casea} \\
 l\,\GW_{B_{k,l}}^{M_\Lambda} (Z_1Z_b)+(k-2l)\GW_{B_{k,l}}^{M_\Lambda} (Z_2Z_b)-\GW_{B_{k,l}}^{M_\Lambda} (Z_1Z_2)&=4l-k-2 \,, &\quad \mbox{if $k < l$.}  \label{final2caseb}
  \end{align}
\end{lemm}

\begin{lemm}\label{lemma:GWcomp3}
  From $\GW_{B_{k,l},4}^{M_\Lambda} (Z_2,Z_b,Z_b,Z_2; \pt)$ we deduce:
  \begin{align}
(2l+2d_2-1)\,\GW_{B_{k,l}}^{M_\Lambda} (Z_2Z_b)+d_2\GW_{B_{k,l}}^{M_\Lambda} (Z_1Z_2)&=0 \,, &\quad \mbox{if $k \geq l$,}   \label{final3casea} \\
(2l+2d_2-1)\,\GW_{B_{k,l}}^{M_\Lambda} (Z_2Z_b)+d_2\GW_{B_{k,l}}^{M_\Lambda} (Z_1Z_2)&=1-2l \,, &\quad \mbox{if $k < l$,}  \label{final3caseb}
  \end{align}
where $d_2$ comes from the matrix $G$, see Table \eqref{eq:table-part-of-B}.
\end{lemm}

In order to conclude the proof of Theorem \ref{GWinvariants}, we consider two linear systems: 
\begin{itemize}
\item one given by the equations \eqref{final1casea}, \eqref{final2casea}, \eqref{final3casea}, corresponding to the case $k \geq l$ of Lemmas \ref{lemma:GWcomp1}, \ref{lemma:GWcomp2}, and \ref{lemma:GWcomp3} above,
\item the other given by the equations \eqref{final1caseb}, \eqref{final2caseb}, \eqref{final3caseb} corresponding to the case $k < l$.
\end{itemize}
The unknowns of these linear systems are the Gromov--Witten invariants we are looking for, namely, $\GW_{B_{k,l}}^{M_\Lambda} (Z_1Z_b)$, $\GW_{B_{k,l}}^{M_\Lambda} (Z_1Z_2)$, and $\GW_{B_{k,l}}^{M_\Lambda} (Z_2Z_b)$. The unique solutions of these systems give us the desired result. 

%%%%%%%%%%%%%%%%%%%%%%%%%%%%%%%%%%%%%%%%%%
%             
%        Applications and explicit examples
%
%%%%%%%%%%%%%%%%%%%%%%%%%%%%%%%%%%%%%%%%%%

\section{Applications and explicit examples}
\label{sec:appl-expl-exampl}

In this section we show some applications of our results and illustrate their relevance with some particular examples. More precisely, in Section \ref{sec:land-ginzb-potent} we show how to obtain an expression for the Landau--Ginzburg superpotential from the moment polytope
of a NEF toric 4--manifold. In Section \ref{sec:nef-exples-45pt-bu} we compute the Seidel elements, the quantum homology ring and the Landau--Ginzburg superpotential for two examples of NEF toric surfaces, namely $ \bbcp^2$ blown--up at 4 or 5 points. Finally, in Section 
\ref{sec:non-nef-examples} we show how we can use  the Fano and NEF computations to obtain explicit expressions of Seidel elements for some particular non-NEF manifolds, namely the Hirzebruch surfaces $\F_{2k}$ or $\F_{2k-1}$ with $k \geq 2$. As an example, we compute them explicitly for $\F_4$.

\subsection{The Landau--Ginzburg potential} 
\label{sec:land-ginzb-potent}

In this section we follow the works of McDuff--Tolman \cite{McDuffTolman06} and Ostrover--Tyomkin \cite{OstroverTyomkin09} which were themselves developments of original ideas due to Batyrev \cite{Batyrev93} and Givental \cite{Giv1,Giv2}. In particular, we will also use quantum cohomology. The definition is similar to quantum homology in Section \ref{sec:quantum-homology}, except that the coefficient ring is $\check{\Pi}:= \check{\Pi}^{\mathrm{univ}}[q,q^{-1}]$, with
$$ \check{\Pi}^{\mathrm{univ}}:= \left\{ \sum_{\kappa \in \R} r_\kappa t^\kappa \,\big|\, r_\kappa \in \Q, \ \#\{ \kappa < c\mid r_\kappa \neq 0\} < \infty, \forall c \in \R \right \}$$
(compare with \eqref{eq:pi-univ}) and that the product on $\QH^*(M; \omega) =  H^*(M; \Q) \otimes_\Q \check{\Pi}$ is Poincar\'e dual to the intersection product and is called the quantum cup product.

Let us recall some notation. Consider a torus $T$ with Lie algebra $\mathfrak t$ and lattice $\mathfrak t_\Z$. Let $(M, \omega)$ be a smooth toric $2m$--manifold with moment map $\Phi: M \to \mathfrak t^*$ and with moment polytope $P$. Let $D_1, \hdots, D_n$ be the facets of $P$, inducing homology classes $A_i=[\Phi^{-1}(D_i)] \in H_2(M; \Z)$, and let $v_1, \hdots, v_n$ denote the outward primitive integral vectors normal to the facets.
The moment polytope is given by 
$$P = \{ x \in \mathfrak t^* \,|\, \langle x, v_j \rangle \leq  \kappa_j, \mbox{ for } j=1, \hdots, n\}$$
where $ \kappa_j \in \R$. 
Any face of $P$, given as the intersection of facets $D_{j_1}, \hdots, D_{j\ell}$, admits a \emph{dual cone} consisting by definition of those elements in $\mathfrak t$ which are positive linear combinations of $v_{j_1}, \hdots, v_{j_\ell}$. As explained in \cite[Section 5.1]{McDuffTolman06}, any vector in $\mathfrak t$ lies in the dual cone of a unique face of $P$. Therefore, a subset $I= \{i_1, \hdots, i_k\} \subseteq \{1, \hdots n\}$ determines a unique face of $P$ whose dual cone contains $v_{i_1}+ \hdots+ v_{i_k}$. This face is given as the intersection of facets which we (still) denote by $D_{j_1}, \hdots, D_{j\ell}$ and there exist unique positive integers $c_1, \hdots, c_\ell$ so that 
$$v_{i_1}+ \hdots+ v_{i_k} -c_1v_{j_1}- \hdots -c_\ell v_{j_\ell}=0 \,.$$
% Given any face of $P$, let $D_{j_1}, \hdots, D_{j\ell}$ be the facets that intersect to form this face. The \emph{dual cone} is the set of elements in $\mathfrak t$ which can be written as a positive linear combination of $v_{j_1}, \hdots, v_{j_\ell}$. As explained in \cite[Section 5.1]{McDuffTolman06}, every vector in $\mathfrak t$ lies in the dual cone of a unique face of $P$. Therefore, given any subset $I= \{i_1, \hdots, i_k\} \subseteq \{1, \hdots n\}$ there is a unique face of $P$ so that $v_{i_1}+ \hdots+ v_{i_k}$ lies in its dual cone. Let $D_{j_1}, \hdots, D_{j\ell}$ be the facets that intersect to form this unique face. Then there exist unique positive integers $c_1, \hdots, c_\ell$ so that 
% $$v_{i_1}+ \hdots+ v_{i_k} -c_1v_{j_1}- \hdots -c_\ell v_{j_\ell}=0.$$
Batyrev showed that if $I$ is primitive, the sets $I$ and $J= \{j_1, \hdots, j_\ell\}$ are disjoint. Moreover, if $\beta_I \in H_2(M; \Z)$ is the class corresponding to the above relation (recall from Section \ref{sec:toric-geom-algebr} that $H_2(M; \Z)$ is isomorphic to the set of $(a_1, \hdots, a_n) \in \Z^n$ such that $\sum a_iv_i=0$), then by \eqref{ChernClass}:
\begin{align}
c_1(\beta_I) & = k - c_1 - \hdots -c_\ell \,,\label{ChernNumber}\\
\omega(\beta_I) & = v_{i_1}(D_{i_1})+ \hdots+ v_{i_k}(D_{i_k}) -c_1v_{j_1}(D_{j_1})- \hdots -c_\ell v_{j_\ell}(D_{j_\ell}) \notag \\
 & = \kappa_{i_1}+ \hdots+  \kappa_{i_k} -c_1 \kappa_{j_1}- \hdots - c_\ell  \kappa_{j_\ell} \,.\label{symplecticform}
\end{align}
Denote by  $\Lambda_i$  the circle action corresponding to $v_i$, that is, $\Lambda_i$ is the circle action whose moment map $\Phi_{\Lambda_i}$ is given by 
 the composition of the moment map $\Phi: M \to \mathfrak t^*$ with the linear functional $v_i \in \mathfrak t$. Let $S^*(\Lambda_i)=y_i \otimes q^{-1}t^{-v_i(D_i)} \in QH^{\mathrm{ev}}(M,\omega)^{\times}$ be the cohomological counterpart of the Seidel element. 
In \cite{McDuffTolman06} the authors show the following result.

\begin{prop}\label{QuantumRingIso}
Let $QH^*(M, \omega)$ denote the small quantum cohomology of the toric manifold $(M, \omega)$. The map $\Theta$ which sends $Z_i$ to the Poincar\'e dual of $\Phi^{-1}(D_i)$ induces an isomorphism 
$$ \Q[Z_1, \hdots,Z_n] \otimes \check\Pi / (\Lin(P) + SR_Y(P)) \cong QH^*(M, \omega),$$
where the ideal $\Lin(P)$ is generated by the linear relations 
$$ \Lin(P) = \left \langle \sum(x, v_j)Z_j \ | \ x \in \mathfrak t_{\Z}^* \right \rangle$$
 and the ideal $SR_Y(P)$ is given by 
$$ SR_Y(P)= \left \langle Y_{i_1}\hdots Y_{i_k} - Y_{j_1}^{c_1} \hdots Y_{j_\ell}^{c_\ell} \otimes q^{c_1(\beta_I)}t^{\omega(\beta_I)}  | \ I= \{ i_1, \hdots, i_k\} \mbox { is primitive}\right \rangle,$$
where 
\begin{equation}\label{higherterms}
 Y_i = Z_i + \mbox{ higher order terms},
 \end{equation}
 is a lift of the Seidel element $y_i$ in $\Q[Z_1,\hdots, Z_n] \otimes \check\Pi$, such that $\Theta (Y_i)= y_i$.
  \end{prop}
 As McDuff and Tolman explain in \cite{McDuffTolman06}, in general, it is not possible to find $Y_i$ without prior knowledge of the ring structure on $QH^*(M,\omega)$ but, in special cases, we can indeed describe $Y_i$. In the Fano case the higher terms vanish and we may take $Y_i=Z_i$. In the NEF case there might be higher order terms in the Seidel elements $y_i$, however, from \cite[Theorem 1.10]{McDuffTolman06} we know that the lifts $Y_i$ of $y_i$ are determined by some linear combination of the $Z_i$ which is unique up to the additive relations $\Lin(P)$ (see \cite[Example 5.4]{McDuffTolman06} for more details).

\subsubsection{Fano case}

In this case the Landau--Ginzburg superpotential is given by 
$$W= \sum_{j=1}^{n} z^{v_j}t^{\kappa_j}$$
where for $v_j=(v_{j,1}, \hdots, v_{j,m}) \in \Z^m$ the term $z^{v_j}$ is the monomial $z_1^{v_{j,1}}\hdots z_m^{v_{j,m}}$.

We now recall a result obtained by Givental in \cite{Giv2} (which we illustrate with Ostrover--Tyomkin's formalism, see \cite[Proposition 3.3]{OstroverTyomkin09}). 
  \begin{theo}\label{Ostrover}
 If $(M,\omega)$ is a symplectic Fano manifold, then 
 $$ QH^*(M, \omega) \cong \check\Pi[z_1^{\pm}, \hdots, z_m^{\pm}] / J_W \ \mbox{as $\check\Pi$--algebras}$$
 and in particular
 $$ QH^0(M, \omega) \cong \check\Pi^{\mathrm{univ}}[z_1^{\pm}, \hdots, z_m^{\pm}] / J_W \ \mbox{as $\check\Pi^{\mathrm{univ}}$--algebras}$$
 where $J_W$ is the ideal generated by all partial derivatives of $W$. 
 \end{theo}
 
In \cite{OstroverTyomkin09} the authors consider  the natural homomorphism
 $$ \Psi:  \Q[Z_1, \hdots,Z_n] \otimes \check\Pi \to \check\Pi [z_1^{\pm}, \hdots, z_m^{\pm}] $$
 such that $SR_Y(P)$ is in the kernel of $\Psi$ and the image of the additive relations gives the ideal $J_W$. 
 In this case the homomorphism is defined by 
 $$ \Psi(Z_j) = q z^{v_j}t^{\kappa_j} $$ and it is easy to see that $\Psi$ satisfies the desired properties. 
 Indeed,  as we saw above, in the Fano case we may set $Y_i=Z_i$ hence 
 $$ SR_Y(P)= \left \langle Z_{i_1}\hdots Z_{i_k} - Z_{j_1}^{c_1} \hdots Z_{j_\ell}^{c_\ell} \otimes q^{c_1(\beta_I)}t^{\omega(\beta_I)} \ | \ I= \{ i_1, \hdots, i_k\} \mbox { is primitive}\right \rangle$$
 and 
 \begin{equation*}
 \begin{split}
  \Psi(Z_{i_1}& \hdots Z_{i_k} - Z_{j_1}^{c_1} \hdots Z_{j_\ell}^{c_\ell} \otimes q^{c_1(\beta_I)}t^{\omega(\beta_I)} )  \\
& = q^kz^{v_{i_1}}\hdots z^{v_{i_k}}t^{\kappa_{i_1}+\hdots+\kappa_{i_k}} - q^{c_1+ \hdots+c_\ell} z^{c_1v_{j_1}} \hdots z^{c_\ell v_{j_\ell}} t^{c_1\kappa_{j_1}+\hdots +c_\ell \kappa_{j_\ell}} \otimes q^{c_1(\beta_I)}t^{\omega(\beta_I)} = 0
  \end{split}
 \end{equation*}
 by \eqref{ChernNumber} and \eqref{symplecticform}. Therefore the ideal $SR_Y(P)$ is in the kernel of $\Psi$. 
 
 The image of the additive relations is the following
 $$ \Psi \left(\sum_{j=1}^n(x, v_j)Z_j\right)= q \sum_{j=1}^n(x, v_j)z^{v_j}  t^{\kappa_j}.$$
 On the other hand, we have
 $$ qz_i\frac{\partial W}{\partial z_i} = q z_i \sum_{j=1}^n v_{j,i} \, z_1^{v_{j,1}}\hdots z_i^{v_{j,i}-1}\hdots z_m^{v_{j,m}} t^{\kappa_j}=q \sum_{j=1}^n v_{j,i} \, z^{v_j}t^{\kappa_j}.$$
 Note that if $x =e_i$ is the $i$--th vector of the canonical base in $\R^n$ then $(x, v_j)= v_{j,i}$ and one obtains the desired result. 
 
\subsubsection{NEF case}

In this subsection we give the explicit expression of the Landau--Ginzburg superpotential when $M$ is a NEF 4--dimensional toric manifold for which at most 2 of the homology classes $A_i= [\Phi^{-1}(D_i)]$ of the pre-image of the facets $D_i$ have vanishing first Chern number. It follows from the proof of the next proposition that the result generalizes to any number of classes (corresponding to facets of the polytope) with Chern number zero, but the expressions get more complicated as we increase the number of such classes. Moreover, Theorem \ref{Ostrover} still holds for these cases. 
\begin{prop}  \label{Superpotential} If $M$ is a NEF  toric 4--manifold and $A_i= [\Phi^{-1}(D_i)]$ where $D_i$ is a facet of the moment polytope then the Landau--Ginzburg superpotential is given by the following expression: 
\begin{enumerate}
\item if  $c_1$ vanishes only on the class $A_k$ then 
 $$ W = \sum_{j=1}^{n} z^{v_j}t^{\kappa_j} +z^{v_k} t^{\kappa_{k+1}+\kappa_{k-1}-\kappa_k} \,,$$
\item  if $c_1$ vanishes only on the classes $A_{k-1}$ and $A_k$ then 
\begin{align*}
W = \sum_{j=1}^{n} z^{v_j}t^{\kappa_j}& +z^{v_k} t^{\kappa_{k+1}+\kappa_{k-1}-\kappa_k}+z^{v_{k-1}} t^{\kappa_{k}+\kappa_{k-2}-\kappa_{k-1}}\\  
& + z^{v_k} t^{\kappa_{k+1}+\kappa_{k-2}-\kappa_{k-1}} + z^{v_{k-1}} t^{\kappa_{k+1}+\kappa_{k-2}-\kappa_{k}} \,. 
\end{align*}
\end{enumerate}
\end{prop}
\begin{proof}
\emph{Case (1):} in this case the Seidel elements are given by Theorem \ref{theo:SeidelsMorphism}:
\begin{align*}
& S(\Lambda_j)=A_j \otimes q t^{\kappa_j} \ \mbox{ if } \ j \neq k-1, k, k+1, \\ 
& S(\Lambda_{k-1})=A_{k-1} \otimes q t^{\kappa_{k-1}}-  A_k \otimes q \frac{t^{\kappa_{k-1}-\omega(A_k)}}{1- t^{-\omega(A_k)}},\\
& S(\Lambda_k)=A_k \otimes q \frac{t^{\kappa_k}}{1- t^{-\omega(A_k)}},\\
& S(\Lambda_{k+1})=A_{k+1} \otimes q t^{\kappa_{k+1}}-  A_k \otimes q \frac{t^{\kappa_{k+1}-\omega(A_k)}}{1- t^{-\omega(A_k)}}.
\end{align*}
If $S^*$ denotes the Seidel morphism in cohomology then we have 
\begin{align*}
& S^*(\Lambda_j)=Z_j \otimes q^{-1}t^{-\kappa_j} \ \mbox{ if } \ j \neq k-1, k, k+1, \\ 
& S^*(\Lambda_{k-1})=\left( Z_{k-1} -Z_k \otimes \frac{t^{\omega(A_k)}}{1- t^{\omega(A_k)}}\right)\otimes q^{-1} t^{-\kappa_{k-1}},\\
& S^*(\Lambda_k)=Z_k \otimes q^{-1} \frac{t^{-\kappa_k}}{1- t^{\omega(A_k)}},\\
& S^*(\Lambda_{k+1})=\left( Z_{k+1} -Z_k \otimes \frac{t^{\omega(A_k)}}{1- t^{\omega(A_k)}}\right)\otimes q^{-1} t^{-\kappa_{k+1}}.
\end{align*}
Thus in equation \eqref{higherterms} we may take
\begin{align*}
&Y_j =Z_j \ \mbox{ if } \ j \neq k-1, k, k+1,  \quad Y_k =Z_k \otimes \frac{1}{1- t^{\omega(A_k)}},\\
&Y_{k-1} =Z_{k-1} - Z_k \otimes \frac{t^{\omega(A_k)}}{1- t^{\omega(A_k)}}, \quad Y_{k+1} =Z_{k+1} - Z_k \otimes \frac{t^{\omega(A_k)}}{1- t^{\omega(A_k)}}
\end{align*}
where $\omega(A_k)=\kappa_{k+1}+\kappa_{k-1}-2\kappa_k$.
In this case, the definition of the homomorphism $\Psi$ is such that 
\begin{equation}\label{defpsi}
\forall\, {1\leq j \leq n}, \quad \Psi (Y_j)= qz^{v_j}t^{\kappa_j}
\end{equation}
so one obtains 
\begin{align*}
&\Psi(Z_j) = q z^{v_j}t^{\kappa_j} \quad \mbox{ if } \quad j \neq k-1, k, k+1 \,,\\
&\Psi(Z_{k-1}) = q z^{v_{k-1}}t^{\kappa_{k-1}} +q z^{v_k}t^{\kappa_{k+1}+\kappa_{k-1}-\kappa_k} \,,\\
&\Psi(Z_k) = q z^{v_k}t^{\kappa_k} (1-t^{\omega(A_k)})=q z^{v_k}t^{\kappa_k}-q z^{v_k}t^{\kappa_{k+1}+\kappa_{k-1}-\kappa_k} \,, \\
&\Psi(Z_{k+1}) = q z^{v_{k+1}}t^{\kappa_{k+1}} +q z^{v_k}t^{\kappa_{k+1}+\kappa_{k-1}-\kappa_k} \,.
 \end{align*}
 It is clear, by definition of $\Psi$ and the proof in the Fano case that $SR_Y(P)$ is in the kernel of the homomorphism. 
 Computing the image of the additive relations gives
 \begin{equation*}
 \begin{split}
 \Psi \Big(\sum_{j=1}^n(x, v_j)Z_j \Big)  & =  q \sum_{j=1}^n(x, v_j)z^{v_j}  t^{\kappa_j}  - q(x, v_k)z^{v_k}t^{\kappa_{k+1}+\kappa_{k-1}-\kappa_k} \\
 &  +q(x, v_{k-1})z^{v_k}t^{\kappa_{k+1}+\kappa_{k-1}-\kappa_k}+q(x, v_{k+1})z^{v_k}t^{\kappa_{k+1}+\kappa_{k-1}-\kappa_k}.
 \end{split}
 \end{equation*}
 In order to obtain the derivatives of the potential we need 
 $$(x, v_{k-1})+(x, v_{k+1})- (x, v_{k})=(x, v_{k}), \mbox{ that is, } v_{k-1}+v_{k+1}=2 v_k,$$
 which holds, if $\dim M= 4$ and $c_1(A_k)=0$, as noticed already in Section \ref{sec:toric-topol-data}, Equation \eqref{relationvectorsII}.
 
\emph{Case (2):} 
In this case Theorem \ref{theo:SeidelsMorphism} gives the following:
\begin{align*}
& Y_j  =  Z_j \quad \mbox{ if } \quad j \neq k-2,k-1, k, k+1 \,, \\
 & Y_{k-2}  = Z_{k-2} - Z_{k-1} \otimes \frac{t^{\omega(A_{k-1})}}{1- t^{\omega(A_{k-1})}} - Z_{k-1} \otimes \frac{t^{\omega(A_{k-1})+ \omega(A_k)}}{(1- t^{\omega(A_{k-1})})(1- t^{\omega(A_{k-1})+\omega(A_k)})}  \\
  &  \quad \quad \quad \quad + Z_{k} \otimes \frac{t^{\omega(A_{k-1})+ 2\omega(A_k)}}{(1- t^{\omega(A_{k})})(1- t^{\omega(A_{k-1})+\omega(A_k)})} \,,\\
 & Y_{k-1}  = \left( Z_{k-1} \otimes \frac{1}{1- t^{\omega(A_{k-1})}} -Z_k \otimes \frac{t^{\omega(A_k)}}{1- t^{\omega(A_k)}}\right) \frac{1}{1- t^{\omega(A_{k-1})+\omega(A_k)}} \,,\\
 & Y_k  = \left( Z_k \otimes \frac{1}{1- t^{\omega(A_k)}}-Z_{k-1} \otimes \frac{t^{\omega(A_{k-1})}}{1- t^{\omega(A_{k-1})}}\right) \frac{1}{1- t^{\omega(A_{k-1})+\omega(A_k)}} \,,\\
  & Y_{k+1}  = Z_{k+1} - Z_k \otimes \frac{t^{\omega(A_k)}}{1- t^{\omega(A_k)}}-Z_{k} \otimes \frac{t^{\omega(A_{k-1})+ \omega(A_k)}}{(1- t^{\omega(A_{k})})(1- t^{\omega(A_{k-1})+\omega(A_k)})} \\
  &  \quad \quad \quad \quad + Z_{k-1} \otimes \frac{t^{\omega(A_{k})+ 2\omega(A_{k-1})}}{(1- t^{\omega(A_{k-1})})(1- t^{\omega(A_{k-1})+\omega(A_k)})} \,.
\end{align*}
Therefore, as above, if we define $\Psi$ such that it satisfies \eqref{defpsi} then we obtain 
\begin{align*}
& \Psi(Z_j)  = q z^{v_j}t^{\kappa_j} \quad \mbox{ if } \quad j \neq k-2, k-1, k, k+1, \\
& \Psi(Z_{k-2})  = q z^{v_{k-2}}t^{\kappa_{k-2}} +q z^{v_{k}}t^{\kappa_{k+1,k-1}+\kappa_{k-2}} + q z^{v_{k-1}}t^{\kappa_{k-2}}(t^{\kappa_{k,k-1}}+t^{\kappa_{k+1,k}}),\\
& \Psi(Z_{k-1})  = q z^{v_{k-1}}t^{\kappa_{k-1}}(1-t^{\kappa_{k,k-1}+\kappa_{k-2,k-1}})  + q z^{v_{k}}t^{\kappa_{k+1}}(t^{\kappa_{k-1,k}}- t^{\kappa_{k-2,k-1}}),\\
& \Psi(Z_k)  = q z^{v_k}t^{\kappa_k}(1 -t^{\kappa_{k+1,k}+\kappa_{k-1,k}}) + q z^{v_{k-1}}t^{\kappa_{k-2}}(t^{\kappa_{k,k-1}} -t^{\kappa_{k+1,k}}),\\
& \Psi(Z_{k+1})  = q z^{v_{k+1}}t^{\kappa_{k+1}} + q z^{v_{k-1}}t^{\kappa_{k+1,k}+\kappa_{k-2}} + q z^{v_k}t^{\kappa_{k+1}}(t^{\kappa_{k-1,k}}  +t^{\kappa_{k-2,k-1}})
 \end{align*}
 where $\kappa_{i,j}= \kappa_i -  \kappa_j$.
 Again, it is clear that $SR_Y(P)$ is in the kernel of the homomorphism and it is not hard to check that the image of the additive relations gives the derivatives of the superpotential, under the assumptions that $v_{k-1}+v_{k+1}=2 v_k$ and $v_{k-2}+v_{k}=2 v_{k-1}$.
 \end{proof}

\subsection{NEF examples: The case of a blow--up of $\bb C \bb P^2$ at 4 or 5 points}
\label{sec:nef-exples-45pt-bu}

In this section, as an application of our results, we  compute explicitly the small quantum cohomology (and homology) of the manifold obtained from $ \bbcp^2$ by performing 4 and 5 blow-ups,  $\X_4$ and $\X_5$ respectively. Note that these manifolds admit NEF almost complex structures, but no Fano ones. Since the computations are similar, we show the full computations for $\X_4$ and only give the final result for $\X_5$. As already noticed in Example \ref{expl:23-pt-bu}, $\X_4$ is symplectomorphic to the $3$--point blow-up of $S^2 \times S^2$ endowed with the split symplectic form $\omega_\mu$ for which the symplectic area of the first factor is $\mu$ and the area of the second factor is 1 (see \cite[Section 2.1]{AnjosPinsonnault13} for more details). Let $c_i$, $i=1, \hdots,4$ be the capacities of the blow-ups. Let $B$, $F \in H_2(\X_4;\Z)$ be the homology classes defined by  $B=[S^2 \times \{p\}]$, $F=[\{p\} \times S^2]$ and  let $E_i\in H_2(\X_4;\Z)$ be the exceptional class corresponding to the blow-up of capacity $c_i$. Consider $\X_4$ endowed with the standard action of the torus $T=S^1 \times  S^1$ for which the moment polytope is given by 
\begin{equation*}
P =  \left\{ (x_1,x_2) \in \R^2 \mid 0 \leq x_2 \leq \mu, \, x_2+x_1 \leq \mu -c_3, \, -1 \leq x_1 \leq 0,
  \, c_1 \leq x_2 -x_1 \leq \mu+1 -c_2  \right \}
\end{equation*}  so the primitive outward normals to $P$ are as follows:
\begin{equation*} v_1=(0,1), \, v_2=(1,1), \, v_3=(1,0), \,
  v_4=(1,-1), \, v_5=(0,-1), \, v_6=(-1,0), \ \mbox{and} \ v_7=(-1,1).
\end{equation*} 
The normalised moment map $\Phi: \X_4 \to \R^2$ is given by
$$ \Phi(z_1, \hdots, z_7)=\Big( -\frac12 |z_3|^2 + \epsilon_1, -\frac12 |z_1|^2 + \mu - \epsilon_2 \Big),$$
where $$\epsilon_1=\frac{c_1^3+3c_2^2-c_2^3+c_3^3 -3\mu}{3(c_1^2+c_2^2+c_3^2-2\mu)} \quad \mbox{and} \quad  \epsilon_2=\frac{c_1^3 - c_2^3 -c_3^3 +3c_2^2\mu+3c_3^2\mu -3\mu^2}{3(c_1^2+c_2^2+c_3^2-2\mu)} \,.$$ 
     Moreover, the homology classes $A_i= [\Phi^{-1}(D_i)]$ of the pre-images of the corresponding facets $D_i$ are:  
$A_1=F-E_2-E_3 $, $A_2=E_3$, $A_3=B-E_1-E_3$, $A_4=E_1$,  $A_5=F-E_1$, $A_6=B-E_2$, and 
$A_7=E_2$.
Let $\Lambda_i$ be the circle action associated to $v_i$. 
Since the complex structure on $\X_4$ is NEF and $T$--invariant, it follows from Theorem \ref{theo:SeidelsMorphism} that the Seidel elements associated to these actions are given by the following expressions
\begin{align*}
S(\Lambda_1) & =(F-E_2-E_3) \otimes q \frac{t^{\mu-\epsilon_2}}{1- t^{c_2+c_3-1}} \,,\\
S(\Lambda_2) & =E_3 \otimes qt^{\mu -c_3+ \epsilon_1-\epsilon_2} - (F-E_2-E_3) \otimes q \frac{t^{\mu+c_2-1+\epsilon_1-\epsilon_2}}{1- t^{c_2+c_3-1}}  \\ & \qquad \qquad - (B-E_1-E_3) \otimes q \frac{t^{c_1+\epsilon_1-\epsilon_2}}{1- t^{c_1+c_3-\mu}} \,,\\
S(\Lambda_3) & =(B-E_1-E_3) \otimes q \frac{t^{\epsilon_1}}{1- t^{c_1+c_3-\mu}} \,,\\
S(\Lambda_4) & =E_1 \otimes qt^{\epsilon_1+\epsilon_2-c_1} - (B-E_1-E_3) \otimes q \frac{t^{\epsilon_1+\epsilon_2+c_3-\mu}}{1- t^{c_1+c_3-\mu}} \,,\\
S(\Lambda_5) & =(F-E_1)\otimes q {t^{\epsilon_2}},\quad \quad 
S(\Lambda_6)  =(B-E_2)\otimes q {t^{1-\epsilon_1}},\\
S(\Lambda_7) & =E_2 \otimes qt^{\mu+1 -c_2- \epsilon_1-\epsilon_2} - (F-E_2-E_3) \otimes q \frac{t^{\mu+c_3-\epsilon_1-\epsilon_2}}{1- t^{c_2+c_3-1}} \,.
\end{align*}
Therefore we have 
\begin{align*}
S^*(\Lambda_1) & =Z_1 \otimes q^{-1} \frac{t^{\epsilon_2-\mu}}{1- t^{1-c_2-c_3}}, \quad \quad S^*(\Lambda_3)  =Z_3 \otimes q^{-1} \frac{t^{- \epsilon_1}}{1- t^{\mu-c_1-c_3}} \,,\\
S^*(\Lambda_2) & =Z_2 \otimes q^{-1}t^{c_3-\mu-\epsilon_1+\epsilon_2} - Z_1\otimes q^{-1} \frac{t^{1-\mu-c_2-\epsilon_1+\epsilon_2}}{1- t^{1-c_2-c_3}} - Z_3 \otimes q^{-1} \frac{t^{\epsilon_2-\epsilon_1-c_1}}{1- t^{\mu -c_1-c_3}} \,,\\
\end{align*}
\begin{align*}
S^*(\Lambda_4) & =Z_4 \otimes q^{-1}t^{c_1-\epsilon_1-\epsilon_2} - Z_3 \otimes q^{-1} \frac{t^{\mu -c_3 -\epsilon_1-\epsilon_2}}{1- t^{\mu-c_1-c_3}} \,,\\
S^*(\Lambda_5) & =Z_5\otimes q^{-1} {t^{-\epsilon_2}}, \quad \quad
S^*(\Lambda_6)  =Z_6\otimes q^{-1} {t^{\epsilon_1-1}},\\
S^*(\Lambda_7) & =Z_7\otimes q^{-1}t^{c_2-\mu-1 + \epsilon_1+\epsilon_2} - Z_1\otimes q^{-1} \frac{t^{\epsilon_1+\epsilon_2-\mu-c_3}}{1- t^{1-c_2-c_3}} \,.
\end{align*}
Thus in equation \eqref{higherterms} we may take
\begin{align*}
Y_1 & = Z_1 \otimes  \frac{1}{1- t^{1-c_2-c_3}}, \quad Y_2 = Z_2   - Z_1\otimes  \frac{t^{1-c_2-c_3}}{1- t^{1-c_2-c_3}} - Z_3 \otimes \frac{t^{\mu-c_1-c_3}}{1- t^{\mu -c_1-c_3}},\\
Y_3 & = Z_3 \otimes  \frac{1}{1- t^{\mu-c_1-c_3}}, \quad Y_4=Z_4 - Z_3 \otimes  \frac{t^{\mu-c_1-c_3}}{1- t^{\mu-c_1-c_3}}, \quad Y_5 = Z_5,\\
Y_6 & = Z_6, \quad  \mbox{and} \quad Y_7 = Z_7   - Z_1\otimes  \frac{t^{1-c_2-c_3}}{1- t^{1-c_2-c_3}}.
\end{align*}
There are fourteen primitive sets: $$\{1,3\},\{1,4\},\{1,5\},\{1,6\},\{2,4\},\{2,5\},\{2,6\}, \{2,7\},\{3,5\},\{3,6\}, \{3,7\},\{4,6\},\{4,7\},\{5,7\}.$$
 Let $t_{2,3}=1- t^{1-c_2-c_3}$ and $t_{1,3}=1- t^{\mu -c_1-c_3}$. The corresponding multiplicative relations for $QH^*(\X_4,\omega_\mu)$, that is, the generators of the ideal $SR_Y(P)$ defined in Proposition  \ref{QuantumRingIso}, can be written as follows
\begin{equation}\label{MultiplicativeRelations}
\begin{split}
 Z_1Z_3  & = Z_2 \otimes q t^{c_3} t_{2,3}t_{1,3} - Z_1 \otimes q t^{1-c_2}t_{1,3} - Z_3 \otimes qt^{\mu -c_1}t_{2,3},\\
 Z_1Z_4  t_{1,3} & = Z_1Z_3 \otimes t^{\mu -c_1-c_3} + Z_3 \otimes q  t^{\mu-c_1}t_{2,3},\\
 Z_1Z_5   &= \mathbbm{1}\otimes q^2 t^{\mu}t_{2,3}, \\
 Z_1Z_6 & = Z_7 \otimes qt^{c_2 }t_{2,3} -Z_1 \otimes q t^{1-c_3},\\
 Z_2Z_4t_{2,3}t_{1,3} & = Z_3(Z_2+Z_3+Z_4)\otimes t^{\mu-c_1-c_3}t_{2,3}  +Z_1Z_4 \otimes t^{1-c_2-c_3}t_{1,3}\\
  & \quad \quad \quad - Z_1Z_3 \otimes t^{1+\mu -c_1-c_2-2c_3},\\
 Z_2Z_5t_{2,3}t_{1,3} & = Z_1Z_5 \otimes t^{1-c_2-c_3}t_{1,3} +Z_3Z_5\otimes t^{\mu-c_1-c_3}t_{2,3} + Z_3 \otimes q t^{\mu -c_3}t_{2,3},\\
 Z_2Z_6t_{2,3}t_{1,3} & = Z_1Z_6 \otimes t^{1-c_2-c_3}t_{1,3} +Z_3Z_6\otimes t^{\mu-c_1-c_3}t_{2,3} + Z_1 \otimes q t^{1 -c_3}t_{1,3},\\
 Z_2Z_7t_{2,3}t_{1,3} & = Z_1(Z_1 +Z_2+ Z_7) \otimes t^{1-c_2-c_3}t_{1,3}   +Z_3Z_7\otimes t^{\mu-c_1-c_3}t_{2,3} \\
  & \quad \quad \quad - Z_1Z_3 \otimes t^{1+\mu -c_1-c_2-2c_3},\\
   Z_3Z_5 & = Z_4 \otimes qt^{c_1 }t_{1,3} -Z_3 \otimes q t^{\mu-c_3},\\
 Z_3Z_6  &= \mathbbm{1}\otimes q^2 tt_{1,3},\\
 Z_3Z_7  t_{2,3} & = Z_1Z_3 \otimes t^{1-c_2-c_3} + Z_1 \otimes q  t^{1-c_2}t_{1,3},\\
 Z_4Z_6t_{1,3} & = Z_5 \otimes qt^{1-c_1 }t_{1,3} +Z_3Z_6 \otimes  t^{\mu-c_1- c_3},\\
 Z_4Z_7t_{2,3}t_{1,3} & = Z_1Z_4 \otimes t^{1-c_2-c_3}t_{1,3}  +Z_3Z_7\otimes t^{\mu-c_1-c_3}t_{2,3} \\
 & \hspace{2,45cm} - Z_3 Z_1\otimes qt^{1+\mu -c_1-c_2-2c_3}+ \mathbbm{1}\otimes q^2 t^{\mu+1-c_1-c_2}t_{2,3}t_{1,3},\\
Z_5Z_7  t_{2,3}& = Z_1Z_5 \otimes t^{1-c_2-c_3} + Z_6 \otimes q  t^{\mu-c_2}t_{2,3}
\end{split}
\end{equation}
where we should also take in account the additive relations $Z_6= Z_1 +2 Z_2 + Z_3- Z_5$ and $Z_7= -Z_1-Z_2 +Z_4 +Z_5$. 
It follows from Proposition \ref{QuantumRingIso} that $QH^*(\X_4,\omega_\mu)$  is isomorphic as a ring to $\Q[Z_1, \hdots,Z_n] \otimes \check\Pi / I $ where $I$ is the ideal generated by the relations above.  We can describe the result also in terms of homology. For that consider the homology classes $A_i= [\Phi^{-1}(D_i)] \in H_2(\X_4;\Z)$. They are additive generators of $H_2(\X_4;\Z)$ and multiplicative generators of $QH_*(\X_4,\omega_\mu)$. Moreover $QH_4(\X_4,\omega_\mu)$ is generated, as a subring of $QH_*(\X_4,\omega_\mu)$, by the elements $qA_i$. 
 These generators are $E_i \otimes q$, where $i=1, 2, 3$, $ (F-E_1) \otimes q$, $(B-E_2)\otimes q$, $(F-E_2-E_3) \otimes q$, and $(B-E_1-E_3) \otimes q$. In what follows in order to simplify notation we shall drop the sign $*$ for the quantum product. The multiplicative relations \eqref{MultiplicativeRelations} translated to homology together with the additive relations give a complete description of the $\Pi^{\mathrm{univ}}$--algebra $QH_4(\X_4,\omega_\mu)$.  More precisely, we obtain
$$QH_4(\X_4,\omega_\mu)\cong \Pi^{\mathrm{univ}}[u,v]/J$$ where 
$u=(F-E_2-E_3) \otimes q(1-t^{c_2+c_3-1})^{-1}$, $v= (B-E_1-E_3) \otimes q(1-t^{c_1+c_3-\mu})^{-1}$, and $J$ is the ideal generated by the two following relations: 
\begin{gather}\label{QuantumRelations}
u^2t^\mu (v+t^{c_2-1})(1+vt^{c_3})=v(1+vt^{c_1}), \quad \mbox{and} \quad v^2t(u+t^{c_1-\mu})(1+ut^{c_3})=u(1+ut^{c_2}).
\end{gather}
It follows from Proposition \ref{Superpotential} (1) that the Landau--Ginzburg superpotential is given in this example by 
\begin{equation}\label{eq:LG-superpot-X4}
\begin{split}
W & = z_2t^{\mu} +z_1z_2t^{\mu-c_3} + z_1+z_1z_2^{-1}t^{-c_1} +z_2^{-1}+z_1^{-1}t +z_1^{-1}z_2t^{\mu+1-c_2} \\
 & \quad +z_1t^{\mu-c_1-c_3}+z_2t^{\mu+1-c_2-c_3}.
\end{split}
\end{equation}
Therefore we have 
\begin{align*}
\frac{\partial W}{\partial z_1} & =  z_2t^{\mu-c_3} +1 + z_2^{-1}t^{-c_1} -z_1^{-2} t -z_1^{-2}z_2t^{\mu+1-c_2}+t^{\mu-c_1-c_3}, \\
\frac{\partial W}{\partial z_2} & = t^{\mu} + z_1t^{\mu-c_3} -z_1z_2^{-2}t^{-c_1} -z_2^{-2}+z_1^{-1}t^{\mu+1-c_2}+t^{\mu+1-c_2-c_3}.
\end{align*}
Passing to homology, simplifying the expressions and setting $u=z_2^{-1}t^{-\mu}$ and $v=z_1^{-1}$ we obtain relations \eqref{QuantumRelations}, as we wish. 

Similar arguments give an explicit description of the quantum homology algebra $QH_4(\X_5,\omega_\mu)$. Moreover, we have
$$QH_4(\X_5,\omega_\mu)\cong \Pi^{\mathrm{univ}}[u,v]/J$$ where again 
$u=(F-E_2-E_3) \otimes q(1-t^{c_2+c_3-1})^{-1}$, $v= (B-E_1-E_3) \otimes q(1-t^{c_1+c_3-\mu})^{-1}$ and $J$ is now the ideal generated by the two following relations: 
\begin{gather*}u^2t^\mu (v+t^{c_2-1})(1+vt^{c_3})=(1+vt^{c_1})(v+t^{c_4-1}), 
\\ v^2t(u+t^{c_1-\mu})(1+ut^{c_3})=(1+ut^{c_2})(u+t^{c_4-\mu}).
\end{gather*}
In this case the Landau--Ginzburg superpotential is given by 
\begin{equation*}
\begin{split}
W & = z_2t^{\mu} +z_1z_2t^{\mu-c_3} + z_1+z_1z_2^{-1}t^{-c_1} +z_2^{-1}+ z_1^{-1}z_2^{-1}t^{1-c_4}+z_1^{-1}t \\ 
& \ +z_1^{-1}z_2t^{\mu+1-c_2}  +z_1t^{\mu-c_1-c_3}+z_2t^{\mu+1-c_2-c_3}+z_1^{-1}t^{\mu+1-c_2-c_4}+ z_2^{-1}t^{1-c_1-c_4}.
\end{split}
\end{equation*}
\begin{rema}\label{rema:CLpotential}
Note that these results agree with the results of Chan and Lau. The manifolds $\X_4$ and $\X_5$ coincide with the surfaces $X_7$  and $X_{10}$, respectively, described in \cite[Appendix A]{CL}. We obtain the same expressions for the potential after changes of variable: replacing $z_2$ by $z_1z_2^{-1}t^{-c_1}$, keeping the variable $z_1$  and letting $q_1= t^{\mu-c_1-c_3}$, $q_2=t^{\mu-c_2}$, $q_3=t^{c_2}$, $q_4=t^{1-c_2-c_3}$ and $q_5=t^{c_3}$ in  the potential for $X_7$ leads to \eqref{eq:LG-superpot-X4} above. Similarly, making the same change of variable for $X_{10}$ and letting $q_1= t^{\mu-c_1-c_3}$, $q_2=t^{c_4}$, $q_3= t^{\mu-c_2-c_4}$, $q_4=t^{c_2}$, $q_5=t^{1-c_2-c_3}$ and $q_6=t^{c_3}$ we see again that the two expressions for the potential agree. \end{rema}

\subsection{Non-NEF examples}
\label{sec:non-nef-examples}

Particularly interesting examples which are relevant for our study are the Hirzebruch surfaces. We use the conventions and the description adopted in \cite{AnjosPinsonnault13} for these surfaces. We recall that the toric ``even'' Hirzebruch surfaces $(\F_{2k},\omega_{\mu})$, $0\leq k\leq \ell$ with $\ell \in \N$ and $\ell < \mu \le \ell+1$, can be identified with the symplectic manifolds $(S^{2}\times S^{2},\omega_\mu)$ where $\omega_\mu$ is  the split symplectic form with area $\mu \geq 1$ for the first $S^2$--factor, and with area 1 for the second factor. The moment polytope of $\F_{2k}$ is
$$\left\{ (x_1,x_2) \in \R^2 \mid 0 \leq x_1 \leq 1, \ x_2+kx_1 \geq 0, \ x_2-kx_1 \leq \mu-k\right\}$$ 
and its primitive outward normals are 
$$v_1=(1,0),\,  v_2=(-k,-1),\, v_3=(-1,0), \ \mbox{and} \ v_4=(-k,1).$$
Let $\Lambda^{2k}_{e_1}$ and $\Lambda^{2k}_{e_2}$ represent the circle actions whose moment maps are, respectively, the first and the second component of the moment map associated to the torus action $T_{2k}$ acting on $\F_{2k}$. We will also denote by $\Lambda^{2k}_{e_1},\Lambda^{2k}_{e_2}$ the generators in $\pi_1(T_{2k})$. It follows from the classification of 4--dimensional Hamiltonian $S^1$--spaces given by Karshon in \cite{Ka99} that $\Lambda^{2k}_{e_1},\Lambda^{2k}_{e_2}$ satisfy the relations $\Lambda^{2k}_{e_1}= k\Lambda^{2}_{e_1}+ (k-1)\Lambda^{0}_{e_1}$ and $\Lambda^{2k}_{e_2}=k\Lambda^{0}_{e_1}+\Lambda^{0}_{e_2}$. Since $\F_0$ is Fano and $\F_2$ is NEF we can obtain from our results the Seidel elements associated to $\Lambda^{0}_{e_1}$, $\Lambda^{0}_{e_2}$, and $\Lambda^{2}_{e_1}$, and thus the ones associated to the circle actions  of $\F_{2k}$ even though for all $k \geq 2$, $\F_{2k}$ is non-NEF.

In particular, we can give explicit expressions for the Seidel elements associated to $\F_4$ which admits a pseudo-holomorphic sphere with negative first Chern number, representing the class $B-2F$ where $B=[S^2 \times \{p\}]$, and $F=[\{p\} \times S^2]$. Since $\F_0$ is Fano it is easy to check that the Seidel elements associated to the circle actions $\Lambda^{0}_{e_1}$ and $\Lambda^{0}_{e_2}$ are given by $S(\Lambda^{0}_{e_1})=B \otimes qt^{\frac12}$ and $ S(\Lambda^{0}_{e_2})=F \otimes qt^{\frac\mu2}$ (see \cite[Example 5.7]{McDuffTolman06}). From this case we can also obtain the following products in the quantum homology ring: $F * F =  \mathbbm 1 \otimes q^{-2}t^{-\mu}$, $B * B = \mathbbm 1\otimes q^{-2}t^{-1}$, $F*B=p$ and deduce the remaining products from these ones.  

For the toric manifold $\F_2$ 
 the normalised moment map is given by 
$$ \Phi(z_1,z_2,z_3,z_4)= \left(-\frac12|z_1|^2+\frac12 -\epsilon, -\frac12|z_1|^2-\frac12|z_4|^2+ \frac{\mu+1}{2}\right),$$
where $\epsilon= \frac{1}{6\mu}.$
Let $\Lambda_{v_i}^{2k}$ denote the circle action associated to the normal vector $v_i$ to the polytope of the surface $\F_{2k}$. Then Theorem \ref{theo:SeidelsMorphism} implies that, in the case of  $\F_2$,  the Seidel elements associated to these actions are given by 
\begin{equation*}
\begin{split}
S(\Lambda_{v_1}^2) & =(B+F) \otimes qt^{\frac12 -\epsilon}, \ S(\Lambda_{v_3}^2)=(B-F) \otimes q\frac{t^{\frac12 +\epsilon}}{1-t^{1-\mu}} \quad  \mbox{and} \\
S(\Lambda_{v_2}^2) & =S(\Lambda_{v_4}^2)= F \otimes qt^{\frac\mu2 +\epsilon}-(B-F) \otimes q\frac{t^{1-\frac\mu2 +\epsilon}}{1-t^{1-\mu}} \,.
\end{split}
\end{equation*}

Since $\Lambda^{2}_{e_1}=\Lambda_{v_1}^2$, $S(\Lambda^{2}_{e_1})=S(\Lambda_{v_1}^2)$ and it follows that for the non-NEF toric manifold $\F_4$ the Seidel elements associated to the circle actions $\Lambda^{4}_{e_1}$ and $\Lambda^{4}_{e_2}$ are given by 
\begin{equation*}
\begin{split}
S(\Lambda^{4}_{e_1}) & =S(\Lambda^{2}_{e_1})^2*S(\Lambda^{0}_{e_1})= (B+2F) \otimes qt^{\frac12 -2\epsilon}+ B\otimes qt^{\frac32 -\mu -2\epsilon},\\
S(\Lambda^{4}_{e_2}) & =S(\Lambda^{0}_{e_1})^2*S(\Lambda^{0}_{e_2})= S(\Lambda^{0}_{e_2})= F \otimes qt^{\frac\mu2}, 
\end{split}
\end{equation*}
because $S(\Lambda^{0}_{e_1})^2= \mathbbm 1$. 
 Therefore in this case, since  $\Lambda_{e_1}^4=\Lambda^4_{v_1}$, it follows that 
 $$ S(\Lambda_{v_1}^4)= qt^{\frac12 -2\epsilon} \otimes (B+2F + B\otimes qt^{1-\mu}).$$
  Since $v_1+v_3=0$
it follows that $S(\Lambda_{v_3}^4)= S(\Lambda^{4}_{e_1})^{-1}= \left(S(\Lambda^{2}_{e_1})^{-1}\right)^2 *S(\Lambda^{0}_{e_1})^{-1}$ and since
$$S(\Lambda^{2}_{e_1})^{-1}=(B-F) \otimes q\frac{t^{\frac12 +\epsilon}}{1-t^{1-\mu}}$$
we obtain 
$$S(\Lambda_{v_3}^4)= \frac{qt^{\frac12 +2\epsilon}}{(1-t^{1-\mu})^2}\otimes \left[ B-2F  +B\otimes qt^{1-\mu}\right].$$
Finally, since $v_4=2v_3 +(0,1)$,  $v_2=2v_3 +(0,-1)$, and $S(\Lambda^{4}_{e_2})=S(\Lambda^{4}_{e_2})^{-1}$ it follows that 
$ S(\Lambda_{v_2}^4)= S(\Lambda_{v_4}^4)=S(\Lambda_{v_3}^4)^2*S(\Lambda^{4}_{e_2})$, hence
\begin{equation*}
S(\Lambda_{v_2}^4)= S(\Lambda_{v_4}^4)  = \frac{qt^{\frac\mu2 +4\epsilon}}{(1-t^{1-\mu})^4} \left[ F\otimes(1-t^{1-\mu})^2 - 4t^{1-\mu}(B-2F  +B\otimes qt^{1-\mu})\right] \,.
\end{equation*}
It follows that in equation \eqref{higherterms}  we may take
\begin{align*}
Y_1 & = Z_1 + (Z_3+Z_2+Z_4) \otimes t^{\mu-1}, \ Y_3 = \frac{1}{(1- t^{\mu-1})^2} (Z_3 + (Z_3+Z_2+Z_4) \otimes t^{\mu-1}),\\
Y_2 & = \frac{1}{(1- t^{\mu-1})^2}(Z_2 - 4t^{\mu-1}Y_3), \quad Y_4= \frac{1}{(1- t^{\mu-1})^2}(Z_4 - 4t^{\mu-1}Y_3).
\end{align*}
Since the ring structure on the quantum homology is known we can check that this choice of $Y_i$ satisfies the equations induced by the primitive relations, that is,
$$ Y_1Y_3-\mathbbm 1 \otimes q^2 t \quad \mbox{and} \quad Y_2Y_4- (Y_3)^4 \otimes q^{-2}t^{\mu-2}$$
are generators of  the ideal $SR_Y(P)$.
In order to have a potential $W$ such that the isomorphism in Theorem \ref{Ostrover} holds we need that the homomorphism $\Psi$, inducing the isomorphism,  satisfies equations \eqref{defpsi}. Recall that the generators of the ideal $SR_Y(P)$ should be in the kernel of $\Psi$ and the image of the additive relations gives the derivatives of the potential. 
\begin{align}\label{defpsiII}
  \begin{split}
\Psi(Y_1) =qz_1t  &\;\Leftrightarrow\; \Psi (Z_1)+ \Psi (Z_2+Z_3 +Z_4)t^{\mu-1}= qz_1t \\
\Psi(Y_2) =qz_1^{-2}y^{-1} &\;\Leftrightarrow\; \Psi (Z_2) -4 t^{1-\mu}\Psi(Y_3)=qz_1^{-2}y^{-1}(1- t^{\mu-1})^2 \\
\Psi(Y_3) =qz_1^{-1}  &\;\Leftrightarrow\; \Psi (Z_3)+ \Psi (Z_2+Z_3 +Z_4)t^{\mu-1}= qz_1^{-1} (1- t^{\mu-1})^2 \\
\Psi(Y_4) =qz_1^{-2}yt^{\mu-2} &\;\Leftrightarrow\; \Psi (Z_4) -4 t^{1-\mu}\Psi(Y_3)=qz_1^{-2}yt^{\mu-2}(1- t^{\mu-1})^2
  \end{split}
\end{align}
Since the additive relations are $Z_1-Z_3 -2Z_2-2Z_4=0$ and $Z_4-Z_2=0$ it follows from equations \eqref{defpsiII} that the derivatives of the potential $W$ are given by the following expressions:
\begin{equation*}
\begin{split}
q z_1 \frac{\partial W}{\partial z_1} & = \Psi(Z_1)-\Psi(Z_3) -2\Psi(Z_2)-2\Psi(Z_4) \\
& = qz_1t-qz_1^{-1} (1- t^{\mu-1})^2-16 qz_1^{-1}t^{\mu-1} -2 (qz_1^{-2}z_2^{-1}+qz_1^{-2}z_2t^{\mu-2})(1- t^{\mu-1})^2,\\
 q z_2 \frac{\partial W}{\partial z_2} & = \Psi (Z_4) -\Psi(Z_2) = (qz_1^{-2}z_2t^{\mu-2}-qz_1^{-2}z_2^{-1})(1- t^{\mu-1})^2 .
\end{split}
\end{equation*}
Therefore the potential is given by 
\begin{equation}\label{potentialF4}
W= z_1t +(z_1^{-1}+z_1^{-2}z_2^{-1}+z_1^{-2}z_2t^{\mu-2})(1- t^{\mu-1})^2 +16z_1^{-1}t^{\mu-1}.
\end{equation}
\begin{rema}
In this non-NEF example we see that  the number of terms  corresponding to the quantum corrections in the Landau--Ginzburg superpotential is still finite. In the formalism of \cite{CL} and \cite{CLLT} the primitive rays of the fan (or the interior normal vectors of the polytope) are given by 
$v_1=(1,0),\,  v_2=(0,1),\, v_3=(-1,-4)$, and $v_4=(0,-1)$
and the polytope is defined by the following inequalities 
$$x_1 \geq 0, \;  x_2 \geq 0, \; 4t_1+t_2 -x_1-4x_2 \geq 0 \; \mbox{and} \; t_1 -x_2 \geq 0,$$
where the $t_l$'s are positive numbers. Let $q_l= \exp (-t_l)$ be the K\"ahler parameters. Then, in their formalism, the potential is  given by 
$$W= z_1(1-2q_1q_2+q_1^2q_2^2) +z_2 + \frac{q_1^4q_2}{z_1z_2^4}(1-2q_1q_2+q_1^2q_2^2) + \frac{q_1}{z_2}(1+14q_1q_2+q_1^2q_2^2).$$
In  this expression $z_1$ and $z_2$ correspond to  $z_1^{-2}z_2^{-1}$ and $z_1t$, respectively, in equation \eqref{potentialF4} while $q_1 = t$ and $q_2=t^{\mu-2}$. 
Moreover, if \cite[Conjecture 6.7]{CLLT} holds then we can obtain the  open Gromov--Witten invariants  of $\F_4$ from our computation of the potential. In particular we see that there must be some negative open Gromov--Witten invariants, phenomenon which does never happen in the NEF case. 
\end{rema} 
We conclude that, even in this non-NEF example, although there are infinitely many contributions to the Seidel elements associated to the Hamiltonian circle actions, these quantum classes can still be expressed by explicit closed formulas. It is clear that as we increase the value of $k$ the expressions for the Seidel elements corresponding to the circle actions $\Lambda^{2k}_{e_1},\Lambda^{2k}_{e_2}$ in $\F_{2k}$ are going to be harder to write explicitly. However, from the work of Abreu and McDuff in \cite{AMcD} we know that the generators of the fundamental group of the symplectomorphism group of $(S^2\times S^2,\omega_\mu)$ are given by $\Lambda^{0}_{e_1},\Lambda^{0}_{e_2}$ and $\Lambda^{2}_{e_1}$, so our computations allow us to give a complete description of the Seidel representation for these manifolds (regardless of the value of $\mu$ provided that $\mu \geq 1$). 

\begin{rema}
The ``Odd'' Hirzebruch surfaces $(\F_{2k-1},\omega_{\mu}')$, $1\leq k\leq \ell$ with $\ell \in \N$ and  $\ell < \mu \le \ell+1$, can be identified with the symplectic manifolds $(\PbP,\omega_\mu')$ where the symplectic area of the exceptional divisor is  $\mu >0$ and the area of the projective line is $\mu +1$. Its moment polytope is 
$$\left\{ (x_1,x_2) \in \R^2 \mid 0 \leq x_1+x_2 \leq 1, \ x_2(k-1)+kx_1 \geq 0, \ kx_2+(k-1)x_1 \geq k -\mu-1\right\}.$$
Similar computations can be made for $\F_{2k-1}$, since $\F_1$ is Fano and we can show that 
$\Lambda^{2k-1}_{e_1}=\Lambda^{2k-1}_{e_2}= (k-1)\Lambda^{1}_{e_1}+ k\Lambda^{1}_{e_2}$, using Karshon's classification of Hamiltonian circle actions. 
\end{rema}

% ***********************************************************************************************
%
%                                         APPENDICES
%
% ***********************************************************************************************

\appendix
\section{Additional computations of Seidel's elements}
\label{sec:BrutalComputationsSeidSMorph}

We gather here results of computations of Seidel's elements in the case when the number of facets, in the vicinity of $D_n$, corresponding to spheres in $M$ with vanishing first chern number is $3$ (this is complementary to Theorem \ref{theo:SeidelsMorphism}, see Figure \ref{fig:cases}). In order to ease the reading, we denote the weights $\omega(A_i)$ by $\omega_i$.

(2c) If $c_1(A_n)=c_1(A_1)=c_1(A_2)=0$ but $c_1(A_{n-1})$ and $c_1(A_{3})$ are non-zero, then
  \begin{align*}
S(\Lambda) = &\left[ \left( A_n\otimes q\,\frac{t^{\Phi_\max}}{1-t^{-\omega_n}}-A_1\otimes q\,\frac{t^{\Phi_\max-\omega_1}}{1-t^{-\omega_1}} \right)\cdot \frac{1}{1-t^{-\omega_n-\omega_1}}\right.\\
& \qquad \left.-\left(A_1\otimes q\,\frac{t^{\Phi_\max}}{1-t^{-\omega_1}}-A_2\otimes q\,\frac{t^{\Phi_\max-\omega_2}}{1-t^{-\omega_2}}\right) \cdot\frac{t^{-\omega_1-\omega_2}}{1-t^{-\omega_1-\omega_2}}\right]  \cdot\frac{1}{1-t^{-\omega_n-\omega_1-\omega_2}}    
  \end{align*}

(2d) If $c_1(A_n)=c_1(A_{n-1})=c_1(A_1)=0$ but $c_1(A_2)$ and $c_1(A_{n-2})$ are non-zero, then
  \begin{align*}
S(\Lambda) = &\left[ \left( A_n\otimes q\,\frac{t^{\Phi_\max}}{1-t^{-\omega_n}}-A_{n-1}\otimes q\,\frac{t^{\Phi_\max-\omega_{n-1}}}{1-t^{-\omega_{n-1}}} \right)\cdot \frac{1}{1-t^{-\omega_n-\omega_{n-1}}}\right.\\
&+\left. \left( A_n\otimes q\,\frac{t^{\Phi_\max}}{1-t^{-\omega_n}}-A_1\otimes q\,\frac{t^{\Phi_\max-\omega_1}}{1-t^{-\omega_1}} \right)\cdot \frac{1}{1-t^{-\omega_n-\omega_1}} \right. \\
&\left.-A_n\otimes q\,\frac{t^{\Phi_\max}}{1-t^{-\omega_n}}\right]\cdot \frac{1}{1-t^{-\omega_n-\omega_{n-1}-\omega_1}}
  \end{align*}

(3d) If $c_1(A_1)=c_1(A_2)=c_1(A_{3})=0$ but $c_1(A_n)$, $c_1(A_{4})$ and $c_1(A_{n-1})$ are non-zero, then 
  \begin{align*}
S(\Lambda) = A_n\otimes & qt^{\Phi_\max}-A_1\otimes q\,\frac{t^{\Phi_\max-\omega_1}}{1-t^{-\omega_1}}\\
& -\left(A_1\otimes q\,\frac{t^{\Phi_\max}}{1-t^{-\omega_1}}-A_2\otimes q\,\frac{t^{\Phi_\max-\omega_2}}{1-t^{-\omega_2}}\right)\cdot\frac{t^{-\omega_1-\omega_2}}{1-t^{-\omega_1-\omega_2}}\\
& -\left(A_1\otimes q\,\frac{t^{\Phi_\max}}{1-t^{-\omega_1}}-A_2\otimes q\,\frac{t^{\Phi_\max-\omega_2}}{1-t^{-\omega_2}}\right)\cdot 
 \frac{t^{-\omega_1-\omega_2-\omega_{3}}}{(1-t^{-\omega_1-\omega_2-\omega_{3}})(1-t^{-\omega_1-\omega_2})}\\
& + \left(A_2\otimes q\,\frac{t^{\Phi_\max}}{1-t^{-\omega_2}}-A_{3}\otimes q\,\frac{t^{\Phi_\max-\omega_{3}}}{1-t^{-\omega_{3}}}\right)\cdot 
\frac{t^{-\omega_1-2\omega_2-2\omega_{3}}}{(1-t^{-\omega_2-\omega_{3}})(1-t^{-\omega_1-\omega_2-\omega_{3}})}    
  \end{align*}

(3e) If $c_1(A_{n-1})=c_1(A_1)=c_1(A_{2})=0$ but $c_1(A_n)$, $c_1(A_{3})$ and $c_1(A_{n-2})$ are non-zero, then 
\begin{align*}
  S(\Lambda)  =  A_n\otimes qt^{\Phi_\max} &-A_{n-1}\otimes q\,\frac{t^{\Phi_\max-\omega_{n-1}}}{1-t^{-\omega_{n-1}}} -A_1\otimes q\,\frac{t^{\Phi_\max-\omega_1}}{1-t^{-\omega_1}}  \\
 &-\left(A_1\otimes q\,\frac{t^{\Phi_\max}}{1-t^{-\omega_1}}-A_2\otimes q\,\frac{t^{\Phi_\max-\omega_2}}{1-t^{-\omega_2}}\right)\cdot 
\frac{t^{-\omega_1-\omega_2}}{1-t^{-\omega_1-\omega_2}}
\end{align*}


\begin{thebibliography}{DDDD}
  
\bibitem{AMcD}
M. Abreu and D. McDuff, \emph{Topology of symplectomorphism groups of rational ruled surfaces}, {J. Amer. Math. Soc.} \textbf{13} (2000), {971--1009 (electronic)}.

\bibitem{AnjosPinsonnault13}
  S. Anjos and M. Pinsonnault, \emph{The homotopy Lie algebra of symplectomorphism groups of some blow-ups of the projective plane}, Math. Z. {\bf 275} (2013), 245--292.
  
\bibitem{BarraudCornea}
  J.-F. Barraud and O. Cornea, \emph{Higher order Seidel invariants for loops of hamiltonian isotopies}, in preparation.

\bibitem{Batyrev93}
  V. Batyrev, \emph{Quantum cohomology rings of toric manifolds}, Ast\'erisque {\bf 218} (1993), 9--34.

\bibitem{Cannas03}
A. Cannas da Silva, {Symplectic Toric Manifolds}. In \emph{Symplectic Geometry of Integrable Hamiltonian Systems}, M. Castellet ed., Advanced Courses in Mathematics, CRM Barcelona (2003), 85--173.

\bibitem{CL}
K. Chan, and S.-C. Lau, \emph{Open Gromov-Witten invariants and superpotentials for semi-Fano toric surfaces}, Int. Math. Res. Not.  (2013), doi: 10.1093/imrn/rnt050.

\bibitem{CLLT}
K. Chan, S.-C. Lau, N C Leung, and H.-H. Tseng, \emph{Open Gromov--Witten invariants, mirror maps, and Seidel's representations for toric manifolds}, arXiv:1209.6119 (2012).

\bibitem{CharetteCornea}
  F. Charette and O. Cornea, \emph{Categorification of Seidel's representation}, Israel J. Math. (to appear), arXiv:1307.7235 (2013).

\bibitem{Cox-Katz99}
D. Cox and S. Katz, \emph{Mirror Symmetry and Algebraic Geometry}, Math. Surveys and Monographs vol 68, AMS, Providence (1999).
  
\bibitem{Delzant88}
T. Delzant, \emph{Hamiltoniens p\'eriodiques et image convexe de l'application moment}, Bulletin de la Soci\'et\'e Math\'ematique de France {\bf 116} (1988), 315--339.   

\bibitem{EntovPolterovich08}
  M. Entov and L. Polterovich, \emph{Symplectic quasi-states and semi-simplicity of quantum homology}. In \emph{Toric topology}, vol. 460 of Contemp. Math., Amer. Math. Soc., Providence, RI (2008), 47--70.

\bibitem{FOOO10a}
   K. Fukaya, Y.-G. Oh, H. Ohta, and K. Ono, \emph{Lagrangian Floer theory on compact toric manifolds. I}, 
   Duke Math. J. \textbf{151} (2010), 23--174.

\bibitem{FOOO10}
   K. Fukaya, Y.-G. Oh, H. Ohta, and K. Ono, \emph{Lagrangian Floer theory and mirror symmetry on compact toric manifolds}, arXiv:1009.1648 (2010).

\bibitem{FOOO11}
   K. Fukaya, Y.-G. Oh, H. Ohta, and K. Ono, \emph{Spectral invariants with bulks, quasi-morphisms, and Lagrangian Floer theory},  arXiv:1105.5123 (2011).
   
\bibitem{Fulton} 
W. Fulton, \emph{Introduction to Toric Varieties}, Annals of Mathematics Studies, Princeton University Press (1993).

\bibitem{Giv1}
A. B. Givental, \emph{Equivariant {G}romov-{W}itten invariants}, Internat. Math. Res. Notices \textbf{13} (1996), {613--663}.

\bibitem{Giv2}
A. B. Givental, \emph{A mirror theorem for toric complete intersections}, in \emph{Topological field theory, primitive forms and related topics ({K}yoto, 1996)}, Progr. Math. \textbf{160} (1998), 141--175, {Birkh\"auser Boston, Boston, MA}.

\bibitem{GonzalesIritani11}
  E. Gonz\'alez and H. Iritani, \emph{Seidel elements and mirror transformations}, Selecta Math. (N.S.) {\bf 18} (2012), no. 3, 557--590.

\bibitem{GP} T. Graber and R. Pandharipande, \emph{Localization of virtual classes}, Invent. Math. {\bf 135} (1999), 487--518. 

\bibitem{HuLalonde}
S. Hu and F. Lalonde, \emph{A relative Seidel morphism and the Albers map}, Trans. amer. Math. Soc. \textbf{362} (2010), 1135--1168.

\bibitem{HuLalondeLeclercq}
S. Hu, F. Lalonde, and R. Leclercq, \emph{Homological Lagrangian monodromy}, Geom. Topol. \textbf{15} (2011), 1617--1650.

\bibitem{Hutchings}
M. Hutchings, \emph{Floer homology for families. I}, Algebr. Geom. Topol. \textbf{8} (2008), 435--492. 

\bibitem{Hyvrier}
 C. Hyvrier, \emph{Lagrangian circle actions}, arXiv:1307.8196 (2013).
 
 \bibitem{Iritani}
 H. Iritani, \emph{Convergence of quantum cohomology by quantum Lefschetz}, J. Reine. Angew. Math. {\bf 610} (2007), 29--69. 

\bibitem{Ka99} Y. Karshon, \emph{Periodic Hamiltonian Flows on Four Dimensional Manifolds},  Mem. Amer. Math. Soc. 141, no. 672 (1999). 

\bibitem{KarshonKesslerPinsonnault}
  Y. Karshon, L. Kessler, and M. Pinsonnault, \emph{A compact symplectic four-manifold admits only finitely many inequivalent toric actions}, {J. Symplectic Geom.} \textbf{5} (2007), {139--166}.

\bibitem{LMcDP99}
    F. Lalonde, D. McDuff, and L. Polterovich, \emph{Topological rigidity of {H}amiltonian loops and quantum homology}, {Invent. Math.} \textbf{135} ({1999}), {369--385}.
  
\bibitem{Melissa13} C.--C. M. Liu, \emph{Localization in Gromov--Witten theory and orbifold Gromov--Witten theory}. In \emph{Handbook of Moduli}, Volume II, Adv. Lect. Math., (ALM) {\bf 25}, International Press and Higher Education Press (2013), 353--425.

\bibitem{M} Y. Manin, {Generating functions in algebraic geometry and sums over trees}. In \emph{The moduli space of curves},
R. Dijkgraaf, C. Faber, and G. van der Geer, eds., Birkhauser (1995), 401--417.
  
\bibitem{McDuffSalamon04}
  D. McDuff and D. Salamon, \emph{$J$--holomorphic Curves and Symplectic Topology}, Amer. Mat. Soc., Providence, RI (2004).

\bibitem{McDuffTolman06}
  D. McDuff and S. Tolman, \emph{Topological properties of Hamiltonian circle actions}, Int. Math. Res. Papers (2006), doi:10.1155/IMRP/2006/72826.

\bibitem{OstroverTyomkin09}
  Y. Ostrover and I. Tyomkin, \emph{On the quantum homology algebra of toric Fano manifolds}, Selecta Math. (N.S.) \textbf{15} (2009), 121--149.

\bibitem{Pinsonnault}
  M. Pinsonnault, \emph{Symplectomorphism groups and embeddings of balls into rational ruled surfaces}, Compos. Math. \textbf{144}  (2008), 787--810.
  
\bibitem{Savelyev}
Y. Savelyev, \emph{Quantum characteristic classes and the {H}ofer metric}, Geom. Topol. \textbf{12} (2008), 2277--2326.

\bibitem{Seidel97}
  P. Seidel, \emph{$\pi_1$ of symplectic automorphism groups and invertibles in quantum homology rings}, Geom. Funct. Anal. \textbf{7} (1997), {1046--1095}.

\bibitem{Spielberg99}
H. Spielberg, \emph{Une formule pour les invariants de Gromov--Witten des vari\'et\'es toriques}, Ph.D. thesis (1999).

\bibitem{Spielberg00}
H. Spielberg, \emph{The Gromov--Witten invariants of symplectic toric manifolds}, arXiv: math/0006156v1 (2000).

\bibitem{Usher10}
  M. Usher, \emph{ Deformed Hamiltonian Floer theory, capacity estimates, and Calabi quasimorphisms}, Geom. Topol. \textbf{15} (2011), 1313--1417.

\end{thebibliography}
\end{document}